\documentclass[10pt]{article}
\usepackage{amssymb}
\usepackage[dvipsnames]{xcolor}

\usepackage{amsmath, amssymb, amsthm, mathtools}
\usepackage{tikz}
\usetikzlibrary{arrows.meta}
\usetikzlibrary{positioning}
\usepackage{caption, subcaption, float}
\usepackage[makeroom]{cancel}
\usepackage{xcolor}
\usepackage{enumitem}
\DeclareSymbolFont{symbolsC}{U}{txsyc}{m}{n}
\DeclareMathSymbol{\strictif}{\mathrel}{symbolsC}{74}

\newtheorem{theorem}{Theorem}[section]

\newtheorem{lemma}[theorem]{Lemma}

\newtheorem*{lemma*}{Lemma}

\theoremstyle{definition}
\newtheorem{definition}[theorem]{Definition}

\newtheorem{proposition}[theorem]{Proposition}

\newtheorem{corollary}[theorem]{Corollary}

\theoremstyle{remark}
\newtheorem{remark}[theorem]{Remark}

\theoremstyle{remark}

\newcommand\restr[2]{{
\left.\kern-\nulldelimiterspace 
#1 
\vphantom{\big|} 
\right|_{#2} 
}}

\newcommand{\A}{\mathfrak{A}}

\newcommand{\B}{\mathfrak{B}}
\newcommand{\C}{\mathfrak{C}}

\newcommand{\F}{\mathcal{F}}

\newcommand{\II}{\mathbb{I}}

\newcommand{\subal}{\mathbb{S}}
\newcommand{\prodal}{\mathbb{P}}

\newcommand{\homal}{\mathbb{H}}

\newcommand{\NN}{\mathbb{N}}

\let\phi\varphi
\title{Stable Canonical Rules for Intuitionistic Modal Logic}
\author{Cheng Liao}
\date{}
\begin{document}

\maketitle

\begin{abstract}
This paper develops stable canonical rules for intuitionistic modal logics, \textcolor{black}{which} were first introduced for superintuitionistic logics and transitive normal modal logics in \cite{ncsl} and \cite{nrule} respectively. We first prove that every intuitionistic modal multi-conclusion consequence relation is axiomatizable by stable canonical rules. \textcolor{black}{This allows us to assume, without loss of
generality, that rules considered by us are stable canonical ones.} The idea turns out to be useful. In particular, using stable canonical rules, we get an alternative proof of the Blok-Esakia theorem for intuitionistic modal logics which was first proved in \cite{zfon} and generalize it to \textcolor{black}{multi-conclusion} consequence relations. We also prove the Dummett-\textcolor{black}{Lemmon} conjecture for intuitionistic modal multi-conclusion consequence relations, which, as far as we know, is a new result.
\end{abstract}

\section{Introduction}
\label{sec:sample1}


In the study of non-classical logics, semantic methods in general and uniform axiomatization \textcolor{black}{techniques} in particular, play a central role. One such important method was developed by Zakharyaschev in 1980s. He introduced \textit{canonical formulas} for superintuitionistic logics\footnote{These are extensions of the intuitionistic propositional calculus.} and proved that every superintuitionistic logic can be axiomatized by canonical formulas \cite{Zphd}.
The brief idea is as follows: for every formula, using a variant of selective \textcolor{black}{filtration}, one can obtain a finite set of finite refutation patterns (i.e., a finite intuitionistic frame with a set of parameters). \textcolor{black}{Canonical formulas syntactically encode these finite
refutation patterns in such a way that the conjunction of them is
equivalent to the original formula.} Zakharyaschev \cite{zc1,zc2} then also developed canonical formulas for transitive normal modal logics in a series of papers and proved that every normal extension of \textbf{K4} can be axiomatized by modal canonical formulas. Following the same idea, Je$\check{\text{r}}$$\Acute{\text{a}}$bek \cite{can} generalized the result to multi-conclusion rules and showed that every normal modal multi-conclusion consequence relation over \textbf{K4} is axiomatizable by \textit{canonical rules}. It turns out that canonical formulas and rules offer a uniform method to study superintuitionistic and modal logics, and are thus quite useful. For example, using canonical formulas, Zakharyaschev proved the Dummett-Lemmon conjecture stating that a superintuitionistic logic is Kripke complete iff its least modal companion is \cite{zcss}, and with canonical rules, Je$\check{\text{r}}$$\Acute{\text{a}}$bek gave an alternative proof of decidability of admissibility in the intuitionistic propositional calculus \cite{can}. 

However, the mechanisms of developing canonical formulas and rules are model-theoretic and \textcolor{black}{are} quite involved. Bezhanishvili \textit{et al.} \cite{nin,nmo,nrule} developed an algebraic approach to canonical formulas and rules via duality. \textcolor{black}{They showed} that from the algebraic perspective, Zakharyaschev's canonical formulas for superintuitionistic logics encode the $\lor$-free reducts of Heyting algebras \textcolor{black}{which are locally finite}, and the so-called \textit{closed domain condition} encoded in the formulas \textcolor{black}{corresponds to} the preservation of joins of certain elements \cite{nin}.

Besides, the algebraic perspective raises a natural question: can we develop canonical formulas and rules based on the $\rightarrow$-free reducts of Heyting algebras which are also locally finite? This idea \textcolor{black}{led} to the study of stable canonical formulas and rules, \textcolor{black}{by G.Bezhanishvili, N. Bezhanishvili, Iemhoff and their collaborators first} \cite{nloc,gn,nrule,ncsl}, which \textcolor{black}{encode these $\rightarrow$-free reducts, and} are alternatives to Zakharyaschev's canonical formulas and Je$\check{\text{r}}$$\Acute{\text{a}}$bek's canonical rules. Since stable canonical formulas and rules encode finite refutation patterns constructed by taking filtrations instead of selective filtrations, they can apply to non-transitive logics where Zakharyaschev's canonical formulas and Je$\check{\text{r}}$$\Acute{\text{a}}$bek's canonical rules do not apply.\footnote{More details can be found in \cite{jphd}.} In particular, it was proved in \cite{nrule} that  every normal modal multi-conclusion consequence relation is axiomatizable by stable canonical rules, which partially answered the question about developing canonical formulas for extensions of \textbf{K} proposed in \cite[Problem 9.5]{ChagrovZakharyashev}.

Although the research on stable canonical formulas and rules is still in its infancy, quite \textcolor{black}{some} effort has already been put into the extension of stable canonical formulas and rules to different settings and development of related theories. For example, in her PhD thesis \cite{jphd}, Ilin gave a thorough analysis of \textit{stable logics} which \textcolor{black}{enjoy \textcolor{black}{the finite model property} and can be axiomatized by a certain type of stable formulas}. Melzer \cite{mmaster} developed the stable canonical formulas for the lax logic while Cleani \cite{amaster, an} generalized stable canonical formulas and rules to the setting of bi-superintuitionistic logics, and also \textcolor{black}{used} them to prove the Blok-Esakia theorem and the Dummett-Lemmon conjecture \textcolor{black}{in that setting}.

On the other hand, compared to \textcolor{black}{superintuitionistic} logics and classical modal logics, the area of  intuitionistic modal logics (i.e. superintuitionistic logics plus modal operators) is \textcolor{black}{much more involved and less well-understood}. Intuitionistic modal logics were introduced by Fisher Servi \cite{fish}, whose basic aim was to define modalities from an intuitionistic point of view. These logics are applicable to a wide variety of situations, ranging from computer science \cite{cs} to epistemic logics \cite{ap}. However, the study of intuitionistic modal logics is far from an easy combination of the study of classical modal logics and of superintuitionistic logics. In fact, there still remain some fundamental philosophical and technical questions. \textcolor{black}{For example, unlike in classical modal logics, $\Box$ and $\Diamond$ are not dual to each other in intuitionistic modal logics. \textcolor{black}{A natural question is: what counts as a reasonable relation between these two operators? We still do not} know much about the finite model property of intuitionistic modal logics\footnote{See \cite{sphd} for an example, and one more recent result can be found in \cite{FD}.}} when they have both $\Box$ and $\Diamond$. Thus, \textcolor{black}{this area may still benefit from some new uniform and effective methods.}

Considering the above situation, in this paper we \textcolor{black}{will} develop stable canonical rules for intuitionistic modal logics. These on the one hand 
 will deepen our \textcolor{black}{understanding} of the theory of stable canonical formulas and rules and verify its wide applicability, and on the other hand may provide us with a uniform method to study intuitionistic modal logics and thus \textcolor{black}{pave} the way for further study. 

In particular, we prove that every intuitionistic modal multi-conclusion consequence relation is axiomatizable by stable canonical rules. Following the main proof strategy of \cite{amaster, an}, we give an alternative proof of the Blok-Esakia theorem for intuitionistic modal logics \cite{zfint} and generalize it to multi-conclusion consequence relations, \textcolor{black}{establishing} that the lattice of intuitionistic modal multi-conclusion consequence relations is isomorphic to the lattice of extensions of the bimodal multi-conclusion consequence relation $(\textbf{Grz}\otimes \textbf{K})\oplus \textbf{Mix}^R$. By adjusting the proof strategy of \cite{amaster, an}, we give a proof of the Dummett-Lemmon conjecture for intuitionistic modal multi-conclusion consequence relations \textcolor{black}{stating that an intuitionistic modal multi-conclusion consequence relation is Kripke complete if and only if its least modal companion is,} which, as far as we know, is a new result.\footnote{\textcolor{black}{The original proof in \cite{amaster} contains a gap which is fixed in \cite{an}. Our proof is different from that of \cite{an} in some important aspect, and can also fill the gap when restricted to the setting of superintuitionistic logics.}} 

The structure of this paper is as follows. Section 2 is devoted to general preliminaries which are needed \textcolor{black}{throughout} the paper. In Section 3, we develop stable canonical rules for intuitionistic modal logics. Using duality theory, a dual description of stable canonical rules for intuitionistic modal logics is also given. Section 4 is about the application of stable canonical rules. We first introduce the G\"odel translation for intuitionistic modal logic, and then use stable canonical rules for bimodal logics to prove the Blok-Esakia theorem. Then with the Blok-Esakia theorem and stable canonical rules for intuitionistic modal logics, we prove the Dummett-Lemmon conjecture for intuitionistic modal multi-conclusion consequence relations. \textcolor{black}{In Section 5, we conclude and discuss} possible directions for future work.

\section{Preliminaries}

In this section, we present \textcolor{black}{basic} notations, definitions and facts that will be used throughout the paper. We assume familiarity with standard set-theore\-tic notations, elementary lattice theory and some fundamental concepts from topology and first-order logic. Familiarity with the categorical notion of duality is also expected but \textcolor{black}{is} not necessary.

\subsection{Ordered Sets}

We begin with the following notations and definitions related to ordered sets.

\begin{definition}[Maximal and passive elements]

Let $X$ be a set, $R$ be a transitive binary relation on $X$ and $U\subseteq X$. \textcolor{black}{We say that $x\in U$ is an \textit{$R$-maximal} element of $U$ if for any $y\in U$, $Rxy$ implies that $x=y$.} We define $\max_R(U)$ as the set of all $R$-maximal elements of $U$.

An element $x\in U$ is called $R$-\textit{passive} in $U$ if for all $y\in X\setminus U (\text{or } \Bar{U})$, if $Rxy$, then there is no $z\in U$ such that $Ryz$\footnote{\textcolor{black}{Intuitively speaking, it says that there is no $R$-path starting from $x$ that first goes out of $U$ and then goes back to $U$.}}. The set of all $R$-passive elements of $U$ is denoted as $pas_R(U)$.

\end{definition}

\begin{definition}

For any reflexive and transitive (binary) relation $R$ on a set $X$, a subset $C\subseteq X$ is an \textit{$R$-cluster} if it is an equivalence class under the relation $\backsim_R $ where $x\backsim_R y$ iff $xRy$ and $yRx$. 

An $R$-cluster is \textit{proper} if it contains more than one element. 

For any $U\subseteq X$, $U$ is said to \textit{cut an $R$-cluster $C$} if $U\cap C\not=\emptyset$ and $C\setminus U\not=\emptyset$.

\end{definition}

\begin{remark}

As usual, for any equivalence relation $R$ on a set $X$, we use $[x]$ (or $[x]_R$) to denote the equivalence class of $x$ where $x\in X$. 

By definition, $U$ does not cut an $R$-cluster $C$ if it either contains $C$ or is disjoint from $C$. 
    
\end{remark}

\begin{definition}[Upsets and downsets]
    Let $(A, \leq)$ be a poset (partially ordered set) and let $B \subseteq A$. We call $B$ \textit{upwards closed} or an \textit{upset} if $x \in B$, $y \in A$ and $x \leq y$ imply $y \in B$. If $C \subseteq A$, we write ${\uparrow}C$ for the least upset that contains $C$, namely $\{y \in A \mid \exists x \in C: x\leq y \}$. And we use $Up(A)$ to denote the set of all upsets of $(A, \leq)$.
    
    Similarly, we call $B$ \textit{downwards closed} or a \textit{downset} if $x \in B$, $y \in A$ and $y \leq x$ imply $y \in B$. If $C \subseteq A$, we define ${\downarrow}C=\{ y \in A \mid \exists x \in C: y\leq x \}$.
\end{definition}

\subsection{Universal Algebra}

We then recall a few facts about universal algebra. \textcolor{black}{Everything in this subsection can be found in \cite{BurrisSankappanavar}.}

\begin{definition}[Algebra]
    Let $A$ be a non-empty set, $\tau: \F \to \NN$ be a signature 
    \textcolor{black}{where $\F$ is a set of function symbols}
     and $F = \{ f^A \mid f \in \F, f^A: A^{\tau(f)} \to A \}$, we call $(A, F)$ an \textit{algebra} in the signature $\tau$ (or simply $\tau$-\textit{algebra}). If the context is clear, we will \textcolor{black}{denote by} $A$ both the algebra and the underlying set (also called the \textit{carrier}).
\end{definition}

\begin{remark}
    For convenience, we will use the same notations for function symbols and their corresponding interpretations.
\end{remark}

\vspace{1mm}
Let $(A,F)$ be a $\tau$-algebra. \textcolor{black}{Then} $(A,F')$ is called the \textit{reduct} of $(A,F)$ if $F'\subseteq F$. Because of the above remark, we may also write a reduct of $(A,F)$ as $(A,F)|_{\tau'}$ or simply $A|_{\tau'}$ where $\tau'\subseteq \tau$.

Let $\mathcal{K}$ be a class of $\tau$-algebras, we then introduce the following class operators:
    \begin{align*}
    \II(\mathcal{K}) &\coloneqq \{ A \mid A \text{ is isomorhpic to some } B \in \mathcal{K} \};  \\
            \textcolor{black}{\homal(\mathcal{K})} &\textcolor{black}{\coloneqq \{ A \mid A \text{ is a homomorphic image of some } B \in \mathcal{K} \};} \\
        \subal(\mathcal{K}) &\coloneqq \II(\{ A \mid A \text{ is a subalgebra of some } B \in \mathcal{K} \}); \\
                \textcolor{black}{\prodal(\mathcal{K})} &\textcolor{black}{\coloneqq \II(\{ A \mid A \text{ is a product of some } \{B_i\}_{i \in I} \subseteq \mathcal{K} \});}\\
        \prodal_u(\mathcal{K}) &\coloneqq \II(\{ A \mid A \text{ is an ultraproduct of some } \{B_i\}_{i \in I} \subseteq \mathcal{K} \}).
    \end{align*}

Let $\tau$ be an arbitrary signature. A class of $\tau$-algebras is a \textit{universal class} if it is the class of models of some set of universal sentences\footnote{``Universal" in the sense of first-order logic. See \cite[Def. 5.6]{ml}.}.

\begin{definition}
If $\mathcal{K}$ is a class of $\tau$-\textcolor{black}{algebras}, we write $Uni(\mathcal{K})$ for the least universal class that contains $\mathcal{K}$. $Uni(\mathcal{K})$ is also called the universal class $\textit{generated by}$ $\mathcal{K}$.
\end{definition}

The following is a useful characterisation of $Uni$ in terms of the operators we have just defined.

\begin{theorem}\cite[Thm. 2.20]{BurrisSankappanavar}
\label{thm:univClassSPu}
    Universal classes are closed under $\subal$ and $\prodal_u$. Furthermore:
    \begin{equation*}
        Uni(\mathcal{K}) = \subal \prodal_u(\mathcal{K}).
    \end{equation*}
\end{theorem}

\textcolor{black}{We now introduce another important class of algebras called \textit{variety}.} 

\begin{definition}[Variety]
Let $\tau$ be an arbitrary signature,  a class of $\tau$-algebras $V$ is a \textit{variety} if it is closed under $\homal$, $\subal$ and $\prodal$. 

Let $\mathcal{K}$ be a class of $\tau$-algebras, we write $Var(\mathcal{K})$ for the least variety which contains $\mathcal{K}$. $Var(\mathcal{K})$ is also called the variety \textit{generated by} $\mathcal{K}$.
\end{definition}

Similarly to universal classes, varieties are defined by special formulas called \textit{equations}. 

\begin{definition}[Equation]
    A first-order sentence $\phi$ is called an \textit{equation} if it is of the form $\textcolor{black}{\forall v_1...\forall v_n(\sigma(v_1,...,v_n) =\tau(v_1,...,v_n))}$, where $\sigma(v_1,...,v_n)$ and $\tau(v_1,...,v_n)$ are terms \textcolor{black}{whose free variables are among $\{v_1,...,v_n\}$}.
\end{definition}

This following result is called Birkhoff's Theorem.

\begin{theorem}\cite[Thm. 11.9]{BurrisSankappanavar}
\label{2.2.18}
Let $\tau$ be an arbitrary signature and Let $V$ be a class of  $\tau$-algebras. Then $V$ is a variety if and only if $V$ is definable by equations, namely there is a set of equations $\Phi$ such that:
    \begin{equation*}
        V = \{ A \text{ is an algebra in } \tau \mid A \vDash \Phi \}.
    \end{equation*}
\end{theorem}

\subsection{Deductive Systems}

Now we introduce two types of deductive systems which will be explored throughout the paper. The following presentation is mainly based on \cite{Iem}.

\begin{definition}
 The \textit{set of formulas in signature $v$ over a set of variables $X$} (denoted by $Form_v(X)$) is the least set containing $X$ such that for any $f\in v$, we have that $\varphi_1,...,\varphi_n\in Form_v(X)$ \textcolor{black}{implies that} $f(\varphi_1,...,\varphi_n)\in Form_v(X)$ where $f$ is of arity $n$.   
\end{definition}

Let $Prop$ be a fixed countably infinite set of variables, we write $Form_v$ for $Form_v(Prop)$. A \textit{substitution} $s$ is \textcolor{black}{a} map from $Prop$ to $Form_v$, which can be recursively extended to a map $\Bar{s}$ from $Form_v$ to $Form_v$ in the obvious way.

First, we define what a \textit{logic} is.

\begin{definition}
A \textit{logic} over $Form_v$ is a set $L\subseteq Form_v$ such that if $\varphi\in L$, then $\Bar{s}(\varphi)\in L$ for every substitution $s$.    
\end{definition}

\begin{remark}
By the above definition, classical propositional logic is a logic over $Form_{\land,\lor,\neg}$.     
\end{remark}

\begin{definition}[Multi-conclusion rule]
A \textit{multi-conclusion rule} in signature $v$ over a set of variables $X$ is a pair $\Gamma/\Delta$ of finite subsets of $Form_v(X)$.    
\end{definition}

\begin{remark}
 In case $\Delta=\{\psi\}$, we simply write $\Gamma/\psi$ for $\Gamma/\Delta$, similarly if $\Gamma=\{\varphi\}$.
\end{remark}
 
We write $Rul_v(X)$ for the \textit{set of all multi-conclusion rules in signature $v$ over the set of variables $X$}, and we let $Rul_v$ stand for $Rul_v(Prop)$.   

As logics are defined over formulas, \textit{multi-conclusion consequence relations} are defined over multi-conclusion rules.

\begin{definition}[Multi-conclusion consequence relation]
    A \textit{multi-conclusion consequence relation} over $Rul_v$ is a set $S\subseteq Rul_v$ such that the following hold (\textcolor{black}{``;" means set-theoretic union}):

\begin{itemize}
    \item If $\Gamma/\Delta\in S$, then $\Bar{s}[\Gamma]/\Bar{s}[\Delta]\in S$ for all substitutions $s$.
    \item $\varphi/\varphi\in S$ \textcolor{black}{for all formulas $\varphi$}.
    \item If $\Gamma/\Delta\in S$, then $\Gamma;\Gamma'/\Delta;\Delta'\in S$ for any finite sets of formulas $\Gamma'$ and $\Delta'$.
    \item  $\Gamma/\Delta;\varphi\in S$ and $\Gamma;\varphi/\Delta\in S$, then $\Gamma/\Delta\in S$ (Cut). 
\end{itemize}

\end{definition}

If $L$ is a logic and $\Delta$ is a set of formulas, we write $L\oplus \Delta$ for the least logic extending $L$ that contains $\Delta$, and say that the logic is axiomatized over $L$ by $\Delta$. Similarly we define $S\oplus\Sigma$ where $S$ is a multi-conclusion consequence relation and $\Sigma$ is a set of multi-conclusion rules.

We then define the interpretation of formulas and rules over algebras in the same signature. Let $\mathfrak{A}$ be a $v$-algebra and $A$ be its carrier, a \textit{valuation} on $\mathfrak{A}$ is a map $V$ from $Prop$ to $A$, which can be recursively extended to a map $\Bar{V}$ from $Form_v$ to $A$ in the most obvious way.

In the following, every algebra we consider is assumed to have the top element (or the largest element, denoted by $1$). This will make the following definition of \textit{validity} simpler, and will not cause any problem as we only consider such algebras in this paper. For a more general definition of validity in algebraic semantics, one can consult \cite{Tool}.

\begin{definition}[Validity]

A rule $\Gamma/\Delta$ in signature $v$ is \textit{valid} on a $v$-algebra $\mathfrak{A}$ if for any valuation $V$ on $\mathfrak{A}$, if $\Bar{V}(\gamma)=1$ for any $\gamma\in \Gamma$, then $\Bar{V}(\delta)=1$ for some $\delta\in \Delta$ where $1$ is the top element of $\mathfrak{A}$. We denote this as $\mathfrak{A}\vDash \Gamma/\Delta$.

A formula $\varphi$ in signature $v$ is \textit{valid} on a $v$-algebra $\mathfrak{A}$ if the rule $/\varphi$ is valid on $\mathfrak{A}$, and we denote this by $\mathfrak{A}\vDash \varphi$.
    
\end{definition}

\begin{remark}
With the above definition, the notion of validity of a rule $\Gamma/\Delta$ on a class $\mathcal{K}$ of $v$-algebras and the notion of validity of a set of rules $S$ on a $v$-algebra $\mathfrak{A}$ are defined as usual, and we denote them as $\mathcal{K}\vDash \Gamma/\Delta$ and $\mathfrak{A}\vDash S$ respectively.

Then for any logic $L$, we say that $L$ is \textit{complete} w.r.t a class of algebras $\mathcal{K}$ if $\mathcal{K}\vDash \varphi$ implies that $\varphi\in L$. Similarly we define the completeness of a multi-conclusion consequence relation.
\end{remark}

Finally, we fix two useful notations. Let $\mathcal{A}_v$ be the class of all $v$-algebras and $S$ be a logic or a multi-conclusion consequence relation, we write $Alg(S)$ for the set of all $v$-algebras which validate $S$, i.e., $Alg(S)=\{\mathfrak{A}\in \mathcal{A}_v\mid \mathfrak{A}\vDash S\}$. Conversely, if $\mathcal{K}$ is a set of $v$-algebras, we define $Ru(\mathcal{K})=\{\Gamma/\Delta\in Rul_v\mid \mathcal{K}\vDash\Gamma/\Delta\}$ and $Th(\mathcal{K})=\{\varphi\in Form_v\mid \mathcal{K}\vDash\varphi\}$.

\subsection{Bimodal Logics and Intuitionistic Modal Logics}

After presenting the general theories of universal algebra and deductive systems, we now apply them in more concrete settings.

The \textit{bimodal signature} $bi=\{\land,\lor,\neg,\top,\bot,\Box_I,\Box_M\}$, and the set of \textit{bimodal formulas}  $Form_{bi}$ is then defined recursively as follows:

$$\varphi::=p\mid \top\mid \bot\mid \varphi\land \varphi\mid \varphi\lor\varphi\mid \neg\varphi\mid \Box_I\varphi\mid \Box_M\varphi$$

As usual, $\varphi\rightarrow\psi$ stands for $\neg \varphi\lor \psi$ for any bimodal formulas $\varphi$ and $\psi$.

\begin{remark}

\textcolor{black}{The subscript $I$} means ``intuitionistic" while \textcolor{black}{the subscript $M$} means ``modal". The reason for using these subscripts will only become clear later in this paper. Right now, we use them simply as a way to distinguish these two operators.
    
\end{remark}

\begin{definition}

A logic $L$ over $Form_{bi}$ is a \textit{bimodal logic} if the following hold:

\begin{itemize}
\item \textbf{CPC}\footnote{\textcolor{black}{It stands for classical propositional calculus.}} $\subseteq L$
\item $\Box_I(\varphi\land\psi)\leftrightarrow (\Box_I\varphi\land \Box_I\psi)\in L$ and $\Box_M(\varphi\land\psi)\leftrightarrow (\Box_M\varphi\land \Box_M\psi)\in L$
\item $\varphi\rightarrow \psi\in L$ implies $\Box_I\varphi \rightarrow \Box_I\psi \in L$ and $\Box_M\varphi \rightarrow \Box_M\psi \in L$ (Reg)

\item $\varphi\in L$ implies $\Box_I\varphi\in L$ and $\Box_M\varphi\in L$ (Nec)
\item $\varphi\rightarrow\psi \text{ and }\varphi\in L$ implies $\psi\in L$ (MP)

\end{itemize}
\end{definition}

We denote the least bimodal logic \textcolor{black}{by} $\textbf{K}\otimes\textbf{K}$. This notation is justified as if restricted to the signatures $\{\land,\lor,\neg,\top,\bot,\Box_I\}$ or $\{\land,\lor,\neg,\top,\bot,\Box_M\}$\footnote{In either case, we would prefer to use simply $\Box$ instead of $\Box_I$ or $\Box_M$.}, the least bimodal logic is just the least normal modal logic $\textbf{K}$. Then $\textbf{S4}\otimes\textbf{K}$ is  $(\textbf{K}\otimes\textbf{K})\oplus (\Box_I p\rightarrow p)\oplus (\Box_I p\rightarrow \Box_I\Box_I p)$ and $\textbf{Grz}\otimes \textbf{K}$ is $(\textbf{S4}\otimes\textbf{K})\oplus \Box_I(\Box_I (p\rightarrow \Box_I p)\rightarrow p)\rightarrow p$\footnote{When restricted to the signature $\{\land,\lor,\neg,\top,\bot,\Box_I\}$, $\textbf{S4}\otimes\textbf{K}$ is just $\textbf{S4}$, and $\textbf{Grz}\otimes\textbf{K}$ is just $\textbf{Grz}$. See \cite{ChagrovZakharyashev} for more information about these normal modal logics.}. 

Then we introduce \textit{bimodal multi-conclusion consequence relations} as follows:

\begin{definition}

A \textit{bimodal multi-conclusion consequence relation} is a multi-conclusion consequence relation $M$ over $Rul_{bi}$ satisfying the following conditions:

\begin{itemize}
    \item $/\varphi\in M$ whenever $\varphi\in \textbf{K}\otimes\textbf{K}$ 
    \item $\varphi/\Box_I\varphi\in M$ and $\varphi/\Box_M \varphi\in M$  
    \item $\varphi\rightarrow\psi, \varphi/\psi\in M$ 
    
\end{itemize}

\end{definition}

\begin{remark}
    For convenience, sometimes bimodal logics just mean bimodal deductive systems when whether they are logics or multi-conclusion consequence relations does not matter. This also applies to intuitionistic modal logics as well 
\end{remark}

\textcolor{black}{Elements in $Rul_{bi}$ are called \textit{bimodal multi-conclusion rules}.} If $L$ is a bimodal logic, then $\mathbf{NExt}(L)$ is the lattice of all bimodal logics extending $L$ with $\oplus$ as join and intersection as meet. Similarly we define $\mathbf{NExt}(M)$ where $M$ is a bimodal multi-conclusion consequence relation. Clearly, for any $L\in \mathbf{NExt}(\textbf{K}\otimes\textbf{K})$, there is a least bimodal multi-conclusion consequence relation $L^R$ containing all $/\varphi$ for $\varphi\in L$. In particular, we denote the one corresponding to $\textbf{K}\otimes\textbf{K}$ as $\textbf{K}\otimes\textbf{K}^R$ (the least bimodal multi-conclusion consequence relation) and the one corresponding to $\textbf{S4}\otimes\textbf{K}$ as $\textbf{S4}\otimes\textbf{K}^R$. Conversely, for any $M\in \mathbf{NExt}(\textbf{K}\otimes\textbf{K})$, $Taut(M)=\{\varphi\in Form_{bi}\mid /\varphi\in M\}$ is a bimodal logic.

The following proposition allows us to transfer results about multi-conclusion consequence relations to results about logics. The proof is routine.

\begin{proposition}
\label{2.4.4}
The mappings $(-)^R$ and $Taut(-)$ are mutually inverse complete lattice isomorphisms between  $\mathbf{NExt}(\textbf{K}\otimes\textbf{K})$ and the sublattice of $\mathbf{NExt}(\textbf{K}\otimes\textbf{K}^R)$ consisting of all bimodal multi-conclusion consequence relations $M$ such that $Taut(M)^R=M$.
\end{proposition}

Then we recall algebraic semantics for bimodal logics. 

\begin{definition}

A \textit{modal algebra} is a tuple $\mathfrak{A}=(A,\Box)$ where $A$ is a Boolean algebra, $\Box 1=1$ and $\Box (a\land b)=\Box a\land \Box b$ for any $a,b\in A$. 

A \textit{$K\otimes K$-algebra} (or \textit{bimodal algebra}) is a tuple $\A=(A,\Box_I,\Box_M)$ where $(A,\Box_I)$ and $(A,\Box_M)$ 
are both modal algebras.

For any bimodal logic $L$, we call a bimodal algebra an \textit{$L$-algebra} if it validates $L$.
\end{definition}

\begin{remark}
In particular,  a bimodal algebra $\A=(A,\Box_I,\Box_M)$ is an \textit{$S4\otimes K$-algebra} if $\Box_I a\leq a$ and $\Box_I a\leq \Box_I\Box_I a$ for any $a\in A$ (or equivalently, $\Box_I a\rightarrow a =1$ and $\Box_I a\rightarrow \Box_I \Box_I a=1$). A bimodal algebra $\A=(A,\Box_I,\Box_M)$ is a \textit{$Grz\otimes K$-algebra} if $\Box_I(\Box_I (a\rightarrow \Box_I a)\rightarrow a)\leq a$ for any $a\in A$\footnote{It is well known that every $Grz$-algebra is an $S4$-algebra as $\textbf{Grz}$ is an extension of $\textbf{S4}$. See \cite{ChagrovZakharyashev}. }.    
\end{remark}

Let $\mathbf{BMA}$ be the class of all bimodal algebras, by Theorem \ref{2.2.18}, $\mathbf{BMA}$ is a variety. Let $\textbf{Var}(\mathbf{BMA})$ and $\textbf{Uni}(\mathbf{BMA})$ denote the lattice of subvarieties and the
lattice of universal subclasses of $\mathbf{BMA}$ respectively,  we have the following result as usual. It says that there is a correspondence between varieties and logics, and a correspondence between universal classes and multi-conclusion consequence relations. Proofs of similar results for modal algebras and (unary) normal modal logics can be found in \cite[Thm. 2.5]{mul} and \cite[Thm. 7.56]{ChagrovZakharyashev}.

\begin{theorem}
\label{2.4.7}
The following maps \textcolor{black}{form} pairs of mutually inverse isomorphisms:

\begin{itemize}
    \item Alg: $\mathbf{NExt}(\textbf{K}\otimes\textbf{K})\rightarrow \textbf{Var}(\mathbf{BMA})$ and Th: $\textbf{Var}(\mathbf{BMA})\rightarrow \mathbf{NExt}(\textbf{K}\otimes\textbf{K})$
\item Alg: $\mathbf{NExt}(\textbf{K}\otimes\textbf{K}^R)\rightarrow \textbf{Uni}(\mathbf{BMA})$ and Ru: $\textbf{Uni}(\mathbf{BMA})\rightarrow \mathbf{NExt}(\textbf{K}\otimes\textbf{K}^R)$

\end{itemize}

\end{theorem}

\begin{corollary}
\label{combi}
 Every  bimodal logic (resp.  bimodal multi-conclusion consequence relation) is complete with respect to some variety (resp. universal class) of bimodal algebras.   
\end{corollary}

Similarly, we can define two deductive systems for intuitionistic modal logics and give the corresponding algebraic semantics.

The \textit{intuitionistic modal signature} $i_\Box=\{\land,\lor,\rightarrow,\top,\bot,\Box\}$, and the set of \textit{intuitionistic modal formulas} $Form_{i_\Box}$ is defined recursively as follows: $$\varphi::=p\mid \top\mid \bot\mid \varphi\land \varphi\mid \varphi\lor\varphi\mid \varphi\rightarrow\varphi\mid \Box\varphi$$

$\varphi\leftrightarrow\psi$ stands for $(\varphi\rightarrow \psi)\land(\psi\rightarrow \varphi)$.

\vspace{1mm}
Let $\textbf{IPC}$ denote the intuitionistic propositional calculus, \textit{intuitionistic modal logics} are defined as follows:

\begin{definition}
A logic $L$ over $Form_{i_\Box}$ is an \textit{intuitionistic modal logic} if the following hold:

\begin{itemize}
\item \textbf{IPC} $\subseteq L$
\item $\Box(\varphi\land\psi)\leftrightarrow (\Box\varphi\land \Box\psi)\in L$
\item $\varphi\rightarrow \psi\in L$ implies $\Box\varphi \rightarrow \Box\psi \in L$ (Reg)
\item $\varphi\in L$ implies $\Box\varphi\in L$ (Nec)
\item $\varphi\rightarrow\psi,\varphi\in L$ implies $\psi\in L$ (MP)

\end{itemize}

\end{definition}

We denote the least intuitionistic modal logic \textcolor{black}{by} \textbf{IntK$_\Box$}, and \textbf{IntS4$_\Box$} is just \textbf{IntK$_\Box$}$\oplus(\Box p\rightarrow p)\oplus (\Box p\rightarrow \Box\Box p)$.

\begin{definition}

An \textit{intuitionistic modal multi-conclusion consequence relation} is a multi-conclusion consequence relation $M$ over $Rul_{i_\Box}$ satisfying the following conditions:

\begin{itemize}
    \item $/\varphi\in M$ whenever $\varphi\in \textbf{IntK}_\Box$
    \item $\varphi/\Box\varphi\in M$
    \item $\varphi\rightarrow\psi, \varphi/\psi\in M$
    
\end{itemize}

\end{definition}

Elements in $Rul_{i_\Box}$  are called \textit{intuitionistic modal multi-conclusion rules.} We then define the notations $\mathbf{NExt}(L)$, $\mathbf{NExt}(M)$, $\textbf{IntK$^R_\Box$}$ and \textbf{IntS4$^R_\Box$} where $L$ is an intuitionistic logic and $M$ is an intuitionistic multi-conclusion consequence relation in the same way as we did for  bimodal logics and bimodal multi-conclusion consequence relations.

We have the following counterpart to Proposition \ref{2.4.4} as well.

\begin{proposition}
\label{2.5.3}
The mappings $(-)^R$ and $Taut(-)$ are mutually inverse complete lattice isomorphisms between  $\mathbf{NExt}(\textbf{IntK$_\Box$})$ and the sublattice of $\mathbf{NExt}(\textbf{IntK$^R_\Box$})$ consisting of all intuitionistic modal multi-conclusion consequence relations $M$ such that $Taut(M)^R=M$.    
\end{proposition}

The algebraic semantics for intuitionistic modal logics is given by so-called \textit{modal Heyting algebras}. We start with the definition of \textit{Heyting algebras}.

\begin{definition}
A tuple $\A=(A,\land,\lor,\rightarrow,0,1)$ is a \textit{Heyting algebra} if $(A,\land,\lor,0,1)$ is a bounded distributive lattice such that for any $a,b,c\in A$, $c\land a\leq b$ iff $c\leq a\rightarrow b$.    
\end{definition}

\begin{definition}
A \textit{modal Heyting algebra} is a tuple $\mathfrak{A}=(A,\land,\lor,\rightarrow,0,1,\Box)$ where $(A,\land,\lor,\rightarrow,0,1)$ is a Heyting algebra such that $\Box 1=1$ and $\Box(a\land b)=\Box a\land \Box b$ for any $a,b\in A$. 
\end{definition}

For simplicity, we will write $\mathfrak{A}=(A,\Box)$ for $\mathfrak{A}=(A,\land,\lor,\rightarrow,0,1,\Box)$ where $A$ is assumed to be a Heyting algebra. 

Let $\mathbf{MHA}$ be the class of all modal Heyting algebras, by Theorem \ref{2.2.18}, $\mathbf{MHA}$ is a variety. Let $\textbf{Var}(\mathbf{MHA})$ and $\textbf{Uni}(\mathbf{MHA})$ denote the lattice of subvarieties and the
lattice of universal subclasses of $\mathbf{MHA}$ respectively. \textcolor{black}{Then} we have the following results. Proofs of similar results for Heyting algebras and superintuitionistic logics can be found in \cite[Thm. 7.56]{ChagrovZakharyashev} and \cite[Thm. 2.2]{can}.

\begin{theorem}
\label{2.5.6}
The following maps form pairs of mutually inverse isomorphisms:

\begin{itemize}
    \item Alg: $\mathbf{NExt}(\textbf{IntK}_\Box)\rightarrow \textbf{Var}(\mathbf{MHA})$ and Th: $\textbf{Var}(\mathbf{MHA})\rightarrow \mathbf{NExt}(\textbf{IntK}_\Box)$
\item Alg: $\mathbf{NExt}(\textbf{IntK$^R_\Box$})\rightarrow \textbf{Uni}(\mathbf{MHA})$ and Ru: $\textbf{Uni}(\mathbf{MHA})\rightarrow \mathbf{NExt}(\textbf{IntK$^R_\Box$})$

\end{itemize}

\end{theorem}

\begin{corollary}
\label{comInt}
 Every intuitionistic modal logic (resp. intuitionistic modal multi-conclusion consequence relation) is complete with respect to some variety (resp. universal class) of modal Heyting algebras.   
\end{corollary}

\subsection{Duality}

Finally, we recall dual descriptions of bounded distributive lattices  \textcolor{black}{and} Heyting algebras, which will be frequently used in this paper.

We start with Priestley duality for bounded distributive lattices, for which we refer the reader to \cite{pri}. 

\begin{definition}[Priestley space]

A tuple $(X,\leq)$ (where $\leq$ is a partial order on $X$) is a \textit{Priestley space} if $X$ is a compact space and for any $x,y\in X$, if $x\not\leq y$, then there is a clopen (closed and open) upset $U$ of $X$ such that $x\in U$ while $y\not\in U$.
    
\end{definition}

\begin{definition}[Priestley morphism]

For Priestley spaces $(X,\leq)$ and $(Y,\leq)$, a map $f:X\rightarrow Y$ is a \textit{Priestley morphism} if $f$ is
continuous and order-preserving.
    
\end{definition}

Let $\mathbf{BDL}$ be the category of bounded distributive lattices with bounded
lattice homomorphisms and $\mathbf{PS}$ be the category of Priestley spaces with Priestley morphisms, the functors $(-)_*:\mathbf{BDL}\rightarrow\mathbf{PS}$ and $(-)^*:\mathbf{PS}\rightarrow \mathbf{BDL}$ that establish Priestley duality are constructed as follows. For a bounded distributive lattice $\A$, its dual $\A_*$ is the set of all prime filters $X_A$ of $\A$ with $\subseteq$ as the order and $\{\beta(a)\mid a\in A\}\cup\{X_A\setminus \beta(a)\footnote{We may also denote \textcolor{black}{$X_A\setminus \beta(a)$} as $\delta_a$ for convenience.}\mid a\in A\}$ where $\beta(a)=\{x\in X_A\mid a\in x\}$ as the basis. For a bounded lattice homomorphism $h:A\rightarrow B$, its dual $h_*$ is given by $h^{-1}$. For a Priestley space $\mathcal{X}=(X,\leq)$, its dual $\mathcal{X}^*$ is the bounded distributive lattice of clopen upsets of $\mathcal{X}$ with intersection as meet and union as join. For a Priestley morphism $f:X\rightarrow Y$, its dual $f^*:Y^*\rightarrow X^*$ is given by $f^{-1}$.

We summarise some useful details about Priestley duality in the following theorem.

\begin{theorem}
 $\mathbf{BDL}$ is dually equivalent to $\mathbf{PS}$, which is witnessed by $(-)^*$ and $(-)_*$. In particular, for any bounded distributive lattice $\A$, $\A\cong (\A_*)^*$ witnessed by $\beta$ where $\beta(a)=\{x\in A_*\mid a\in x\}$, and for any Priestley space $\mathcal{X}$, \textcolor{black}{we have} $\mathcal{X}\cong (\mathcal{X}^*)_*$ witnessed by $\epsilon$ where $\epsilon(x)=\{U\in X^*\mid x\in U\}$.   
\end{theorem}

The dual description of Heyting algebras is then given by Esakia duality. The reader may refer to \cite{esa} for more details.

\begin{definition}[Esakia space]

A Priestley space $(X,\leq)$ is an \textit{Esakia space} if for any clopen set $U$ of $X$, \textcolor{black}{we have that} ${\downarrow}U$ is clopen. 
    
\end{definition}

\begin{remark}
\label{2.6.6}
Using an easy argument about general topology, one can easily check that every clopen subset of an Esakia space is of the form $\bigcup_{1\leq i\leq n} (U_i\setminus V_i)$ where $n\in \mathbb{N}$ and $U_i, V_i$'s are clopen upsets.

\end{remark}

For any topological space $X$, we will write $Clop(X)$ for the set of all clopen subsets of $X$.

\begin{definition}
    For Esakia spaces $(X,\leq)$ and $(Y,\leq)$, $f:X\rightarrow Y$ is an \textit{Esakia morphism} if it is continuous, order-preserving and for any $x\in X$, $f(x)\leq z$ implies that there is $\textcolor{black}{y\geq x}$ such that $f(y)=z$.
\end{definition}

Let $\mathbf{ES}$ be the category of Esakia spaces with Esakia morphisms and $\mathbf{HA}$ be the category of Heyting algebras with Heyting algebra homomorphisms, the functors $(-)_*:\mathbf{HA}\rightarrow \mathbf{ES}$ and $(-)^*:\mathbf{ES}\rightarrow \mathbf{HA}$ that establish Esakia duality are constructed as follows: $(-)_*$ is the same as above. For an Esakia space $\mathcal{X}=(X,\leq)$, its dual $\mathcal{X}^*$ is the Heyting algebra of clopen upsets of $\mathcal{X}$ where $U\rightarrow V=X\setminus {\downarrow} (U\setminus V)$. \textcolor{black}{And} $f^*=f^{-1}$ for any Esakia morphism $f$. In particular, we have the following theorem.

\begin{theorem}
 $\mathbf{HA}$ is dually equivalent to $\mathbf{ES}$, which is witnessed by $(-)^*$ and $(-)_*$. In particular, for any Heyting algebra $\A$, $\A\cong (\A_*)^*$ witnessed by $\beta$ where $\beta(a)=\{x\in A_*\mid a\in x\}$, and for any Esakia space $\mathcal{X}$, \textcolor{black}{we have} $\mathcal{X}\cong (\mathcal{X}^*)_*$ witnessed by $\epsilon$ where $\epsilon(x)=\{U\in X^*\mid x\in U\}$.       
\end{theorem}

We finish this section with some useful results about \textit{modal spaces}.

\begin{definition}[Stone space]

A topological space is a \textit{Stone space} if it is a compact Hausdorff space which has a basis of clopen sets.
    
\end{definition}

\begin{remark}
    The category of Boolean algebras with Boolean algebra homomorphisms is dually equivalent to the category of Stone spaces with continuous maps. This duality is called \textit{Stone duality} which is not considered independently in this paper. One only needs to note that for any Boolean algebra $\A$, its dual Stone space is simply its dual Esakia space without the order.
\end{remark}

\begin{definition}[Modal space]

A \textit{modal space} $(X,R)$ consists of a Stone space $X$ and a binary relation $R$ on $X$ such that the following two hold:

\begin{itemize}
    \item For any $x\in X$, \textcolor{black}{the set} $R[x]=\{y\in X\mid xRy\}$ is closed.
    \item For any clopen subset $U$ of $X$, $R^{-1}[U]=\{x\in X\mid xRy \text{ for some }y\in U \}$ is clopen.
\end{itemize}
    
\end{definition}

\begin{remark}

The category of modal spaces with their ``corresponding maps" is dually equivalent to the category of modal algebas with their homomorphisms. We will only spell out some details about this duality (in fact, for bimodal spaces and bimodal algebras) in Section 4 when we need it.
    
\end{remark}

Because of the duality just mentioned, it is not surprising that we can define a \textit{valuation} on a modal space for unary modal logic\footnote{They are simply bimodal logics restricted to one modal operator.}: a \textit{valuation} on a modal space $(X,R)$ is a map $V:Prop \rightarrow Clop(X)$ which can be extended to all modal formulas in the standard way \footnote{Usually people still write $V$ for the extended map.}. For any modal formula $\varphi$, we write $(X,R)\vDash \varphi$ if for any valuation $V$, $V(\varphi)=X$. Then for any 
unary modal logic $L$, we call a modal space $(X,R)$ an \textit{$L$-space} if for any $\varphi\in L$, we have that $(X,R)\vDash \varphi$. In particular, an \textit{$S4$-space} is a modal space $(X,R)$ where $R$ is reflexive and transitive.

We recall the following useful results about $Grz$-spaces, whose proofs can be found in \cite[Ch. 3]{grz}.

\begin{theorem}
\label{2.6.13}
For any $Grz$-space $(X,R)$ and $U\in Clop(X)$, the following hold:

\begin{itemize}

\item $max_R(U)$ is closed.

\item $max_R(U)\subseteq pas_R(U)$.

\item $max_R(U)$ does not cut any $R$-cluster. 

\end{itemize} 
\end{theorem}

\textcolor{black}{This concludes our general preliminaries. We can now start developing the theory of stable canonical rules for intuitionistic modal logics.}

\section{Stable Canonical Rules for Intuitionistic Modal Logics}

We begin by introducing stable canonical rules and \textcolor{black}{proving} that every intuitionistic modal multi-conclusion consequence
relation is axiomatizable by stable canonical rules. Then using the duality built in \cite{Al}, we give a dual (geometric) characterization of our stable canonical rules, which will be quite useful in their applications.

\label{3.1}
We start with the definition of \textit{stable maps}.

\begin{definition}

Let $\A=(A,\Box)$ and $\B=(B,\Box)$ be modal Heyting algebras, 
and let $h:A\rightarrow B$ be a bounded lattice homomorphism. \textcolor{black}{We say that} $h$ is \textit{stable} if for any $a\in A$, \textcolor{black}{we have} $h(\Box a)\leq \Box h(a)$.
\end{definition}

This definition is the analogue to the one given in \cite[Def. 3.1]{nrule} in the setting of classical modal logics.

\begin{definition}
Let $\A=(A,\Box)$, $\B=(B,\Box)$ be \textcolor{black}{modal} Heyting algebras, $D^\rightarrow\subseteq A\times A$ and $D^\Box\subseteq A$. A bounded lattice embedding $h:A\rightarrow B$ satisfies

\begin{itemize}
    \item  the \textit{closed domain condition} (CDC for short) for $D^\rightarrow$ if $h(a\rightarrow b)=h(a)\rightarrow h(b)$ for any $(a,b)\in D^\rightarrow$. 
    \item the \textit{closed domain condition} (CDC for short) for $D^\Box$ if $h(\Box a)=\Box h(a)$ for any $a\in D^\Box$.    
\end{itemize}  
\end{definition}

The following proposition relates each intuitionistic modal multi-conclusion rule with finitely many finite refutation patterns by stable bounded lattice embeddings which satisfy CDC \textcolor{black}{for some parameters}.

\begin{proposition}
\label{3.1.3}
    For each intuitionistic modal multi-conclusion rule $\Gamma/\Delta$, there exist $(\A_1,D^{\rightarrow}_1,D^{\Box}_1)$ $,...,(\A_n,D^{\rightarrow}_n,D^{\Box}_n)$ such that each $\A_i$ is a finite modal Heyting algebra, $D^{\rightarrow}_i\subseteq A_i\times A_i$ and $D^{\Box}_i\subseteq A_i$, and for each modal Heyting algebra $\B=(B,\Box)$, \textcolor{black}{we have} that $\B\not\vDash \Gamma/\Delta$ iff there is $i\leq n$ and a stable bounded lattice embedding $h:A_i\rightarrow B$ satisfying CDC for $D^{\rightarrow}_i$ and $D^{\Box}_i$.
\end{proposition}

\begin{proof}
Let $\Gamma/\Delta$ be an arbitrary intuitionistic modal multi-conclusion rule. If $ \Gamma/\Delta\in \textbf{IntK$^R_\Box$}$, take $n=0$. Suppose $ \Gamma/\Delta\not\in \textbf{IntK$^R_\Box$}$, let $\Theta$ be the set of all subformulas of the formulas in $\Gamma\cup\Delta$. \textcolor{black}{Clearly} $\Theta$ is finite. Assume $|\Theta|=m$, since the variety of bounded distributive lattices is locally finite, there are only finitely many pairs $(\A,D^{\rightarrow},D^{\Box})$ satisfying the following two conditions up to isomorphism:

\begin{itemize}
    \item[i)] $\A=(A,\Box)$ is a finite modal Heyting algebra such that $\A|_{\{\land,\lor,1,0\}}$ is at most $m$-generated as a bounded distributive lattice and $\A\not\vDash \Gamma/\Delta$.
    \item[ii)] $D^\rightarrow=\{(V(\varphi),V(\psi))\mid \varphi\rightarrow\psi\in\Theta\}$ and $D^\Box=\{V(\psi)\mid \Box\psi\in\Theta\}$ where $V$ is a valuation on $\A$ witnessing $\A\not\vDash \Gamma/\Delta$.
\end{itemize}

Let $(\A_1,D^{\rightarrow}_1,D^{\Box}_1),...,(\A_n,D^{\rightarrow}_n,D^{\Box}_n)$ be the enumeration of such pairs. For any modal Heyting algebra $\B=(B,\Box)$ , we prove that $\B\not\vDash\Gamma/\Delta$ iff there is $i\leq n$ and a stable bounded lattice embedding $h:A_i\rightarrow B$ satisfying CDC for $D^{\rightarrow}_i$ and $D^{\Box}_i$.

For the right-to-left direction, suppose there is $i\leq n$ and a stable bounded lattice embedding $h:A_i\rightarrow B$ satisfying CDC for $D^{\rightarrow}_i$ and $D^{\Box}_i$. Define a valuation $V_B$ on $\B$ by $V_B(p)=h(V_i(p))$ for any propositional letter $p$ where $V_i$ is \textcolor{black}{the valuation on $\A_i$ witnessing $\A_i\not\vDash\Gamma/\Delta$.} We then prove by induction that $V_B(\psi)=h(V_i(\psi))$ for any $\psi\in \Theta$. We only consider following two cases as other cases are trivial ($h$ is a bounded lattice embedding):

If $\psi=\Box\varphi$, then as $\Box\varphi\in \Theta$, $V_i(\varphi)\in D^{\Box}_i$. 

\textcolor{black}{\begin{tabular}{l l l l}
$V_B(\Box\varphi)$ &=&$\Box V_B(\varphi)$\\  
   &=& $\Box h(V_i(\varphi))$ & (IH) \\
   &=& $h(\Box V_i(\varphi))$ & (CDC) \\
   &=& $h(V_i(\Box\varphi)).$   
\end{tabular}}

If $\psi=\varphi\rightarrow\chi$, then as $\varphi\rightarrow\chi\in \Theta$, $(V_i(\varphi),V_i(\chi))\in D^{\rightarrow}_i$. 
\vspace{2mm}

\textcolor{black}{\begin{tabular}{l l l l}
$V_B(\varphi\rightarrow\chi)$ &=&$V_B(\varphi)\rightarrow V_B(\chi)$\\
   &=& $h(V_i(\varphi))\rightarrow h(V_i(\chi))$ & (IH) \\
   &=& $h(V_i(\varphi)\rightarrow V_i(\chi))$ & (CDC) \\
   &=& $h(V_i(\varphi\rightarrow\chi)).$  \end{tabular}}

\vspace{1mm}
\textcolor{black}{Since $V_i(\gamma)=1_{A_i}$ for any $\gamma\in\Gamma$ and $h$ is a bounded lattice embedding, $V_B(\gamma)=h(V_i(\gamma))=h(1_{A_i})=1_B$ for any $\gamma\in\Gamma$. Since $V_i(\delta)\not=1_{A_i}$ for any $\delta\in\Delta$ and $h$ is a bounded lattice embedding,
$V_B(\delta)=h(V_i(\delta))\not=1_B$ for any $\delta\in \Delta$. Thus $\B\not\vDash\Gamma/\Delta$.}

For the left-to-right direction, suppose $\B\not\vDash\Gamma/\Delta$. There exists a valuation $V_B$ on $B$ such that $V_B(\gamma)=1_B$ for any $\gamma\in\Gamma$ and $V_B(\delta)\not=1_B$ for any $\delta\in \Delta$. Let $B'$ be the bounded sublattice of $B$ generated by $V_B(\Theta)=\{V_B(\varphi)\mid \varphi\in \Theta\}$. \textcolor{black}{Note that $B'$ is finite as the variety of bounded distributive lattices is locally finite}. Clearly $|V_B(\Theta)|\leq |\Theta|$. Let $D^\Box=\{V_B(\psi)\mid \Box\psi\in\Theta\}$ and $D^\rightarrow=\{(V_B(\varphi),V_B(\psi))\mid \varphi\rightarrow\psi\in \Theta\}$. We define $\rightarrow'$ and $\Box'$ on $B'$ as follows: $a\rightarrow' b=\bigvee \{d\in B'\mid d\land a\leq b\}$ for any $a,b\in B'$; $\Box'a=\bigvee\{\Box b\mid \Box b\leq \Box a\text{ and } b,\Box b\in B'\}$ for any $a\in B'$.

We first check that $(B',\rightarrow',\Box')$ is a modal Heyting algebra. Clearly, $(B',\rightarrow')$ is a Heyting algebra \textcolor{black}{by the definition of $\rightarrow'$}. Since $\Box 1=1$ and $1\in B'$, \textcolor{black}{we have that} $\Box' 1=\Box 1=1$. 

\vspace{1mm}
For any $a,b\in B'$, $\Box'a\land \Box'b$
~\\
$=\bigvee\{\Box x\leq \Box a\text{ and } x,\Box x\in B'\}\land \bigvee \{\Box y\mid \Box y\leq \Box b\text{ and } y,\Box y\in B'\}$
~\\
$=\bigvee\{\Box x\land \Box y\mid \Box x\leq\Box a, \Box y\leq\Box b \text{ where } x, y, \Box x,\Box y\in B'\}$\text{(distributivity)}
~\\
$=\bigvee \{\Box(x\land y)\mid \Box x\leq \Box a\text{ and }\Box y\leq \Box b \text{ where } x,y,\Box x,\Box y\in B'\}$
~\\
$=\bigvee\{\Box z\mid \Box z\leq \Box(a\land b)\text{ and } z,\Box z\in B'\}=\Box'(a\land b)$.

This proves that $(B',\rightarrow',\Box')$ is a modal Heyting algebra. Let $h:(B',\rightarrow',\Box')\rightarrow (B,\Box)$ be the inclusion map, $h$ is clearly a bounded lattice embedding as $B'$ is a bounded sublattice of $B$. $h$ is stable as $\Box'a\leq \Box a$ for any $a\in B'$ by definition.

Then we check that $h$ satisfies CDC for $D^\rightarrow$ and $D^\Box$. For any $a\in D^\Box$, $a=V_B(\psi)$ for some $\Box\psi\in \Theta$. \textcolor{black}{And} $V_B(\Box\psi)=\Box V_B(\psi)=\Box a\in B'$. Thus $\Box' a=\Box a$ by the definition of $\Box'$. For any $(a,b)\in D^\rightarrow$, $a=V_B(\varphi)$ and $b=V_B(\psi)$ for some $\varphi\rightarrow\psi\in\Theta$. Thus $V_B(\varphi\rightarrow\psi)=V_B(\varphi)\rightarrow V_B(\psi)=a\rightarrow b\in B'$. Then $a\rightarrow b'=a\rightarrow b$ by the definition of $\rightarrow'$. Therefore, the stable bounded lattice embedding $h$ satisfies CDC for $D^\rightarrow$ and $D^\Box$.

Let $V'$ be the valuation $V_B$ restricted to $B'$, we then prove that for any $\varphi\in \Theta$, $V'(\varphi)=V_B(\varphi)$ by induction on $\varphi$. We only consider the following two cases as others are trivial ($B'$ is a bounded sublattice of $B$):

If $\varphi=\psi\rightarrow \chi$, as $\psi\rightarrow \chi\in\Theta$, we have that $(V_B(\psi),V_B(\chi))\in D^\rightarrow$, and $V_B(\psi)\rightarrow V_B(\chi)\in B'$. 

\vspace{2mm}
\textcolor{black}{\begin{tabular}{l l l l}
$V'(\psi\rightarrow \chi)$ &=&$V'(\psi)\rightarrow' V'(\chi)$\\  
   &=& $V_B(\psi)\rightarrow' V_B(\chi)$ & (IH) \\
   &=& $V_B(\psi)\rightarrow V_B(\chi)$ & (By the definition of $\rightarrow'$) \\
   &=& $V_B(\psi\rightarrow\chi).$   
\end{tabular}}

\vspace{1mm}
If $\varphi=\Box\psi$, as $\Box\psi\in \Theta$, we have that $V_B(\Box\psi)\in B'$, and $V_B(\psi),\Box V_B(\psi)\in B'$.

\vspace{1mm}
\textcolor{black}{\begin{tabular}{l l l l}
$V'(\Box\psi)$ &=&$\Box' V'(\psi)$\\  
   &=& $\Box' V_B(\psi)$ & (IH) \\
   &=& $\Box V_B(\psi)$ & (By the definition of $\Box'$) \\
   &=& $V_B(\Box\psi).$   
\end{tabular}}

\vspace{1mm}
Since $V_B$ is a valuation which refutes $\Gamma/\Delta$ on $\B$, $V'$ is a valuation which refutes $\Gamma/\Delta$ on $(B',\rightarrow',\Box')$ by the above result. \textcolor{black}{Thus} $(B',\rightarrow',\Box')\not\vDash \Gamma/\Delta$. As for any $\varphi\in \Theta$, $V'(\varphi)=V_B(\varphi)$, \textcolor{black}{we have that} $D^\Box=\{V_B(\psi)\mid \Box\psi\in\Theta\}=\{V'(\psi)\mid \Box\psi\in\Theta\}$ and $D^\rightarrow=\{(V_B(\varphi),V_B(\psi))\mid \varphi\rightarrow\psi\in \Theta\}=\{(V'(\varphi),V'(\psi))\mid \varphi\rightarrow\psi\in \Theta\}$. As $B'$ is generated by $V_B(\Theta)$ whose cardinality is no larger than that of $\Theta$, $(B',\rightarrow',\Box',D^\rightarrow, D^\Box)$ must be one of $(\A_1,D^{\rightarrow}_1,D^{\Box}_1),...,(\A_n,D^{\rightarrow}_n,D^{\Box}_n)$. As $h$ is a stable bounded lattice embedding from $B'$ to $B$ satisfying CDC for $D^\rightarrow$ and $D^\Box$, we get what we want.

\end{proof}

In the above proof, what we have done is essentially a filtration in algebraic terms\footnote{See \cite{jphd} for more details.}: for the left-to-right direction, we begin with the assumption that $\B\not\vDash \Gamma/\Delta$, and then use $\B$ to construct a finite modal \textcolor{black}{Heyting} algebra $(B',\rightarrow',\Box')$ \textcolor{black}{-- a filtrated algebra --} which still refutes $\Gamma/\Delta$. And it turns out that the relation between the filtrated algebra $(B',\rightarrow',\Box')$ and the original algebra $\B$ (i.e., a stable bounded lattice embedding satisfying CDC for some parameters) can be coded syntactically. 

\begin{definition}
\label{3.1.4}

Let $\A=(A,\Box)$ be a finite modal Heyting algebra, $D^\rightarrow \subseteq A\times A$ and $D^\Box\subseteq A$. For each $a\in A$, we introduce a new propositional letter $p_a$ and define the \textit{stable canonical rule} $\rho(\A,D^\rightarrow,D^\Box)$ based on $(\A,D^\rightarrow,D^\Box)$ as follows:

\vspace{1mm}
\begin{tabular}{l l l l}
\renewcommand\arraystretch{2}
$\Gamma$&=&$\{p_{a\lor b}\leftrightarrow p_a\lor p_b\mid a,b\in A\}\cup\{p_0\leftrightarrow \bot, p_1\leftrightarrow \top\}$\\  
   & & $\cup\{p_{a\land b}\leftrightarrow p_a\land p_b\mid a,b\in A\}\cup\{p_{\Box a}\rightarrow \Box p_a\mid a\in A\}$ &  \\
   & & $\cup\{p_{a\rightarrow b}\leftrightarrow (p_a\rightarrow p_b)\mid (a,b)\in D^\rightarrow\}\cup\{\Box p_a\rightarrow p_{\Box a}\mid a\in D^\Box\}$ 
\end{tabular}

$$\Delta=\{p_a\leftrightarrow p_b\mid a\not=b\in A\}$$

$\rho(\A,D^\rightarrow,D^\Box)=\Gamma/\Delta$.

\end{definition}

The above definition can be seen as a combination of \cite[Def. 3.1]{nloc} and \cite[Def. 5.2]{nrule}. \textcolor{black}{It is easy to see that the following proposition holds:}

\begin{proposition}
\label{3.1.5}
Let $\A=(A,\Box)$ be a finite modal Heyting algebra, $D^\rightarrow \subseteq A\times A$ and $D^\Box\subseteq A$, \textcolor{black}{then} $\A\not\vDash\rho(\A,D^\rightarrow,D^\Box)$.
\end{proposition}

\begin{proof}
Define a valuation $V$ on $A$ by $V(p_a)=a$ for any $a\in A$. It is then easy to check that $V$ refutes $\rho(\A,D^\rightarrow,D^\Box)$ on $\A$.
\end{proof}

\textcolor{black}{The next result shows that} the stable canonical rule does encode a stable bounded lattice embedding satisfying CDC for $D^\rightarrow$ and $D^\Box$. 

\begin{proposition}
\label{3.1.6}
Let $\A=(A,\Box)$ be a finite modal Heyting algebra, $D^\rightarrow \subseteq A\times A$, $D^\Box\subseteq A$, and $\B=(B,\Box)$ be a modal Heyting algebra. Then $\B\not\vDash\rho(\A,D^\rightarrow,D^\Box)$ iff there is a stable bounded lattice embedding $h:A\rightarrow B$ satisfying CDC for $D^\rightarrow$ and $D^\Box$.
\end{proposition}

\begin{proof}

For the right-to-left direction, suppose there is a stable bounded lattice embedding $h:A\rightarrow B$ satisfying CDC for $D^\rightarrow$ and $D^\Box$. Define $V_B$ on $B$ by $V_B(p_a)=h(V(p_a))=h(a)$ for any $a\in A$ where $V$ is just the valuation in the proof of Proposition \ref{3.1.5}. As $h$ is a bounded lattice embedding, for any $a,b\in A$, \textcolor{black}{it follows that} $h(a\land b)=h(a)\land h(b)$, $h(a\lor b)=h(a)\lor h(b)$, $h(0)=0$ and $h(1)=1$. It is then easy to check that $V_B(p_{a\lor b}\leftrightarrow p_a\lor p_b)=1$, $V_B(p_{a\land b}\leftrightarrow p_a\land p_b)=1$, $V_B(p_0)=0$ and $V_B(p_1)=1$.

As $h$ is stable, $h(\Box a)\leq \Box h(a)$ for any $a\in A$. \textcolor{black}{Thus} $V_B(p_{\Box a})=h(\Box a)\leq \Box h(a)=\Box V_B(p_a)=V_B(\Box p_a)$ for any $a\in A$. Therefore, $V_B(p_{\Box a}\rightarrow \Box p_a)=V_B(p_{\Box a})\rightarrow V_B(\Box p_a)=1$ for any $a\in A$.

As $h$ satisfies CDC for $D^\rightarrow$ and $D^\Box$, for any $a\in D^\Box$, \textcolor{black}{we have that} $h(\Box a)=\Box h(a)$; for any $(a,b)\in D^\rightarrow$, \textcolor{black}{we have that} $h(a\rightarrow b)=h(a)\rightarrow h(b)$. Thus
$V_B(p_{\Box a})=h(\Box a)=\Box h(a)=\Box V_B(p_a)=V_B(\Box p_a)$, we get $V_B(\Box p_a\rightarrow p_{\Box a})=1$ for any $a\in D^\Box$. 

For any $(a,b)\in D^\rightarrow$, we have that

\vspace{1mm}
\textcolor{black}{\vspace{1mm}
\begin{tabular}{l l l l}
$V_B(p_{a\rightarrow b}\leftrightarrow p_a\rightarrow p_b)$ &=&$V_B(p_{a\rightarrow b})\leftrightarrow V_B(p_a\rightarrow p_b)$\\  
   &=& $V_B(p_{a\rightarrow b})\leftrightarrow (V_B(p_a)\rightarrow V_B(p_b))$ & \\
   &=& $h(a\rightarrow b)\leftrightarrow (h(a)\rightarrow h(b))$ &  \\
   &=& $1.$   
\end{tabular}}

Since $h$ is an embedding, for any $a\not=b\in A$, $h(a)\not=h(b)$. Thus $V_B(p_a)\not=V_B(p_b)$ \textcolor{black}{and} $V_B(p_a\leftrightarrow p_b)\not=1$. Therefore, for any $\gamma\in\Gamma$, \textcolor{black}{we have that} $V_B(\gamma)=1$ while for any $\delta\in\Delta$, $V_B(\delta)\not=1$. Thus $V_B$ refutes $\Gamma/\Delta$ on $\B$, $\B\not\vDash\rho(\A,D^\rightarrow,D^\Box)$.

For the left-to-right direction, suppose $\B\not\vDash\rho(\A,D^\rightarrow,D^\Box)$. \textcolor{black}{Then} there exists a valuation $V$ on $B$ such that $V(\gamma)=1$ for any $\gamma\in\Gamma$, and $V(\delta)\not=1$ for any $\delta\in \Delta$. Define $h:A\rightarrow B$ by $h(a)=V(p_a)$ for any $a\in A$. It is then easy to check that $h$ is a bounded lattice homormophism. As for any $a\in A$, $V(p_{\Box a}\rightarrow \Box p_a)=1$, $V(p_{\Box a})\leq \Box V(p_a)$, and thus $h(\Box a)\leq \Box h(a)$. $h$ is stable.

For any $(a,b)\in D^\rightarrow$, as $V(p_{a\rightarrow b}\leftrightarrow p_a\rightarrow p_b)=1$, we have $V(p_{a\rightarrow b})=V(p_a)\rightarrow V(p_b)$, and thus $h(a\rightarrow b)=V(p_{a\rightarrow b})=V(p_a)\rightarrow V(p_b)=h(a)\rightarrow h(b)$. For any $a\in D^\Box$, as $V(\Box p_a\rightarrow p_{\Box a})=1$, $V(\Box p_a)\leq V(p_{\Box a})$, $h(\Box a)=V(p_{\Box a})\geq V(\Box p_a)=\Box V(p_a)=\Box h(a)$. Therefore, $h(\Box a)=\Box h(a)$ for any $a\in D^\Box$.

For any $a\not=b\in A$,  as $V(p_a\leftrightarrow p_b)\not=1$, \textcolor{black}{it follows that} $V(p_a)\not=V(p_b)$ and $h(a)=V(p_a)\not=V(p_b)=h(b)$. Therefore, $h$ is a stable bounded lattice embedding satisfying CDC for for $D^\rightarrow$ and $D^\Box$.
\end{proof}

Now, combining Propositions \ref{3.1.3} and \ref{3.1.6}, we \textcolor{black}{obtain} immediately the following result:

\begin{theorem}
\label{3.1.7}
For an intuitionistic modal multi-conclusion rule $\Gamma/\Delta$, there exist $(\A_1,D^{\rightarrow}_1,D^{\Box}_1),$ $...,(\A_n,D^{\rightarrow}_n,D^{\Box}_n)$ such that each $\A_i$ is a finite modal Heyting algebra, $D^{\rightarrow}_i\subseteq A_i\times A_i$ and $D^{\Box}_i\subseteq A_i$, and for each modal Heyting algebra $\B=(B,\Box)$, we have: $$\B\vDash \Gamma/\Delta \text{ iff } \B\vDash \rho(\A_1,D^{\rightarrow}_1,D^{\Box}_1),...,\rho(\A_n,D^{\rightarrow}_n,D^{\Box}_n).$$
\end{theorem}

As a corollary, \textcolor{black}{we arrive at the main theorem of this section.}

\begin{theorem}
\label{3.1.9}
Every intuitionistic modal multi-conclusion consequence relation is axiomatizable by stable canonical rules.    
\end{theorem}

\begin{proof}
 \textcolor{black}{Let $M\in\mathbf{NExt}(\textbf{IntK$^R_\Box$})$ be arbitrary}. \textcolor{black}{Then} $M=\textbf{IntK$^R_\Box$}\oplus \{\Gamma_i/\Delta_i\mid i\in I\}$ where $\Gamma_i/\Delta_i$ is an intuitionistic modal multi-conclusion rule. For any $i\in I$, by Theorem \ref{3.1.7}, there exist stable canonical rules $\rho(\A_{i1},D^{\rightarrow}_{i1},D^{\Box}_{i1}),...,\rho(\A_{in_i},D^{\rightarrow}_{in_i},D^{\Box}_{in_i})$ such that for any modal Heyting algebra $\B=(B,\Box)$, \textcolor{black}{we have that} $\B\vDash \Gamma_i/\Delta_i$ iff $\B\vDash \rho(\A_{i1},D^{\rightarrow}_{i1},D^{\Box}_{i1}),...,\rho(\A_{in_i},D^{\rightarrow}_{in_i},D^{\Box}_{in_i})$. Therefore, for any modal Heyting algebra $\B=(B,\Box)$, \textcolor{black}{we have that} $\B$ validates $M$ iff $\B$ validates $\{\rho(\A_{i1},D^{\rightarrow}_{i1},D^{\Box}_{i1}),...,\rho(\A_{in_i},D^{\rightarrow}_{in_i},D^{\Box}_{in_i})\mid i\in I\}$. By Corollary \ref{comInt}, this means that $M=\textbf{IntK$^R_\Box$}\oplus\{\rho(\A_{i1},D^{\rightarrow}_{i1},D^{\Box}_{i1}),...,\rho(\A_{in_i},D^{\rightarrow}_{in_i},D^{\Box}_{in_i})\mid i\in I\}$. \textcolor{black}{Therefore,} $M$ is axiomatized by stable canonical rules. This proves that every intuitionistic modal multi-conclusion consequence  is axiomatizable by stable canonical rules.
     
\end{proof}

As every modal multi-conclusion consequence relation\footnote{ $\Box$ is the primitive operator and $\Diamond$ is defined by $\Box$.} is also an intuitionistic modal multi-conclusion consequence relation, Theorem \ref{3.1.9} is a generalization of \cite[Thm. 5.6]{nrule}. When considering intuitionistic modal multi-conclusion rules, the above theorem allows us to assume that they are stable canonical rules in many cases. Since stable canonical rules are in a certain syntactical shape, it is more manageable to work with them instead of  arbitrary intuitionistic modal multi-conclusion rules. This point may become more evident when we see the dual description \textcolor{black}{of} stable canonical rules, which gives us geometric intuitions about \textcolor{black}{how these rules work}.

Now, we recall the duality between the category of modal Heyting algebras and the category of modal Esakia spaces which is established in \cite{Al} and makes our dual description possible.

\begin{definition}
\label{3.3.1}
Let $(X,\leq,R)$ be a triple such that $(X,\leq)$ is an Esakia space and $R\subseteq X\times X$, then $(X,\leq,R)$ is a \textit{modal Esakia space}\footnote{It is also called \textit{$I_\Box$-space} in \cite{Al}.} if the following two conditions hold:

\begin{itemize}
    \item If $U$ is a clopen upset of $X$, then $\Box_R U$ is a clopen upset as well where $\Box_R U=\{x\in X\mid R[x]\subseteq U\}$.
    \item For every $x\in X$, $R[x]$ is a closed upset.
\end{itemize}

\end{definition}

The morphisms between modal Esakia spaces are given as follows:

\begin{definition}
Let $(X_1,\leq,R_1)$ and $(X_2,\leq, R_2)$ be modal Esakia spaces, a map $f:X_1\rightarrow X_2$ is called a \textit{p-morphism} if the following conditions are satisfied for any $x,y\in X_1$ and $z\in X_2$:

\begin{itemize}
    \item[1.] $f$ is continuous.
    \item[2.] If $x\leq y$, then $f(x)\leq f(y)$.
    \item[3.] If $f(x)\leq z$, then $f(x')=z$ for some $\textcolor{black}{x'\geq x}$.
    \item[4.] If $xR_1y$, then $f(x)R_2f(y)$.
    \item[5.] If $f(x)R_2 z$, then $f(x')\leq z$ for \textcolor{black}{some $x'\in X_1$ s.t} $xR_1x'$.
\end{itemize}

\end{definition}

Let $\mathbf{MHA}$ be the category of modal Heyting algebras and modal Heyting algebra homomorphisms, $\mathbf{MES}$ be the category of modal Esakia spaces and p-morphisms, the functors $(-)_*:\mathbf{MHA}\rightarrow\mathbf{MES}$ and $(-)^*:\mathbf{MES}\rightarrow \mathbf{MHA}$ that establish the duality are constructed as follows. For a modal Heyting algebra $\A=(A,\Box)$, let $\A_*=(A_*,R)$ where $A_*$ is the Esakia space of $A$ and $xRy$ iff $\forall \Box a\in A(\Box a\in x\implies a\in y)$. For a modal Esakia space $\mathcal{X}=(X,\leq, R)$, let $\mathcal{X}^*=(X^*,\Box_R)$ where $X^*$ is the Heyting algebra of clopen upsets of $X$ and $\Box_RU=\{x\in X\mid R[x]\subseteq U\}$. The duals of maps are exactly the same as that in Esakia duality. We spell out some useful details about the duality in the following theorem.

\begin{theorem}\cite[Thm 6.12]{Al}
\label{3.3.3}
$\mathbf{MHA}$ is dually equivalent to $\mathbf{MES}$, which is witnessed by $(-)^*$ and $(-)_*$. In particular, for any modal Heyting algebra $\A$, $\A\cong (\A_*)^*$ witnessed by $\beta$ where $\beta(a)=\{x\in A_*\mid a\in x\}$, and for any modal Esakia space $\mathcal{X}$, $\mathcal{X}\cong (\mathcal{X}^*)_*$ witnessed by $\epsilon$ where $\epsilon(x)=\{U\in X^*\mid x\in U\}$.
\end{theorem}

By the above duality, we can now give a dual description \textcolor{black}{of} stable bounded lattice homomorphisms.

\begin{definition}
Let $(X,\leq, R)$ and $(Y,\leq, R)$ be modal Esakia spaces and $f:X\rightarrow Y$ be a Priestley morphism. The map $f$ is called \textit{stable} if for any $x,y\in X$, $xRy$ implies $f(x)Rf(y)$. 
\end{definition}

\begin{proposition}
\label{3.3.6}
Let $\A=(A,\Box)$ \textcolor{black}{and} $\B=(B,\Box)$ be modal Heyting algebras.   \textcolor{black}{Let} $(X_A,\leq,R)$ and $(X_B,\leq, R)$ be the \textcolor{black}{duals} of $\A$ and $\B$ respectively. For a bounded lattice homomorphism $h:A\rightarrow B$, $h$ is stable iff $h_*:X_B\rightarrow X_A$ is stable. 
\end{proposition}

\begin{proof}
By Priestley duality, $h_*$ is a Priestley morphism. Thus, it suffices to prove that $h(\Box a)\leq \Box h(a)$ for any $a\in A$ iff $xRy$ implies $h_*(x)Rh_*(y)$ for any $x,y\in X_B$.

Suppose $h(\Box a)\leq \Box h(a)$ for any $a\in A$. Suppose $xRy$ and $\Box a\in h_*(x)$, then $h(\Box a)\in x$. As $h(\Box a)\leq \Box h(a)$ and $x$ is a prime filter, $\Box h(a)\in x$. As $xRy$, \textcolor{black}{it follows that} $h(a)\in y$ \textcolor{black}{and} $a\in h_*(y)$. Thus $h_*(x)Rh_*(y)$.

For the other direction, suppose $xRy$ implies $h_*(x)Rh_*(y)$ for any $x,y\in X_B$. For any $x\in \beta(h(\Box a))$, \textcolor{black}{we have that} $h(\Box a)\in x$ and $\Box a\in h_*(x)$. Now for any $xRy$, by assumption, $h_*(x)Rh_*(y)$. As $\Box a\in h_*(x)$, \textcolor{black}{it follows that} $a\in h_*(y)$ and $h(a)\in y$. Thus $R[x]\subseteq \beta(h(a))$, and $x\in \Box_R(\beta(h(a)))=\beta(\Box h(a))$. As $x\in \beta(h(\Box a))$ is arbitrary, this proves that $\beta(h(\Box a))\subseteq \beta(\Box h(a))$. As $\beta$ is an isomorphism, $h(\Box a)\leq \Box h(a)$.
\end{proof}

Similarly, we can use the duality to give a dual description of CDC for $D^\rightarrow$ and $D^\Box$.

\begin{definition}
Let $(X,\leq, R)$ and $(Y,\leq, R)$ be modal Esakia spaces, $f: X\rightarrow Y$ be a Priestley morphism, and $D$ be
a clopen subset of $Y$. We say that $f$ satisfies the \textit{implication closed domain condition (CDC$_\rightarrow$)} for $D$ if \textcolor{black}{the following holds: 
$$\uparrow f(x)\cap D\not=\emptyset \text{ implies } f[\uparrow x]\cap D\not=\emptyset.$$}
Furthermore, let $\mathcal{D}$ be a collection of clopen subsets of $Y$, $f$ satisfies the \textit{implication closed
domain condition (CDC$_\rightarrow$)} for $\mathcal{D}$ if $f$ satisfies (CDC$_\rightarrow$) for each $D\in \mathcal{D}$.

\end{definition}

We then have the following proposition \textcolor{black}{analogous to \cite[Lem. 4.3]{nloc}}, which connects the algebraic CDC for $D^\rightarrow$ with the geometric CDC$_\rightarrow$.

\begin{proposition}\cite[Lem. 4.3]{nloc}

Let $\A$ and $\B$ be modal Heyting algebras, $h:A\rightarrow B$ be a bounded lattice homomorphism, and $a,b\in A$, then the following two conditions are equivalent:

\begin{itemize}
    \item [1.] $h(a\rightarrow b)=h(a)\rightarrow h(b)$.
    \item [2.] $h_*$ satisfies CDC$_\rightarrow$ for $\beta(a)\setminus \beta(b)$.
\end{itemize}

\end{proposition}

For the dual description of CDC for $D^\Box$, we have the following.

\begin{definition}
\label{3.3.9}
Let $(X,\leq, R)$ and $(Y,\leq, R)$ be modal Esakia spaces, $f: X\rightarrow Y$ be a Priestley morphism, and $D$ be
a clopen subset of $Y$. We say that $f$ satisfies the \textit{modal closed domain condition (CDC$_\Box$)} for $D$ if \textcolor{black}{the following holds: 
$$f[R[x]]\subseteq D \text{ implies } R[f(x)]\subseteq D.\footnote{By contraposition, this is equivalent to the condition that $R[f(x)]\cap \Bar{D}\not=\emptyset$ implies $f[R[x]]\cap \Bar{D}\not=\emptyset$. Sometimes it may be more convenient to use this one.}$$}
Furthermore, let $\mathcal{D}$ be a collection of clopen subsets of $Y$, $f$ satisfies the \textit{modal closed
domain condition (CDC$_\Box$)} for $\mathcal{D}$ if $f$ satisfies (CDC$_\Box$) for each $D\in \mathcal{D}$.

\end{definition}

\begin{proposition}
Let $\A=(A,\Box)$ and $\B=(B,\Box)$ be modal Heyting algebras, $h:A\rightarrow B$ be a stable bounded lattice homomorphism, and $a\in A$, then the following are equivalent:

\begin{itemize}
    \item [1.] $h(\Box a)=\Box h(a)$.
    \item [2.] $h_*:X_B\rightarrow X_A$ satisfies CDC$_\Box$ for $\beta(a)$.
\end{itemize}

\end{proposition}

\begin{proof}
As $h$ is stable and $\beta$ is an isomorphism, $h(\Box a)=\Box h(a)$ iff $\Box h(a)\leq h(\Box a)$ iff $\beta(\Box h(a))\subseteq \beta(h(\Box a))$ iff $\Box_R \beta(h(a))\subseteq \beta(h(\Box a))$.

Suppose $h(\Box a)=\Box h(a)$, thus $\Box_R \beta(h(a))\subseteq \beta(h(\Box a))$. Suppose $h_*[R[x]]\subseteq \beta(a)$, \textcolor{black}{then} $R[x]\subseteq h^{-1}_*(\beta(a))$. For any $y\in R[x]$, \textcolor{black}{we have that} $h_*(y)\in \beta(a)$ and $a\in h_*(y)$, $h(a)\in y$. Thus $R[x]\subseteq \beta(h(a))$, and $x\in \Box_R(\beta(h(a)))\subseteq\beta(h(\Box a))$. Namely, $h(\Box a)\in x$ \textcolor{black}{and} $\Box a\in h_*(x)$. Thus, for any  $y\in R[h_*(x)]$, \textcolor{black}{we have that} $a\in y$. Therefore, $R[h_*(x)]\subseteq \beta(a)$. \textcolor{black}{This proves that} $h_*$ satisfies CDC$_\Box$ for $\beta(a)$.

For the other direction, suppose $h_*$ satisfies CDC$_\Box$ for $\beta(a)$. Suppose $x\in \Box_R(\beta(h(a)))$, then $R[x]\subseteq \beta(h(a))$. For any $xRy$, \textcolor{black}{we have that} $h(a)\in y$, and thus $a\in h_*(y)$. As $y$ is arbitrary, this means that $h_*[R[x]]\subseteq \beta(a)$. As $h_*$ satisfies CDC$_\Box$ for $\beta(a)$, \textcolor{black}{it follows that} $R[h_*(x)]\subseteq \beta(a)$. Then by \cite[Cor. 5.6]{Al} which says that for any $w\in X_A$, $\Box a\not\in w$ implies that $a\not\in v$ for some $wRv$, \textcolor{black}{we have that} $\Box a\in h_*(x)$ and $h(\Box a)\in x$, namely $x\in \beta(h(\Box a))$. Thus, $\Box_R(\beta(h(a)))\subseteq \beta(h(\Box a))$. \textcolor{black}{This proves that} $h(\Box a)=\Box h(a)$.

\end{proof}

Combining the above results together, we obtain the following dual description of the algebraic CDC for $D^\rightarrow$ and $D^\Box$:

\begin{proposition}
 \label{3.3.11}   
Let $\A=(A,\Box)$ and $\B=(B,\Box)$ be modal Heyting algebras, $h:A\rightarrow B$ be a stable bounded lattice homomorphism, $D^\Box\subseteq A$ and $D^\rightarrow\subseteq A\times A$, the following two are equivalent:

\begin{itemize}
    \item $h$ satisfies CDC for $D^\Box$ and $D^\rightarrow$.
    \item $h_*$ satisfies CDC$_\Box$ for any $\beta(a)$ where $a\in D^\Box$ and satisfies CDC$_\rightarrow$ for any $\beta(a)\setminus \beta(b)$ where $(a,b)\in D^\rightarrow$.
\end{itemize}
\end{proposition}

By Propositions \ref{3.3.6} and \ref{3.3.11}, we now know what stable canonical rules code geometrically, \textcolor{black}{which is a dual analogue of Proposition \ref{3.1.6}.}

\begin{proposition}
\label{3.3.12}
Let $\A=(A,\Box)$ be a finite modal Heyting algebra, $D^\rightarrow \subseteq A\times A$, $D^\Box\subseteq A$, and \textcolor{black}{let} $\B=(B,\Box)$ be a modal Heyting algebra. Then $\B\not\vDash\rho(\A,D^\rightarrow,D^\Box)$ iff there is a surjective stable Priestley morphism $f:X_B\rightarrow X_A $ satisfying CDC$_\Box$ for any $\beta(a)$ where $a\in D^\Box$ and CDC$_\rightarrow$ for any $\beta(a)\setminus \beta(b)$ where $(a,b)\in D^\rightarrow$.   
\end{proposition}

\begin{proof}

According to Proposition \ref{3.1.6}, $\B\not\vDash\rho(\A,D^\rightarrow,D^\Box)$ iff there is a stable bounded lattice embedding $h:A\rightarrow B$ satisfying CDC for $D^\rightarrow$ and $D^\Box$. Then by Proposition \ref{3.3.6} and Proposition \ref{3.3.11}, the result follows immediately.

\end{proof}

By the above proposition, we are then well-justified to write $\rho(\A, D^\rightarrow, D^\Box)$ as $\rho(\A_*,\mathfrak{D}_\rightarrow,\mathfrak{D}_M)$  where $\mathfrak{D}_\rightarrow=\{\beta(a)\setminus\beta(b)\mid (a,b)\in D^\rightarrow\}$, $\mathfrak{D}_M=\{\beta(a)\mid a\in D^\Box\}\footnote{``$M$" stands for ``modal".}$. This notation will become quite useful in the next section when we need to operate on those parameters.

\section{Applications of Stable Canonical Rules for Intuitionistic Modal Logics}

We have introduced stable canonical rules for intuitionistic modal logics and their dual characterization. This section will be devoted to the applications of stable canonical rules in establishing some intrinsic properties of intuitionistic modal logics.

First, using stable canonical rules for bimodal logics, we will give a new and self-contained proof of the Blok-Esakia theorem for intuitionistic modal logics which was first proved by Wolter and Zakharyaschev in \cite{zfon}, and generalize it naturally to intuitionistic modal multi-conclusion consequence relations. Then we will proceed to \textcolor{black}{proving} the Dummett-Lemmon conjecture in this setting by our stable canonical rules for intuitionistic modal logics.  

\subsection{The Blok-Esakia Theorem for Intuitionistic Modal Logics}

It is well-known that every superintuitionistic logic can be embedded into an extension of $\textbf{S4}$ via the G\"odel translation\footnote{\textcolor{black}{See Definition \ref{4.1.16} below.}}. The embedding led to \textcolor{black}{many important} results, one of which is  
the Blok-Esakia theorem. It states that the lattice of superintuitionistic logics is
isomorphic to the lattice of normal extensions of $\textbf{Grz}$, via
the mapping which sends each superintuitionistic logic $L$ to the normal extension
of $\textbf{Grz}$ with the set of all G\"odel translations of formulas in $L$. \textcolor{black}{As its name suggests, the Blok-Esakia theorem was proved independently by Blok using algebraic methods \cite{blok} and by Esakia using duality theory \cite{esath}. It allows us to study superintuitionistic logics by methods \textcolor{black}{and} results from normal modal logics and vice versa\footnote{See \cite{czmod} for more details.}.}
In \cite{zfon} and \cite{zfint}, Wolter and Zakharyaschev extended the G\"odel translation to intuitionistic modal logics, allowing them to embed each intuitionistic modal logic to an extension of the bimodal logic $\textbf{S4}\otimes\textbf{K}$. Besides, they proved a Blok-Esakia theorem for this translation, assuming in the proof that the reader is familiar with the so-called \textit{selection procedure} and \textit{cofinal subreductions} which are developed in \cite{zc1} and are quite involved. 

In this section, we will use stable canonical rules to prove the Blok-Esakia theorem for intuitionistic modal multi-conclusion consequence relations, which generalizes the analogous theorem for intuitionistic modal logics. \textcolor{black}{The proof strategy was adopted from the one in \cite{amaster, an} where Cleani proved the Blok-Esakia theorem for superintuitionistic logics using stable canonical rules for classical (unary) normal modal logics. Compared to the proof in \cite{zfon}, we believe that our alternative proof provides a new perspective and is arguably more self-contained.}

\subsubsection{Stable Canonical Rules for Bimodal Logics}

We first introduce stable canonical rules for bimodal logics. \textcolor{black}{They} are simple and natural generalizations of those in \cite{nrule} from the unimodal case (only one modal operator) to the bimodal case.

\begin{definition}

Let $\A=(A,\Box_I,\Box_M)$ and $\B=(B,\Box_I,\Box_M)$ be bimodal algebras, 
and $h:A\rightarrow B$ be a Boolean homomorphism, $h$ is \textit{stable} if for any $a\in A$, we have $h(\Box_I a)\leq \Box_I h(a)$ and $h(\Box_M a)\leq \Box_M h(a)$.
\end{definition}

\begin{definition}
Let $\A=(A,\Box_I,\Box_M)$, $\B=(B,\Box_I,\Box_M)$ be bimodal algebras, $D^I\subseteq A$ and $D^M\subseteq A$. A Boolean embedding $h:A\rightarrow B$ is said to satisfy

\begin{itemize}
    \item the \textit{closed domain condition} (CDC for short) for $D^I$ if $h(\Box_I a)=\Box_I h(a)$ for any $a\in D^I$.
    \item the \textit{closed domain condition} (CDC for short) for $D^M$ if $h(\Box_M b)=\Box_M h(b)$ for any $b\in D^M$.
\end{itemize}
  
\end{definition}

We can again encode a stable Boolean embedding which satisfies CDC for some parameters in a certain syntactic form in analogy with Definition \ref{3.1.4}. This is exactly what the following definition does. 

\begin{definition}

Let $\A=(A,\Box_I,\Box_M)$ be a finite bimodal algebra, $D^I \subseteq A$ and $D^M\subseteq A$. For each $a\in A$, we introduce a new propositional letter $p_a$ and define the \textit{stable canonical rule} $\mu(\A,D^I,D^M)$ based on $(\A,D^I,D^M)$ as follows:

\vspace{2mm}
\noindent\begin{tabular}{l l l l}
$\Gamma$&=&$\{p_{a\lor b}\leftrightarrow p_a\lor p_b\mid a,b\in A\}\cup\{p_0\leftrightarrow \bot, p_1\leftrightarrow \top\}$\\  
   & & $\cup\{p_{a\land b}\leftrightarrow p_a\land p_b\mid a,b\in A\}\cup\{p_{\neg a}\leftrightarrow \neg p_a\mid a\in A\}$ &  \\
   & & $\cup\{p_{\Box_I a}\rightarrow \Box_I p_a, p_{\Box_M a}\rightarrow \Box_M p_a \mid a\in A\}\cup\{\Box_I p_a\rightarrow p_{\Box_I a}\mid a\in D^I\}$ & \\
   & & $\cup \{\Box_M p_a\rightarrow p_{\Box_M a}\mid a\in D^M\}$
\end{tabular}

\vspace{2mm}
\noindent $\Delta=\{p_a\leftrightarrow p_b\mid a\not=b\in A\}$

\vspace{1mm}
$\mu(\A,D^I,D^M)=\Gamma/\Delta$.

\end{definition}

\textcolor{black}{Using a proof analogous to that
of Proposition \ref{3.1.6}, we have the following proposition.} The proof for the unimodal case can be found in \cite[Thm. 5.4]{nrule}. 

\begin{proposition}
\label{4.1.4}
Let $\A=(A,\Box_I,\Box_M)$ be a finite bimodal algebra, $D^I \subseteq A$, $D^M\subseteq A$, and $\B=(B,\Box_I,\Box_M)$ be a bimodal algebra. Then $\B\not\vDash\mu(\A,D^I,D^M)$ iff there is a stable Boolean embedding $h:A\rightarrow B$ satisfying CDC for $D^I$ and $D^M$.
\end{proposition}

The following result is essentially needed in our proof of the Blok-Esakia theorem. For convenience, let $\textbf{Mix}$\footnote{This was introduced by Wolter and Zakharyaschev in \cite{zfon}. \textcolor{black}{The reason why it is called ``\textbf{Mix}" is that semantically the formula says that $R_I\circ R_M\circ R_I= R_M$, which  means that $R_M$ and $R_I$ are ``mixed" together in some sense. \textcolor{black}{Thinking about modal Esakia spaces, it is the minimal interaction between $R_M$ and $R_I$ that one should expect. See Prop 4.11 below.}}} denote the formula $\Box_I\Box_M\Box_I p\leftrightarrow\Box_M p$.

\begin{proposition}
\label{4.1.5}
For any bimodal multi-conclusion rule $\Gamma/\Delta$, there exist tuples $(\A_1,D^I_1,D^M_1),...$ $,(\A_n, D^I_n,D^M_n)$ such that each $\A_i$ is a finite $(S4\otimes K)\oplus Mix$-algebra, $D^I_i\subseteq A_i$ and $D^M_i\subseteq A_i$, and for each $(S4\otimes K)\oplus Mix$-algebra $\B=(B,\Box_I,\Box_M)$, we have that $\B\not\vDash\Gamma/\Delta$ iff there is a stable embedding $h:A_i\rightarrow B$ satisfying CDC for $D^I_i$ and $D^M_i$.
\end{proposition}

\begin{proof}

Let $\Gamma/\Delta$ be an arbitrary bimodal multi-conclusion rule. If $\Gamma/\Delta\in (\textbf{S4}\otimes\textbf{K})\oplus\textbf{Mix}^R$, then take $n=0$. Suppose $\Gamma/\Delta\not\in (\textbf{S4}\otimes\textbf{K})\oplus\textbf{Mix}^R$, let $\Theta$ be the set of all subformulas of the formulas in $\Gamma\cup\Delta$ which is clearly finite, define $\Theta'=\Theta\cup \{\Box_I\varphi\mid \Box_M\varphi\in \Theta\}$. Clearly $\Theta'$ is finite and closed under subformulas, and for any formula $\Box_M\varphi$, we have that $\Box_M\varphi\in \Theta'$ implies $\Box_I\varphi\in\Theta'$. Assume $|\Theta'|=m$, there are only finitely many pairs $(\A,D^I,D^M)$ satisfying the following two conditions up to isomorphism:

\begin{itemize}
    \item[i)] $\A=(A,\Box_I,\Box_M)$ is a finite $(S4\otimes K)\oplus Mix$-algebra such that $A$ is 
    an $m$-generated Boolean \textcolor{black}{algebra} and $\A\not\vDash \Gamma/\Delta$.
    \item[ii)] $D^I=\{V(\varphi)\mid \Box_I\varphi\in\Theta'\}$ and $D^M=\{V(\varphi)\mid \Box_M\varphi\in \Theta' \}$ where $V$ is a valuation on $\A$ witnessing $\A\not\vDash\Gamma/\Delta$.
\end{itemize}

Let $(\A_1,D^I_1,D^M_1),...,(\A_n, D^I_n,D^M_n)$ be an enumeration of such pairs. For any $(S4\otimes K)\oplus Mix$-algebra $\B$, we prove that $\B\not\vDash\Gamma/\Delta$ iff there is a stable embedding $h:A_i\rightarrow B$ satisfying CDC for $D^I_i$ and $D^M_i$. 

For the right-to-left direction, suppose there is a stable embedding $h:A_i\rightarrow B$ satisfying CDC for $D^I_i$ and $D^M_i$. Define a valuation $V_B$ on $\B$ by $V_B(p)=h(V_i(p))$ for any propositional letter $p$ where $V_i$ is the valuation on $\A_i$ witnessing that $\A_i\not\vDash\Gamma/\Delta$. We then can prove by induction that $V_B(\psi)=h(V_i(\psi))$ for any $\psi\in \Theta'$ in the same way as that in the proof of Proposition \ref{3.1.3}.  

For the left-to-right direction, suppose $\B\not\vDash\Gamma/\Delta$. There exists a valuation $V_B$ on $B$ such that $V_B(\gamma)=1_B$ for any $\gamma\in\Gamma$ and $V_B(\delta)\not=1_B$ for any $\delta\in \Delta$. Let $B'$ be the Boolean subalgebra of $B$ generated by $V_B(\Theta')$. Clearly $|V_B(\Theta')|\leq |\Theta'|=m$. Let $D^I=\{V_B(\varphi)\mid \Box_I\varphi\in \Theta'\}$ and  $D^M=\{V_B(\varphi)\mid \Box_M\varphi\in \Theta'\}$, we define $\Box'_I$ and $\Box'_M$ on $B'$ as follows: $\Box'_I a=\bigvee\{\Box_I b\mid \Box_I b\leq \Box_I a\text{ and }b,\Box_I b\in B'\}$ and $\Box'_M a=\bigvee\{\Box_M b\mid b\leq a\text{ and } b,\Box_M b,\Box_I b\in B'\}$ for any $a\in B'$. We first prove that $(B',\Box'_I,\Box'_M)$ is an $(S4\otimes K)\oplus Mix$-algebra. 

As $\B$ is an $(S4\otimes K)\oplus Mix$-algebra and $\Box'_I$ is defined the same way as that of $\Box'$ in the proof of Proposition \ref{3.1.3}, checking that $(B',\Box'_I, \Box'_M)$ is an $S4\otimes K$-algebra is just a routine.  To see that $(B',\Box'_I,\Box'_M)$ is an $(S4\otimes K)\oplus Mix$-algebra, it then suffices to prove that for any $a\in B'$, $\Box'_I\Box'_M a=\Box'_M a$ and $\Box'_M\Box'_I a= \Box'_M a$\footnote{It is easy to check that for any $S4\otimes K$-algebra, it validates $\textbf{Mix}$ iff it validates $\Box_I\Box_M p\leftrightarrow \Box_M p$ and $\Box_M\Box_I p\leftrightarrow \Box_M p$. In fact, in \cite{zfint}, $\textbf{Mix}$ was simply defined to be $(\Box_I\Box_M p\leftrightarrow \Box_M p)\land (\Box_M\Box_I p\leftrightarrow \Box_M p)$.}.

For any $a\in B'$, as $(B',\Box'_I,\Box'_M)$ is an $S4\otimes K$-algebra, $\Box'_I a\leq a$, and thus $\Box'_M\Box'_I a\leq \Box'_M a$. Suppose $b\leq a$ such that $b,\Box_I b,\Box_M b\in B'$, then $\Box_I b\leq \Box_I a$. Thus $\Box_I b\leq \bigvee\{\Box_I x\mid \Box_I x\leq \Box_I a\text{ where } x,\Box_I x\in B'\}=\Box'_I a$. As $\B$ is an $(S4\otimes K)\oplus Mix$-algebra, $\Box_I b=\Box_I \Box_I b\in B'$ and $\Box_M\Box_I b=\Box_M b\in B'$. Therefore, $\Box_M b\in \{\Box_M x\mid x\leq \Box'_I a\text{ where }x, \Box_I x,\Box_M x\in B'\}$. Thus $\Box_M b\leq \bigvee \{\Box_M x\mid x\leq \Box'_I a\text{ where }x, \Box_I x,\Box_M x\in B'\}=\Box'_M\Box'_I a$. By the definition of $\Box'_M$, we have that $\Box'_M a\leq \Box'_M\Box'_I a$. Therefore, $\Box'_M a=\Box'_M\Box'_I a$. For any $a\in B'$, as $(B',\Box'_I,\Box'_M)$ is an $S4\otimes K$-algebra, $\Box'_I\Box'_M a\leq \Box'_M a$. Suppose $b\leq a$ such that $b,\Box_I b,\Box_M b\in B'$, then $\Box_M b\leq \Box'_M a$. As $\B$ is an $(S4\otimes K)\oplus Mix$-algebra, $\Box_I\Box_M b\leq \Box_I \Box'_M a$ and $\Box_I\Box_M b=\Box_M b\in B'$. Thus $\Box_M b=\Box_I\Box_M b\leq \bigvee\{\Box_I x\mid \Box_I x\leq \Box_I\Box'_M a\text{ where }x,\Box_I x\in B'\}=\Box'_I\Box'_M a$. By the definition of $\Box'_M$, we have that $\Box'_M a\leq \Box'_I\Box'_M a$, and thus $\Box'_M a=\Box'_I\Box'_M a$. Therefore, $(B',\Box'_I,\Box'_M)$ is an $(S4\otimes K)\oplus Mix$-algebra. 

Let $h:(B',\Box'_I,\Box'_M)\rightarrow (B,\Box_I,\Box_M)$ be the inclusion map, $h$ is clearly an embedding as $B'$ is a Boolean subalgebra of $B$. As $\Box'_I a\leq \Box_I a$ and $\Box'_M a\leq \Box_M a$ for any $a\in B'$ by definition, $h$ is stable. Then we check that $h$ satisfies CDC for $D^I$ and $D^M$. For any $a\in D^I$, $a=V_B(\varphi)$ for some $\Box_I \varphi\in \Theta'$. Thus $V_B(\Box_I \varphi)=\Box_I a\in B'$. By the definition of $\Box'_I$, we then have that $\Box'_I a=\Box_I a$. For any $b\in D^M$, $b=V_B(\psi)$ for some $\Box_M\psi\in \Theta'$. Thus $\Box_I\psi\in \Theta'$, and $V_B(\Box_I \psi)=\Box_I V_B(\psi)=\Box_I b\in B'$. As $V_B(\Box_M \psi)=\Box_M V_B(\psi)=\Box_M b\in B'$, by the definition of $\Box'_M$, $\Box'_M b=\Box_M b$. This proves that $h$ satisfies CDC for $D^I$ and $D^M$. 

Let $V'$ be the valuation $V_B$ restricted to $B'$, we then prove that for any $\varphi\in \Theta'$, $V'(\varphi)=V_B(\varphi)$ by induction on $\varphi$. We only consider the modal cases as others are trivial:

If $\varphi=\Box_I \psi$, as $\Box_I\psi\in \Theta'$, we have that $\psi\in\Theta'$. Thus $V_B(\Box_I\psi)=\Box_I V_B(\psi)\in B'$ and $V_B(\psi)\in B'$. 

\vspace{1mm}
\textcolor{black}{\begin{tabular}{l l l l}
$V'(\Box_I\psi)$ &=&$\Box'_I V'(\psi)$\\  
   &=& $\Box'_I V_B(\psi)$ & (IH) \\
   &=& $\Box_I V_B(\psi)$ & (By the definition of $\Box'_I$) \\
   &=& $V_B(\Box_I\psi).$   
\end{tabular}}

\vspace{1mm}
If $\varphi=\Box_M\psi$, as $\Box_M\psi\in \Theta'$, we have that $\Box_I\psi,\psi\in\Theta'$. Thus $V_B(\Box_M\psi)=\Box_M V_B(\psi)\in B'$, $V_B(\Box_I\psi)=\Box_I V_B(\psi)\in B'$ and $V_B(\psi)\in B'$. 

\vspace{1mm}
\textcolor{black}{\begin{tabular}{l l l l}
$V'(\Box_M\psi)$ &=&$\Box'_M V'(\psi)$\\  
   &=& $\Box'_M V_B(\psi)$ & (IH) \\
   &=& $\Box_M V_B(\psi)$ & (By the definition of $\Box'_M$) \\
   &=& $V_B(\Box_M\psi).$   
\end{tabular}}

\vspace{1mm}
Since $V_B$ is a valuation which refutes $\Gamma/\Delta$ on $\B$, $V'$ is a valuation on $(B',\Box_I,\Box_M)$ which refutes $\Gamma/\Delta$. As for any $\varphi\in\Theta'$, $V'(\varphi)=V_B(\varphi)$, we have that $D^I=\{V_B(\varphi)\mid \Box_I\varphi\in \Theta'\}=\{V'(\varphi)\mid \Box_I\varphi\in \Theta'\}$ and  $D^M=\{V'(\varphi)\mid \Box_M\varphi\in \Theta'\}$. As $B'$ is generated by $V_B(\Theta')$ whose cardinality is no larger than $m$, $(B', \Box'_I,\Box'_M, D^I, D^M)$ must be one of $(\A_1,D^I_1,D^M_1),...,(\A_n, D^I_n,D^M_n)$. Since $h:(B',\Box'_I,\Box'_M)\rightarrow (B,\Box_I,\Box_M)$ is a stable embedding  satisfying CDC for $D^I$ and $D^M$, we get what we want.

\end{proof}

The above proposition is not a trivial generalization of \cite[Thm. 5.1]{nrule} to bimodal cases as we require that every $A_i$ should validate $\textbf{Mix}$. This is quite crucial in the proof of Lemma \ref{4.1.19} below. In fact, the above proof gives a concrete example of how to do filtrations in polymodal cases when there are interactions between different operators. It is far from clear how to do so in general, and an answer to it may shed light on some open problems about the finite model property of polymodal logics. 

The above two propositions allow us to get the following result, which is analogous to Theorem \ref{3.1.7}.

\begin{proposition}
\label{4.1.6}
For any bimodal multi-conclusion rule $\Gamma/\Delta$, there exist tuples $(\A_1,D^I_1,D^M_1),...$ $,(\A_n, D^I_n,D^M_n)$ such that each $\A_i$ is a finite $(S4\otimes K)\oplus Mix$-algebra, $D^I_i\subseteq A_i$ and $D^M_i\subseteq A_i$, and for each $(S4\otimes K)\oplus Mix$-algebra $\B=(B,\Box_I,\Box_M)$, we have that $\B\vDash\Gamma/\Delta$ iff $\B\vDash\mu(\A_1,D^I_1,D^M_1),...,$ $\mu(\A_n, D^I_n,D^M_n)$.

\end{proposition}

Then we introduce the category of bimodal spaces, which is dually equivalent to the category of bimodal algebras. \textcolor{black}{Similarly to Section 3,} this duality allows us to \textcolor{black}{obtain a geometric characterization of} stable canonical rules for bimodal logics, which \textcolor{black}{makes our construction in the proof of Lemma \ref{4.1.19} possible.}

\begin{definition}

A \textit{bimodal space} is a triple $(X,R_I,R_M)$ where $(X,R_I)$ and $(X,R_M)$ are modal spaces.

Let $\mathfrak{X},\mathfrak{Y}$ be bimodal spaces, a map  $\textcolor{black}{f}:X\rightarrow Y$ is \textit{stable} if for $R\in\{R_I, R_M\}$ and any $x,y\in X$, $xR y$ implies $f(x)R f(y)$. Furthermore, $\textcolor{black}{f}$ is a \textit{bounded morphism} if $\textcolor{black}{f}$ is stable,  and for any $x,y\in X$, $f(x)R y$ implies that there is a $z\in X$ such that $xR z$ and $f(z)=y$ where $R\in\{R_I, R_M\}$.
\end{definition}

A \textit{valuation} on a bimodal space $\mathfrak{X}$ is a map $V:Prop\rightarrow Clop(X)$ which can be extended recursively to a map from $Form_{bi}$ to $Clop(X)$ in the usual way. For any bimodal logic $L$, a bimodal space $\mathfrak{X}$ is \textit{an $L$-modal space} if it validates $L$. In particular, an $S4\otimes K$-modal space $\mathfrak{X}=(X,R_I,R_M)$ is a bimodal space where $(X,R_I)$ is an $S4$-space and $(X,R_M)$ is a modal space.  

Let $\mathbf{BMS}$ be the category of bimodal spaces with continuous bounded morphisms, and $\mathbf{BMA}$ be the category of bimodal algebras with their homomorphisms, the functor $(-)_*:\mathbf{BMA}\rightarrow \mathbf{BMS}$ and $(-)^*:\mathbf{BMS}\rightarrow \mathbf{BMA}$ that establish the duality are constructed as follows. For a bimodal algebra $\A=(A,\Box_I,\Box_M)$, let $\A_*=(A_*,R_I,R_M)$ where $A_*$ is the Stone space of $A$ and $xR_Iy$ iff $\forall\Box_I a\in A(\Box_I a\in x \Rightarrow a\in y)$, 
and similarly we define $R_M$. For a bimodal space $\mathfrak{X}=(X,R_I,R_M)$, let $\mathfrak{X}^*=(X^*,\Box_{R_I},\Box_{R_M})$ where $X^*$ is the Boolean algebra of clopen sets of $X$ and $\Box_{R_I}U=\{x\in X\mid R_I[x]\subseteq U\}$. Similarly we define $\Box_{R_M}$. The duals of maps are the same as those in Esakia duality (namely taking inverse images). 

\begin{theorem}
\label{4.1.8}
$\mathbf{BMA}$ is dually equivalent to $\mathbf{BMS}$, which is witnessed by $(-)^*$ and $(-)_*$. In particular, for any bimodal algebra $\A$, $\A\cong (\A_*)^*$ witnessed by $\beta$ where $\beta(a)=\{x\in A_*\mid a\in x\}$, and for any bimodal space $\mathfrak{X}$, $\mathfrak{X}\cong (\mathfrak{X}^*)_*$ witnessed by $\epsilon$ where $\epsilon(x)=\{U\in X^*\mid x\in U\}$.
\end{theorem}

When restricted to those validating $S4\otimes K$, the above theorem gives us the dual equivalence between the category of $S4\otimes K$-modal spaces with continuous bounded morphisms and the category of $S4\otimes K$-algebras
with their homomorphisms.

Using Theorem \ref{4.1.8}, we can now easily exploit the geometric intuitions about the stable canonical rules for bimodal logics. The proof of the one for unimodal logics can be found in \cite[Thm. 5.4]{nrule}.

\begin{proposition}
\label{4.1.9}

    Let $\A=(A,\Box_I,\Box_M)$ be a finite bimodal algebra, $D^I \subseteq A$ and $D^M\subseteq A$,  and let $\B=(B,\Box_I,\Box_M)$ be a bimodal algebra, then $\B\not\vDash\mu(\A,D^I,D^M)$ iff there is a continuous stable surjection $f:X_B\rightarrow X_A $ satisfying CDC$_\Box$ for any $\beta(a)$ where $a\in D^I$ and for any $\beta(b)$ where $b\in D^M$\footnote{ CDC$_\Box$ is given in Definition \ref{3.3.9}. When there is no ambiguity (in particular, for bimodal logics), we may just write CDC instead of CDC$_\Box$. }. 
\end{proposition}

Because of the above proposition, we are justified to write $\mu(\A,D^I,D^M)$ as $\mu(\A_*,\mathfrak{D}_I,\mathfrak{D}_M)$ where $\mathfrak{D}_I=\{\beta(a)\mid a\in D^I\}$ and $\mathfrak{D}_M=\{\beta(b)\mid b\in D^M\}$. Besides, since $\A_*$ is finite, every subset of $\A_*$ is clopen and thus of the form $\beta(a)$ for some $a\in \A$ by Stone duality. Therefore, we are also justified to write $\mu (\mathfrak{Y}, \mathfrak{D}_I,\mathfrak{D}_M)$ where $\mathfrak{Y}$ is a finite bimodal space, and $\mathfrak{D}_I$, $\mathfrak{D}_M$ are sets of subsets of $\mathfrak{Y}$.

\subsubsection{The G\"odel Translation for Intutionistic Modal Logics and Related Constructions}

Now we introduce the operations from modal Esakia spaces to $S4\otimes K$-modal spaces and vice versa. They are in essence the topological versions of those defined in \cite{zfint} for general frames, and have an intimate relation with the G\"odel translation.

\textcolor{black}{For any $S4\otimes K$-modal space $\mathfrak{Y}=(Y,R_I,R_M,\mathcal{O})$, we write $x\backsim y$ iff $xR_Iy$ and $yR_Iy$, and then define $\rho: Y\rightarrow \mathcal{P}(Y)$ by $\rho(x)=\{y\in Y\mid x\backsim y\}.$} \textcolor{black}{For convenience, we may also write $[x]$ for $\rho(x)$.}

\begin{definition}
\label{4.1.10}
\quad
\begin{itemize}
    \item For any modal Esakia space $\mathcal{X}=(X,\leq,R,\mathcal{O})$ where $\mathcal{O}$ is the topology, we set $\sigma(\mathcal{X})=(X,R_I,R_M,\mathcal{O})$ where $R_I=\leq$ and $R_M=R$.
    \item For any $S4\otimes K$-modal space $\mathfrak{Y}=(Y,R_I,R_M,\mathcal{O})$, we set $\rho(\mathfrak{Y})=(\rho[Y],\leq,[R_I\circ R_M\circ R_I],\rho[\mathcal{O}])$ where $\rho(x)\leq\rho(y)$ iff $xR_I y$, $\rho(x)[R_I\circ R_M\circ R_I]\rho(y)$ iff $xR_I\circ R_M\circ R_I y$\footnote{It is easy to check that they are well defined.} and $\rho(\mathcal{O})$ is the quotient topology.  
\end{itemize}

\end{definition}

First, modal Esakia spaces have \textcolor{black}{the following intrinsic property}: 

\begin{proposition}
\label{4.1.11}
For any modal Esakia space $\mathcal{X}=(X,\leq,R,\mathcal{O})$, $\leq\circ R= R\circ \leq =R$, i.e., $R=\leq \circ R\circ \leq$. 
\end{proposition}

\begin{proof}
By the duality given in Theorem \ref{3.3.3}, it suffices to prove that for any modal Heyting algebra $\A=(A,\Box)$, the relation holds on its dual space $\A_*$. Namely, we need to prove that $\subseteq\circ R\circ\subseteq=R$. Let $x,y\in A_*$ where $x\subseteq \circ R\circ \subseteq y$, then there exists $z,z'\in A_*$ such that $x\subseteq z$, $zRz'$ and $z'\subseteq y$. For any $\Box a\in A$, if $\Box a\in x$, then $\Box a\in z$. As $zRz'$, \textcolor{black}{it follows that} $a\in z'$ and thus $a\in y$. Therefore, $xRy$. We get $\subseteq \circ R\circ \subseteq =R$ (the other direction is obvious).    
\end{proof}

Let $\mathfrak{Y}=(Y,R_I,R_M,\mathcal{O})$ be an $S4\otimes K$-modal space, we say $U\subseteq Y$ is an \textit{upset} if it is an upset w.r.t the quasi-order $R_I$, i.e., $U=\{y\in Y\mid xR_I y \text{ for some } x\in U\}$. We can now check that the operations in Definition \ref{4.1.10} transform modal Esakia spaces to bimodal spaces and vice versa.

\begin{proposition}
\label{4.1.12} 
The following hold:

\begin{itemize}
    \item[1)] For any modal Esakia space $\mathcal{X}$, we have that $\sigma(\mathcal{X})$ is an $S4\otimes K$-modal space.
    \item[2)] For any $S4\otimes K$-modal space $\mathfrak{Y}$, we have that $\rho(\mathfrak{Y})$ is a modal Esakia space. 
\end{itemize}

\end{proposition}

\begin{proof}

For 1), as $\mathcal{X}=(X,\leq,R,\mathcal{O})$ is a modal Esakia space, $(X,\leq,\mathcal{O})$ is an Esakia space. If we omit $R$, $\sigma(\mathcal{X})$ is defined exactly the same as that in \cite[Def. 2.43]{amaster}, which is well-known to be an S4-modal space. Thus we only need to check that $(X,R_M,\mathcal{O})$ is a modal space. For this, it suffices to prove that if $U\subseteq X$ is clopen, then $\Box_{R_M}U=\Box_R U$ is also clopen. Let $U$ be an arbitrary clopen subset of $X$, as $(X,\leq, \mathcal{O})$ is an Esakia space, $U=\bigcup_{1\leq i\leq n}(U_i\setminus V_i)$ where $U_i,V_i$'s are clopen upsets. $\Box_R U=\Box_R (\bigcup_{1\leq i\leq n}(U_i\setminus V_i))=\Box_R((U_1\cap \Bar{V_1})\cup...\cup (U_n\cap \Bar{V_n}))$. By distributivity, $\Box_R U=\Box_R \bigcap_{1\leq i\leq k}(U'_i\cup \Bar{V}'_i)$ for some $k$ where $U'_i, V'_i$'s are clopen upsets. Then $\Box_R U=\bigcap_{1\leq i\leq k} \Box_R(U'_i\cup \Bar{V}'_i)$. Now, for any $1\leq i\leq k$, as $V'_i\rightarrow U'_i\subseteq U'_i\cup \Bar{V}'_i$, we have that $\Box_R (V'_i\rightarrow U'_i)\subseteq\Box_R (U'_i\cup \Bar{V}'_i)$. Suppose $x\in \Box_R (U'_i\cup \Bar{V}'_i)$, then $R[x]\subseteq U'_i\cup \Bar{V}'_i$. For any $xRy$ and $y\leq z$, as $R\circ \leq=R$ 
 by Proposition \ref{4.1.11}, $xRz$, and thus $z\in U'_i\cup \Bar{V}'_i$. Therefore, $y\in V'_i\rightarrow U'_i$, and thus $R[x]\subseteq V'_i\rightarrow U'_i$, namely $x\in \Box_R (V'_i\rightarrow U'_i)$. As $x\in \Box_R (U'_i\cup \Bar{V}'_i)$ is arbitrary, we have that $\Box_R (U'_i\cup \Bar{V}'_i)\subseteq \Box_R (V'_i\rightarrow U'_i)$, and thus $\Box_R (U'_i\cup \Bar{V}'_i)=\Box_R (V'_i\rightarrow U'_i)$. As $V'_i$ and $U'_i$ are clopen upsets and $(X,\leq,\mathcal{O})$ is an Esakia space, $V'_i\rightarrow U'_i$ is a clopen upset. As $\mathcal{X}$ is a modal Esakia space, by item 1 in Definition \ref{3.3.1}, $\Box_R (U'_i\cup \Bar{V}'_i)=\Box_R (V'_i\rightarrow U'_i)$ is clopen. $\Box_R U= \bigcap_{1\leq i\leq k}\Box_R(U'_i\cup \Bar{V}'_i)$ is also clopen. 

\vspace{1mm}
For 2), let $\mathfrak{Y}=(Y,R_I,R_M,\mathcal{O})$ be an $S4\otimes K$-modal space. If we omit $R_M$, $\rho(\mathfrak{Y})$
is defined exactly the same as that in \cite[Def. 2.43]{amaster}, which is well-known to be an Esakia space. Thus it suffices to check that for any clopen upset $U$ of $\rho(\mathfrak{Y})$, $\Box_{[R_I\circ R_M\circ R_I]}U$ is a clopen upset, and for any $x\in Y$, $[R_I\circ R_M\circ R_I][\rho(x)]$ is a closed upset. 

Let $U$ be an arbitrary clopen upset of $\rho(\mathfrak{Y})$, as $\rho(\mathcal{O})$ is the quotient topology, $\rho^{-1}(U)$ is clopen. By the definition of $\leq$, $\rho^{-1}(U)$ is an upset as $U$ is an upset. Clearly, $\rho^{-1}(\Box_{[R_I\circ R_M\circ R_I]}U)=\Box_{R_I}\Box_{R_M}\Box_{R_I}\rho^{-1}(U)$. As $\rho^{-1}(U)$ is a clopen upset and $\mathfrak{Y}$ is an $S4\oplus K$-modal space, $\rho^{-1}(\Box_{[R_I\circ R_M\circ R_I]}U)=\Box_{R_I}\Box_{R_M}\Box_{R_I}\rho^{-1}(U)$ is clopen. As $R_I$ is transitive, obviously $\Box_{R_I}\Box_{R_M}\Box_{R_I}\rho^{-1}(U)$ is an upset. Therefore, $\rho^{-1}(\Box_{[R_I\circ R_M\circ R_I]}U)$ is a clopen upset, so is $\Box_{[R_I\circ R_M\circ R_I]}U$. 

Let $x\in Y$ be arbitrary, clearly, $[R_I\circ R_M\circ R_I][\rho[x]]$ is an upset. By the definition of $\leq$, $\bigcup[R_I\circ R_M\circ R_I][\rho[x]]=(R_I\circ R_M\circ R_I)[x]=R_I[R_M[R_I[x]]]$. By the definition of modal spaces, $R_I[x]$ is closed. As $\mathfrak{Y}$ is an $S4\otimes K$-modal space, $R_M[R_I[x]]$ is also closed\footnote{It is known that for any modal space $\mathfrak{X}$, if $U$ is closed, then $R[U]$ is also closed: suppose $x\not\in R[U]$, then $\Diamond_R \{x\}\cap U=\emptyset$. $\Diamond_R\{x\}=\Diamond_R \bigcap \{Y\mid x\in Y\in Clop(\mathfrak{X})\}$. By Esakia Lemma, $\Diamond_R \bigcap \{Y\mid x\in Y\in Clop(\mathfrak{X})\}=\bigcap\{\Diamond_R Y\mid x\in Y\in Clop(\mathfrak{X})\}$. Thus $\bigcap\{\Diamond_R Y\mid x\in Y\in Clop(\mathfrak{X})\}\cap U=\emptyset$. By compactness, $U\cap \Diamond_R Y_1\cap...\cap\Diamond_R Y_n=\emptyset$ for some $x\in Y_1,...,Y_n\in Clop(\mathfrak{X})$. As $\Diamond_R(Y_1\cap...\cap Y_n)\subseteq \Diamond_R Y_1\cap...\cap \Diamond_R Y_n$, $U\cap\Diamond_R(Y_1\cap...\cap Y_n)=\emptyset$. As $x\in Y_1\cap...\cap Y_n$ is clopen and $Y_1\cap...\cap Y_n\cap R[U]=\emptyset$, $R[U]$ is closed.}, so is $R_I[R_M[R_I[x]]]$. Therefore, $[R_I\circ R_M\circ R_I][\rho(x)]$ is a closed upset.

\end{proof}

Now, by duality, we can also give the dual of $\sigma$ and $\rho$ as follows. Let $\A$ be an arbitrary modal Heyting algebra, $\A_*$ is a modal Esakia space by the duality given in Theorem \ref{3.3.3}. $\sigma(\A_*)$ is an $S4\otimes K$-modal space by Proposition \ref{4.1.12}. By Theorem \ref{4.1.8},  $\sigma(\A_*)^*$ is an $S4\otimes K$-algebra. We set $\sigma(\A)=\sigma(\A_*)^*$ for any modal Heyting algebra $\A$. $\sigma$ is then a map from the class of modal Heyting algebras to the class of $S4\otimes K$-algebras. Similarly, we define $\rho(\B)=\rho(\B_*)^*$ for any $S4\otimes K$-algebra. $\rho$ is then a map from the class of $S4\otimes K$-modal algebras to the class of modal Heyting algebras. 

Then we can get the following proposition  which says that the algebraic $\sigma,\rho$ and the geometric $\sigma,\rho$ are dual to each other respectively.

\begin{proposition} 
\label{4.1.13}
The following hold:

\begin{enumerate}
    \item [1)] For any modal Heyting algebra $\A$, $(\sigma(\A))_*\cong \sigma(\A_*)$. Consequently, $\sigma((\mathcal{X})^*)\cong (\sigma(\mathcal{X}))^*$ for any modal Esakia space $\mathcal{X}$.
    \item [2)] For any $S4\otimes K$-algebra $\B$, $(\rho(\B))_*\cong \rho (\B_*)$. Consequently, $(\rho(\mathfrak{X}))^*\cong \rho (\mathfrak{X}^*)$  for any  $S4\otimes K$-modal space $\mathfrak{X}$. 
\end{enumerate}

\end{proposition}

\begin{proof}

Simply by the definitions of $\sigma$ and $\rho$ and the duality.
    
\end{proof}

Besides, since $R=\leq\circ R\circ\leq$ for any modal Esakia space $\mathcal{X}=(X,\leq, R,\mathcal{O})$, the following proposition is \textcolor{black}{straightforward}.

\begin{proposition}
\label{4.1.14}
For any modal Esakia space $\mathcal{X}=(X,\leq, R,\mathcal{O})$, $\rho(\sigma(\mathcal{X}))\cong \mathcal{X}$. Consequently, for any modal Heyting algebra $\A$, $\rho(\sigma(\A))\cong \A$.    
\end{proposition}

The above proposition is the analogue to the first half of \cite[Prop. 2.45]{amaster}. For the second half of \cite[Prop. 2.45]{amaster}, the analogue does not hold for $S4\otimes K$-algebras in general. We need an extra assumption.

\begin{proposition}
\label{4.1.15}
If $\B=(B,\Box_I,\Box_M)$ is an $(S4\otimes K)\oplus Mix$-algebra, then there is a homomorphic embedding of $\sigma(\rho(\B))$ into $\B$ (usually denoted as $\sigma(\rho(\B))\rightarrowtail \B$).

\end{proposition}

\begin{proof}

Let $\B$ be an $S4\otimes K$-algebra which validates $\Box_I\Box_M\Box_I p\leftrightarrow \Box_M p$. By duality given in Theorem \ref{4.1.8}, $\B_*$ is an $S4\otimes K$-modal space on which $R_I\circ R_M\circ R_I=R_M$, and it suffices to show that there is a surjective continuous bounded morphism from $\B_*$ onto $(\sigma(\rho(\B)))_*$. By Proposition \ref{4.1.14}, $(\sigma(\rho(\B)))_*\cong\sigma(\rho(\B_*))$. We check that $x\mapsto [x]$ is a surjective continuous bounded morphism from $\mathfrak{X}=(X,R_I,R_M,\mathcal{O})$ onto $\sigma(\rho(\mathfrak{X}))$ for any $S4\otimes K$-modal space $\mathfrak{X}$ on which $R_I\circ R_M\circ R_I=R_M$.

Clearly, the map is surjective and continuous (the topology is the quotient topology). If $xR_Iy$, then by definition $[x]\leq [y]$; If $[x]\leq [y]$, then by definition $xR_Iy$. Suppose $xR_M y$, then as $R_I$ is reflexive, $xR_I\circ R_M\circ R_I y$ and thus $[x][R_I\circ R_M\circ R_I][y]$. Suppose $[x][R_I\circ R_M\circ R_I][y]$, then $x R_I\circ R_M\circ R_I y$ by definition. As $R_I\circ R_M\circ R_I =R_M $, $xR_M y$. This proves that the map is a bounded morphism.

\end{proof}

Now, we introduce the \textit{G\"odel translation} for intuitionistic modal logics, which is given in \cite{zfint}. 

\begin{definition}
\label{4.1.16}
The \textit{G\"odel translation} for intuitionistic modal logics $t: Form_{i\Box}\rightarrow Form_{bi}$ is recursively defined as follows:

\begin{itemize}
    \item $t(p)=\Box_I p \text{ where } p \text{ is a propositional variable}$
    \item $t(\bot)=\Box_I\bot$
    \item  $t(\top)=\Box_I \top$
    \item  $t(\varphi\rightarrow\psi)=\Box_I (t(\varphi)\rightarrow t(\psi))$
    \item $t(\varphi\land\psi)=\Box_I (t(\varphi)\land t(\psi))$
    \item $t(\varphi\lor\psi)=\Box_I (t(\varphi)\lor t(\psi))$
    \item $t(\Box\varphi)=\Box_I\Box_M t(\varphi)$
\end{itemize}

\end{definition}

It turns out that for the following results to hold, it makes no difference if we define $t(\varphi\land\psi)= t(\varphi)\land t(\psi)$ and $t(\varphi\lor\psi)= t(\varphi)\lor t(\psi)$ instead. When restricted to the set of intuitionistic propositional formulas, $t$ is just the well-known G\"odel translation. For any intuitionistic modal multi-conclusion rule $\Gamma/\Delta$, we write $t(\Gamma/\Delta)$ for $\{t(\gamma)\mid \gamma\in\Gamma\}/\{t(\delta)\mid \delta\in\Delta\}$.

The following proposition is the counterpart to \cite[Lem. 3.13]{can}.

\begin{proposition}
\label{4.1.17}
If $\A$ is an $S4\otimes K$-algebra, then $\A\vDash t(\Gamma/\Delta)$ iff $\rho(\A)\vDash \Gamma/\Delta$.

\end{proposition}

\begin{proof}

By Theorem \ref{4.1.8}, $\A\vDash\Gamma/\Delta$ iff $\A_*\vDash\Gamma/\Delta$ and $\rho(\A)\vDash\Gamma/\Delta$ iff $\rho(\A_*)=\rho(\A)_*\vDash\Gamma/\Delta$. Thus, it suffices to prove that for any $S4\otimes K$-modal space $\mathfrak{X}=(X,R_I,R_M,\mathcal{O})$,  $(X,R_I,R_M,\mathcal{O})\vDash t(\Gamma/\Delta)$ iff $\rho(X,R_I,R_M,\mathcal{O})\vDash \Gamma/\Delta$.

For the left-to-right direction, suppose $\rho(X,R_I,R_M,\mathcal{O})\not\vDash \Gamma/\Delta$, there is a valuation $V$ on $\rho(X,R_I,R_M,\mathcal{O})$ such that $V(\varphi)=\rho[X]$ for any $\varphi\in \Gamma$ and $V(\psi)\not=\rho[X]$ for any $\psi\in \Delta$. Since $V(p)$ is a clopen upset in $\rho(X,R_I,R_M,\mathcal{O})$ for any propositional variable $p$, $\rho^{-1}(V(p))$ is a clopen upset in $(X,R_I,R_M,\mathcal{O})$. We define a valuation $V'$ on $(X,R_I,R_M,\mathcal{O})$ by setting $V'(p)=\rho^{-1}(V(p))$. Then we prove by induction that for any formula $\varphi$, $V'(t(\varphi))=\rho^{-1}(V(\varphi))$ (note that $V(\varphi)$ is a clopen upset for any formula $\varphi$):

If $\varphi=p$ for some propositional variable $p$, $V'(t(p))=V'(\Box_I p)=\Box_I V'(p)=\Box_I \rho^{-1}(V(p))$. As $\rho^{-1}(V(p))$ is an upset and $R_I$ is reflexive, $\Box_I \rho^{-1}(V(p))=\rho^{-1}(V(p))$. Thus $V'(t(p))=\rho^{-1}(V(p))$.

If $\varphi=\bot$, $V'(t(\bot))=V'(\Box_I\bot)=\Box_I V'(\bot)=\emptyset=\rho^{-1}(V(\bot))$; if $\varphi=\top$, $V'(t(\top))=V'(\Box_I \top)=\Box_I V'(\top)= X= \rho^{-1}(\rho[X])=\rho^{-1}(V(\top))$.

If $\varphi=\psi\land\theta$,

\vspace{1mm}
\textcolor{black}{\begin{tabular}{l l l l}
$V'(t(\psi\land\theta))$ &=&$V'(\Box_I(t(\psi)\land t(\theta)))$\\  
   &=& $\Box_I V'(t(\psi)\land t(\theta))$ &  \\
   &=& $\Box_I (V'(t(\psi))\cap V'(t(\theta)))$ &  \\
   &=& $\Box_I (\rho^{-1}(V(\psi))\cap \rho^{-1}(V(\theta)))$ & (IH)  \\
   &=& $\Box_I \rho^{-1}(V(\psi)\cap V(\theta))$ \\
   &=& $\Box_I \rho^{-1}(V(\psi\land\theta))$ \\
   &=& $\rho^{-1}(V(\psi\land\theta))$.
\end{tabular}}

Note that the last equality holds as $\rho^{-1}(V(\psi\land\theta))$ is an upset and $R_I$ is reflexive. The case when $\varphi=\psi\lor\theta$ is similar.

If $\varphi=\psi\rightarrow \theta$, $V'(t(\psi\rightarrow\theta))=V'(\Box_I(t(\psi)\rightarrow t(\theta)))=\Box_I (V'(t(\psi))\rightarrow V'(t(\theta)))
=\Box_I (X\setminus V'(t(\psi))\cup V'(t(\theta)) )$. By IH, $V'(t(\psi))=\rho^{-1}(V(\psi))$ and $V'(t(\theta))=\rho^{-1}(V(\theta))$. Thus $V'(t(\psi\rightarrow \theta))=\Box_I (X\setminus \rho^{-1}(V(\psi))\cup \rho^{-1}(V(\theta)))= \Box_I (\rho^{-1}(\rho[X]\setminus V(\psi)\cup V(\theta)))$ (note $V(\psi)$ and $V(\theta)$ are upsets).

Suppose $x\in \rho^{-1}(\rho[X]\setminus {\downarrow} (V(\psi)\setminus V(\theta)))$, then $\rho(x)\in \rho[X]\setminus {\downarrow} (V(\psi)\setminus V(\theta))$. For any $xR_I y$, by definition $\rho(x)\leq \rho(y)$, and thus $\rho(y)\not\in V(\psi)\setminus V(\theta)$. As $x R_I y$ is arbitrary, we have that $x\in \Box_I(\rho^{-1}(\rho[X]\setminus V(\psi)\cup V(\theta)))$, and thus $\rho^{-1}(\rho[X]\setminus {\downarrow} (V(\psi)\setminus V(\theta)))\subseteq \Box_I(\rho^{-1}(\rho[X]\setminus V(\psi)\cup V(\theta)))$. Suppose $x\in \Box_I(\rho^{-1}(\rho[X]\setminus V(\psi)\cup V(\theta)))$ while $x\not\in \rho^{-1}(\rho[X]\setminus {\downarrow} (V(\psi)\setminus V(\theta)))$. Then $\rho(x)\not\in \rho[X]\setminus {\downarrow} (V(\psi)\setminus V(\theta)))$, there exists $\rho(x)\leq \rho(y)$ such that $\rho(y)\in V(\psi)\setminus V(\theta)$. Thus $y\not\in \rho^{-1}(\rho[X]\setminus V(\psi)\cup V(\theta))$, and by definition $xR_I y$. This contradicts the assumption that $x\in \Box_I(\rho^{-1}(\rho[X]\setminus V(\psi)\cup V(\theta)))$. Therefore, if $x\in \Box_I(\rho^{-1}(\rho[X]\setminus V(\psi)\cup V(\theta)))$, then $x\in \rho^{-1}(\rho[X]\setminus {\downarrow} (V(\psi)\setminus V(\theta)))$. Thus $\Box_I(\rho^{-1}(\rho[X]\setminus V(\psi)\cup V(\theta)))\subseteq \rho^{-1}(\rho[X]\setminus {\downarrow} (V(\psi)\setminus V(\theta)))$, and $\Box_I(\rho^{-1}(\rho[X]\setminus V(\psi)\cup V(\theta)))=\rho^{-1}(\rho[X]\setminus {\downarrow} (V(\psi)\setminus V(\theta)))$. As $\rho^{-1}(\rho[X]\setminus {\downarrow} (V(\psi)\setminus V(\theta)))=\rho^{-1}(V(\psi\rightarrow \theta))$, this proves that $V'(t(\psi\rightarrow\theta))=\rho^{-1}(V(\psi\rightarrow \theta))$.

If $\varphi=\Box\psi$, $V'(t(\Box\psi))=V'(\Box_I\Box_Mt(\psi))=\Box_I\Box_M V'(t(\psi))=\text{(By IH) }\Box_I\Box_M \rho^{-1}(V(\psi))$. And $\rho^{-1}(V(\Box\psi))=\rho^{-1}(\Box_{[R_I\circ R_M\circ R_I]}V(\psi))=\Box_I\circ \Box_M\circ \Box_I \rho^{-1}(V(\psi))$. As $V(\psi)$ is an upset, so is $\rho^{-1}(V(\psi))$. And since $R_I$ is reflexive, $\Box_I\rho^{-1}(V(\psi))=\rho^{-1}(V(\psi))$, and thus $\Box_I\circ \Box_M\circ \Box_I \rho^{-1}(V(\psi))=\Box_I \circ \Box_M \rho^{-1}(V(\psi))$. We have that $V'(t(\Box\psi))=\rho^{-1}(V(\Box\psi))$. This finishes the induction.

Since $V'(t(\varphi))=\rho^{-1}(V(\varphi))$, $V'(t(\varphi))=X$ for any $\varphi\in \Gamma$ as $V(\varphi)=\rho[X]$; $V'(t(\psi))\not=X$ for any $\psi\in\Delta$ as $V(\psi)\not=\rho[X]$. Thus $V'$ is a valuation on $(X,R_I,R_M,\mathcal{O})$ which refutes $t(\Gamma/\Delta)$, $(X,R_I,R_M,\mathcal{O})\not\vDash t(\Gamma/\Delta)$.

For the right-to-left direction, suppose $(X,R_I,R_M,\mathcal{O})\not\vDash t(\Gamma/\Delta)$, there is a valuation $V$ on $(X,R_I,R_M,\mathcal{O})$ such that $V(t(\varphi))=X$ for any $\varphi\in\Gamma$ and $V(t(\psi))\not=X$ for any $\psi\in\Delta$. Define $V'(p)=\rho[V(\Box_I p)]=\rho[\Box_I V(p)]$ for any propositional variable $p$. As $\Box_I V(p)$ is a clopen upset in $(X,R_I,R_M,\mathcal{O})$, so $\rho[\Box_I V(p)]$ is a clopen upset in $\rho(X,R_I,R_M,\mathcal{O})$. Thus $V'$ is indeed a valuation on $\rho(X,R_I,R_M,\mathcal{O})$. We prove by induction that $V'(\varphi)=\rho[V(t(\varphi))]$ for any formula $\varphi$.

The cases when $\varphi$ is a propositional letters, $\bot$ or $\top$ are easy.

If $\varphi=\psi\lor\theta$, $V'(\psi\lor \theta)=V'(\psi)\cup V'(\theta)=\text{(By IH) } \rho[V(t(\psi))]\cup \rho[V(t(\theta)))]=\rho[V(t(\psi))\cup V(t(\theta))]=\rho[V(t(\psi)\lor t(\theta))]$. As $V(t(\varphi))$ is an upset for any formula $\varphi$, so $V(t(\psi)\lor t(\theta))=V(t(\psi))\cup V(t(\theta))$ is an upset. Thus $\rho[V(t(\psi)\lor t(\theta))]=\rho[\Box_I V(t(\psi)\lor t(\theta))]=\rho[V(\Box_I(t(\psi)\lor t(\theta)))]=\rho[V(t(\psi\lor \theta))]$. The case when $\varphi=\psi\land\theta$ is similar.

If $\varphi=\psi\rightarrow \theta$, $V'(\psi\rightarrow\theta)=\rho[X]\setminus {\downarrow} (V'(\psi)\setminus V'(\theta))=\text{(By IH) } \rho[X]\setminus {\downarrow} (\rho[V(t(\psi))]\setminus \rho[V(t(\theta))])$. And $\rho(V(t(\psi\rightarrow\theta)))=\rho(V(\Box_I(t(\psi)\rightarrow t(\theta))))=\rho[\Box_I(X\setminus V(t(\psi))\cup V(t(\theta)))]$. Now suppose $\rho(x)\in \rho[X]\setminus {\downarrow} (\rho[V(t(\psi))]\setminus \rho[V(t(\theta))])$, namely $\rho(x)\not\in {\downarrow} (\rho[V(t(\psi))]\setminus \rho[V(t(\theta))])$. As $V(t(\psi))$ and $V(t(\theta))$ are upsets, $\rho[V(t(\psi))]\setminus \rho[V(t(\theta))]=\rho[V(t(\psi))\setminus V(t(\theta))]\footnote{$\rho[X\setminus Y]=\rho[X]\setminus \rho[Y]$ for any upsets X, Y: clearly, $\rho[X]\setminus \rho[Y]\subseteq \rho[X\setminus Y]$. Suppose $a\in \rho[X\setminus Y]$, there is $x\in X\setminus Y$ such that $a=\rho(x)$. As $X,Y$ are upsets, for any $x\backsim x'$, $x'\not\in Y$. Thus $\rho(x)\not\in \rho[Y]$, $a=\rho(x)\in \rho[X]\setminus \rho[Y]$. Thus $\rho[X\setminus Y]\subseteq \rho[X]\setminus\rho[Y]$, $\rho[X\setminus Y]= \rho[X]\setminus\rho[Y]$.}$. Thus for any $xR_I y$, by definition $\rho(x)\leq \rho(y)$, thus $\rho(y)\not\in \rho[V(t(\psi))\setminus V(t(\theta))]$, and $y\not\in V(t(\psi))\setminus V(t(\theta)) $. Therefore, $x\in \Box_I(X\setminus V(t(\psi))\cup V(t(\theta)))$, and $\rho(x)\in \rho[\Box_I(X\setminus V(t(\psi))\cup V(t(\theta)))]$. This proves that $\rho[X]\setminus {\downarrow} (\rho[V(t(\psi))]\setminus \rho[V(t(\theta))])\subseteq \rho[\Box_I(X\setminus V(t(\psi))\cup V(t(\theta)))]$. 

\textcolor{black}{For the other inclusion, suppose} $\rho(x)\in \rho[\Box_I(X\setminus V(t(\psi))\cup V(t(\theta)))]$ while $\rho(x)\in {\downarrow} (\rho[V(t(\psi))]\setminus \rho[V(t(\theta))])$. Then $x\backsim y$ for some $y\in \Box_I(X\setminus V(t(\psi))\cup V(t(\theta)))$, and thus $x\in \Box_I(X\setminus V(t(\psi))\cup V(t(\theta)))$. As $\rho(x)\in {\downarrow} (\rho[V(t(\psi))]\setminus \rho[V(t(\theta))])$, there exists $\rho(x)\leq \rho(z)$ such that $\rho(z)\in \rho[V(t(\psi))]\setminus \rho[V(t(\theta))]$. There exists $z'\in V(t(\psi))\setminus V(t(\theta))$ such that $\rho(z)=\rho(z')$. As $\rho(x)\leq \rho(z')$, $x R_I z'$, which contradicts the fact that $x\in \Box_I(X\setminus V(t(\psi))\cup V(t(\theta)))$. Thus if $\rho(x)\in \rho[\Box_I(X\setminus V(t(\psi))\cup V(t(\theta)))]$, then $\rho(x)\not\in {\downarrow} (\rho[V(t(\psi))]\setminus \rho[V(t(\theta))])$, namely $\rho[\Box_I(X\setminus V(t(\psi))\cup V(t(\theta)))]\subseteq \rho[X]\setminus {\downarrow} (\rho[V(t(\psi))]\setminus \rho[V(t(\theta))])$. We thus have that $\rho[\Box_I(X\setminus V(t(\psi))\cup V(t(\theta)))]= \rho[X]\setminus {\downarrow} (\rho[V(t(\psi))]\setminus \rho[V(t(\theta))])$. This proves that $V'(\psi\rightarrow\theta)=\rho(V(t(\psi\rightarrow\theta)))$.

If $\varphi=\Box\psi$, $V'(\Box\psi)=\Box_{[R_I\circ R_M\circ R_I]}V'(\psi)=(\text{IH}) \Box_{[R_I\circ R_M\circ R_I]} \rho[V(t(\psi)]$. And   $\rho[V(t(\Box\psi))]=\rho[V(\Box_I\Box_M t(\psi))]=\rho[\Box_I\Box_M V(t(\psi))]$. Suppose $a\in \rho[\Box_I\Box_M V(t(\psi))]$, then  $a=\rho(x)$ for some $x\in \Box_I\Box_M V(t(\psi))$ . As $V(t(\psi))$ is an upset, $x\in \Box_I\Box_M V(t(\psi))=\Box_I \Box_M \Box_I V(t(\psi))$. Thus $a=\rho(x)\in \Box_{[R_I\circ R_M\circ R_I]} \rho[V(t(\psi))]$. As $a\in \rho[\Box_I\Box_M V(t(\psi))]$ is arbitrary, we have that $\rho[\Box_I\Box_M V(t(\psi))]\subseteq \Box_{[R_I\circ R_M\circ R_I]} \rho[V(t(\psi))]$. 

\textcolor{black}{For the other inclusion,  suppose} $\rho(x)\in \Box_{[R_I\circ R_M\circ R_I]} \rho[V(t(\psi))]$ while $x\not\in \Box_I\Box_M\Box_I V(t(\psi))=\Box_I\Box_M V(t(\psi))$, then there exist $xR_I x_1$, $x_1 R_M x_2 $ and $x_2 R_I y$ such that $y\not\in V(t(\psi))$. Thus,  $\rho(x)[R_I\circ R_M\circ R_I]\rho(y)$ by definition. As $V(t(\psi))$ is an upset, $\rho(y)\not\in \rho[V(t(\psi))]$ which contradicts the assumption that $\rho(x)\in \Box_{[R_I\circ R_M\circ R_I]} \rho[V(t(\psi))]$. Thus if $\rho(x)\in \Box_{[R_I\circ R_M\circ R_I]} \rho[V(t(\psi))]$, then $\rho(x)\in \rho[\Box_I\Box_M V(t(\psi))]$. Therefore, $\Box_{[R_I\circ R_M\circ R_I]} \rho[V(t(\psi))]\subseteq \rho[\Box_I\Box_M V(t(\psi))]$, and thus $\Box_{[R_I\circ R_M\circ R_I]} \rho[V(t(\psi))]= \rho[\Box_I\Box_M V(t(\psi))]$. This proves that $V'(\Box\psi)=\rho[V(t(\Box\psi))]$ and finishes the induction.

Now as $V(t(\varphi))=X$ for any $\varphi\in \Gamma$, $V'(\varphi)=\rho[V(t(\varphi))]=\rho[X]$ for any $\varphi\in\Gamma$. As $V(t(\psi))\not=X$ for any $\psi\in\Delta$ and $V(t(\psi))$ is an upset, $V'(\psi)=\rho[V(t(\psi))]\not=\rho[X]$ for any $\psi\in \Delta$. Thus $V'$ is a valuation on $\rho(X,R_I,R_M,\mathcal{O})$ which refutes $\Gamma/\Delta$. \textcolor{black}{Therefore,} $\rho(X,R_I,R_M,\mathcal{O})\not\vDash \Gamma/\Delta$.

\end{proof}

When restricted to formulas, the above proposition is exactly   \cite[Lem. 5]{zfint} or \cite[Lem. 19]{zfon} in algebraic terms. Thus it is a generalization of that result from formulas to multi-conclusion rules. Now we define the \textit{modal companions} for intuitionistic modal logics. This concept connects intuitionistic modal logics with bimodal logics via the G\"odel translation.

\begin{definition}
\label{4.1.18}
Let $L\in \mathbf{NExt}(\textbf{IntK$^R_\Box$})$ and $M\in \mathbf{NExt}(\textbf{S4}\otimes \textbf{K}^R)$. \textcolor{black}{Then} $M$ is a \textit{modal companion} of $L$ if $\Gamma/\Delta\in L$ $\iff$ $t(\Gamma/\Delta)\in M$. Moreover, let $L\in \mathbf{NExt}(\textbf{IntK$_\Box$})$ and $M\in\mathbf{NExt}(\textbf{S4}\otimes \textbf{K})$, \textcolor{black}{then} $M$ is a \textit{modal companion} of $L$ if $\varphi\in L$ $\iff$ $t(\varphi)\in M$.
\end{definition}

We then define the following maps between the lattices $\mathbf{NExt}(\textbf{IntK$^R_\Box$})$ and $\mathbf{NExt}(\textbf{S4}\otimes \textbf{K}^R)$. The analogues of them for the lattices of superintuitionistic logics and classical modal logics over $\textbf{S4}$ are quite well-known and can be found in many textbooks \textcolor{black}{on} modal logic. See \cite{ChagrovZakharyashev} for more details. 

$$\tau: \mathbf{NExt}(\textbf{IntK$^R_\Box$})\rightarrow \mathbf{NExt}((\textbf{S4}\otimes \textbf{K})\oplus \textbf{Mix}^R)$$ $$\tau(L)=(\textbf{S4}\otimes \textbf{K})\oplus \textbf{Mix}^R\oplus\{t(\Gamma/\Delta)\mid \Gamma/\Delta\in L\}$$

$$\sigma:\mathbf{NExt}(\textbf{IntK$^R_\Box$})\rightarrow \mathbf{NExt}((\textbf{Grz}\otimes \textbf{K})\oplus \textbf{Mix}^R)$$ $$\sigma(L)=(\textbf{Grz}\otimes \textbf{K})\oplus \textbf{Mix}^R\oplus\{t(\Gamma/\Delta)\mid \Gamma/\Delta\in L\}$$

$$\rho:\mathbf{NExt}(\textbf{S4}\otimes \textbf{K}^R)\rightarrow \mathbf{NExt}(\textbf{IntK$^R_\Box$})$$ $$\rho(M)=\{\Gamma/\Delta\mid t(\Gamma/\Delta)\in M\}\footnote{To see why $\rho(M)\in \mathbf{NExt}(\textbf{IntK$^R_\Box$})$, note that $t(\Gamma/\Delta)\in M$ iff (by Theorem \ref{2.4.7}) $Alg(M)\vDash t(\Gamma/\Delta)$ iff (by Proposition \ref{4.1.17}) $\{\rho(\A)\mid \A\in Alg(M)\}\vDash \Gamma/\Delta$ iff $Uni(\{\rho(\A)\mid \A\in Alg(M)\})\vDash \Gamma/\Delta$ iff (by Theorem \ref{2.5.6}) $\Gamma/\Delta \in Ru(Uni(\{\rho(\A)\mid \A\in Alg(M)\}))$. Thus $\rho(M)=Ru(Uni(\{\rho(\A)\mid \A\in Alg(M)\}))\in \mathbf{NExt}(\textbf{IntK$^R_\Box$})$. Similarly we can show that $\rho(M)\in \mathbf{NExt}(\textbf{IntK$_\Box$})$ where $M\in \mathbf{NExt}(\textbf{S4}\otimes \textbf{K})$.}$$

Similarly, these maps can be defined for logics: $$\tau: \mathbf{NExt}(\textbf{IntK$_\Box$})\rightarrow \mathbf{NExt}((\textbf{S4}\otimes \textbf{K})\oplus \textbf{Mix})$$
$$\tau(L)=(\textbf{S4}\otimes \textbf{K})\oplus \textbf{Mix}\oplus\{t(\varphi)\mid \varphi\in L\}$$

$$\sigma:\mathbf{NExt}(\textbf{IntK$_\Box$})\rightarrow \mathbf{NExt}((\textbf{Grz}\otimes \textbf{K})\oplus \textbf{Mix})$$ $$\sigma(L)=(\textbf{Grz}\otimes \textbf{K})\oplus \textbf{Mix}\oplus\{t(\varphi)\mid \varphi\in L\}$$ 

$$\rho:\mathbf{NExt}(\textbf{S4}\otimes \textbf{K})\rightarrow \mathbf{NExt}(\textbf{IntK$_\Box$})$$ $$\rho(M)=\{\varphi\mid t(\varphi)\in M\}$$

Semantically, we can extend the algebraic mappings $\sigma$ and $\rho$ (given below Proposition \ref{4.1.12}) to universal classes:

$$\sigma:\textbf{Uni}(\mathbf{MHA})\rightarrow \textbf{Uni}(\mathbf{S4}\otimes\mathbf{K})$$ $$\sigma(\mathcal{A})=Uni(\{\sigma(\A)\mid \A\in \mathcal{A}\})$$

$$\rho:\textbf{Uni}(\mathbf{S4}\otimes\mathbf{K})\rightarrow \textbf{Uni}(\mathbf{MHA})$$ $$\rho(\mathcal{W})=\{\rho(\A)\mid \A\in \mathcal{W}\}$$

We can also define the semantic analogue to $\tau$ as follows: $$\tau: \textbf{Uni}(\mathbf{MHA})\rightarrow \textbf{Uni}((\mathbf{S4}\otimes\mathbf{K})\oplus\mathbf{Mix})$$ $$\tau(\mathcal{A})=\{\A\text{ is a } (S4\otimes K)\oplus Mix\text{-algebra}\mid \rho(\A)\in \mathcal{A}\}\footnote{It is easy to check that $\rho(\mathcal{W})$ is always a universal class by Proposition \ref{4.1.17}. To see that $\tau(\mathcal{A})$ a universal class, note that $\A\in Alg(\tau(Ru(\mathcal{A})))$ (here $\tau$ is the syntactic one as you can tell) iff $\A$ is a $(S4\otimes K)\oplus Mix\text{-algebra}$ such that $\A\vDash t(\Gamma/\Delta)$ for any $\Gamma/\Delta\in Ru(\mathcal{A})$ iff (by Proposition \ref{4.1.17}) $\A$ is a $(S4\otimes K)\oplus Mix\text{-algebra}$ such that $\rho(\A)\vDash\Gamma/\Delta$ for any $\Gamma/\Delta\in Ru(\mathcal{A})$ iff (by Theorem \ref{2.5.6}) $\A$ is a $(S4\otimes K)\oplus Mix\text{-algebra}$ such that $\rho(\A)\in \mathcal{A}$ iff $\A\in \tau(\mathcal{A})$. Thus $\tau(\mathcal{A})=Alg(\tau(Ru(\mathcal{A})))$ is a universal class.}.$$

\subsubsection{Proof of The Blok-Esakia Theorem}

Now, we can start with the  following lemma which is the counterpart to \cite[Lem. 2.50]{amaster} and plays a key role in the proof of the Blok-Esakia theorem.

\begin{lemma}
\label{4.1.19}
Let $\A$ be a $(Grz\otimes K)\oplus Mix$-algebra, then for any bimodal multi-conclusion rule $\Gamma/\Delta$, \textcolor{black}{we have that} $$\A\vDash \Gamma/\Delta \text{ iff } \sigma(\rho(\A))\vDash\Gamma/\Delta.$$ 
  
\end{lemma}

\begin{proof}

For the left-to-right direction, by Proposition \ref{4.1.15}, $\sigma(\rho(\A))\rightarrowtail \A$. Thus, the result follows immediately.

For the right-to-left direction, we prove that if $\A_*\not\vDash\Gamma/\Delta$, then $\sigma(\rho(\A))_*\not\vDash\Gamma/\Delta$. By Proposition \ref{4.1.13}, $\sigma(\rho(\A))_*=\sigma(\rho(\A)_*)=\sigma(\rho(\A_*))$. And by Proposition \ref{4.1.6}, we can assume that $\Gamma/\Delta=\mu(\B,D^I,D^M)$ where $\B$ is a finite $S4\otimes K$-algebra which validates $\textbf{Mix}$ and $D^I,D^M\subseteq B$. Suppose $\A_*\not\vDash\mu(\B,D^I,D^M)$ where $R_I\circ R_M\circ R_I=R_M$ on $\A_*$ as $\A$ validates $\textbf{Mix}$, then by Proposition \ref{4.1.9}, there is a continuous stable surjection $f:\A_*\rightarrow \B_*$ satisfying CDC for $\{\beta(a)\mid a\in D^I\}$ and $\{\beta(b)\mid b\in D^M\}$. We construct a continuous stable surjection $g:\sigma(\rho(\A_*))\rightarrow \B_*$ satisfying CDC for $\{\beta(a)\mid a\in D^I\}$ and $\{\beta(b)\mid b\in D^M\}$, and this will show that $\sigma(\rho(\A_*))\not\vDash \mu(\B,D^I,D^M)$ by Proposition \ref{4.1.9} and thus finish the proof.

We use the construction in the proof of \cite[Lem. 2.50]{amaster}. Let $C\subseteq \B_*$ be an $R_I$-cluster, consider $Z_C=f^{-1}(C)$. As $f$ is continuous, $Z_C$ is clopen in $\A_*$. Since $f$ is stable, $Z_C$ does not cut any $R_I$-cluster. As $\sigma(\rho(\A_*))$ has the quotient topology, $\rho[Z_C]$ is clopen. Assume $C=\{x_1,...x_n\}$ (note $\B_*$ is finite), $f^{-1}(x_i)\subseteq Z_C$ is clopen. As $\A_*$ is a $Grz\otimes K$-space, by Proposition \ref{2.6.13}, we have that $M_i=max_{R_I}(f^{-1}(x_i))$ is closed and $M_i$ does not cut any $R_I$-cluster. As $\sigma(\rho(\A_*))$ has the quotient topology, $\rho[M_i]$ is closed. And for $i\not=j$, $\rho[M_i]\cap\rho[M_j]=\emptyset$.

Then we find disjoint clopen sets $U_1,...,U_n$ of $\sigma(\rho(\A_*))$ with $\rho[M_i]\subseteq U_i$ and $\bigcup_i U_i=\rho[Z_C]$. Let $k\leq n$ and assume that $U_i$ has been defined for all $i< k$. If $k=n$, let $U_n=\rho[Z_C]\setminus(\bigcup_{i<k}U_i)$. Otherwise, let $V_k=\rho[Z_C]\setminus (\bigcup_{i<k} U_i)$ and note that $V_k$ contains $\rho[M_i]$ for $k\leq i\leq n$. As $\sigma (\rho(\A_*))$ is a Stone space, by its separation properties, for each $k<i\leq n$, there is a clopen set $U_{k_i}$ of $\sigma(\rho(\A_*))$ such that $\rho[M_k]\subseteq U_{k_i}$ and $\rho[M_i]\cap U_{k_i}=\emptyset$. Then let $U_k=\bigcap_{k<i\leq n} U_{k_i}\cap V_k$. Now define $g_C:\rho[Z_C]\rightarrow C$ as follows: $g_C(z)=x_i$ iff $z\in U_i$. Finally, define $g:\sigma(\rho(\A_*))\rightarrow \B_*$ as follows:

\vspace{1mm}
$g(\rho(z))=\begin{cases}f(z) & \text{ if }f(z)\text{ does not belong to any proper $R_I$-cluster}  \cr g_C(\rho(z)) & \text{ where } C \text{ is the proper $R_I$-cluster containing } f(z)    \end{cases}$

\vspace{1mm}
As shown in the proof of \cite[Lem. 2.50]{amaster}, $g$ is surjective, continuous, and relation-preserving w.r.t $R_I$, and satisfies CDC for $\{\beta(a)\mid a\in D^I\}$. We only need to check that $g$ is relation-preserving w.r.t $R_M$ and satisfies CDC for $\{\beta(b)\mid b\in D^M\}$.

As $\B$ validates $\textbf{Mix}$, $R_I\circ R_M\circ R_I=R_M$ on $\B_*$. Now 
suppose $\rho(a) R_M \rho(b)$, as $f$ is stable, $f(a)R_M f(b)$. Since for any $z$, $f(z)$ and $g(\rho(z))$ are mapped to the same $R_I$-cluster, in any case, $g(\rho(a))R_If(a)$ and $f(b)R_I g(\rho(b))$. As $R_I\circ R_M\circ R_I=R_M$ on $\B_*$, $g(\rho(a))R_M g(\rho(b))$. $g$ is thus relation-preserving w.r.t $R_M$. 

Suppose $g(\rho(x))R_M y$ where $y\in \delta$ for some $\delta\in \{\delta_b\mid b\in D^M\}$\footnote{Recall that $\delta_b$ is just the (set-theoretic) complement of $\beta(b)$ as defined in \textcolor{black}{Footnote 11}.}. As $g(\rho(x))$ and $f(x)$ belong to the same $R_I$-cluster and $R_I\circ R_M\circ R_I= R_M$ on $\B_*$, it follows that $f(x)R_M y$. Since $f$ satisfies CDC for $\{\beta(b)\mid b\in D^M\}$, there exists $z\in \A_*$ such that $xR_M z$ and $f(z)\in \delta$. Since $f^{-1}(f(z))$ is clopen in $\A_*$, there exists $z'\in max_{R_I}f^{-1}(f(z))$ such that $zR_I z'$. As $R_I\circ R_M\circ R_I= R_M$ on $\A_*$, we have that $xR_M z'$ and $f(z')\in\delta$. And from $z'\in max_{R_I}f^{-1}(f(z))$, it follows that $f(z')=g(\rho(z'))$ by construction, and thus $g(\rho(z'))\in \delta$. As $x R_M z'$, $\rho(x) [R_I\circ R_M \circ R_I] \rho(z')$, $g$ satisfies CDC for $\{\beta(b)\mid b\in D^M\}$.

\end{proof}

Compared to that of \cite[Lem. 2.50]{amaster}, in the above proof, we have to exploit the assumption that $\B$ validates $\textbf{Mix}$ heavily, which is only given by Proposition \ref{4.1.6}.  

\begin{theorem}
\label{4.1.20}
For every $\mathcal{A}\in \textbf{Uni}((\mathbf{Grz}\otimes\mathbf{K})\oplus \textbf{Mix})$, $\sigma(\rho(\mathcal{A}))=\mathcal{A}$. 
 
\end{theorem}

\begin{proof}

By Proposition \ref{4.1.15}, for every $\A\in\mathcal{A}$, $\sigma(\rho(\A))\rightarrowtail \A$. Thus $\sigma(\rho(\mathcal{A}))\subseteq\mathcal{A}$. Now, suppose $\mathcal{A}\not\vDash\Gamma/\Delta$, there is $\A\in\mathcal{A}$ such that $\A\not\vDash\Gamma/\Delta$. As $\A$ validates $\textbf{Mix}$, by Lemma \ref{4.1.19}, we have $\sigma(\rho(\A))\not\vDash\Gamma/\Delta$. Thus $Ru(\sigma(\rho(\mathcal{A})))\subseteq Ru(\mathcal{A})$, by Theorem \ref{2.4.7}, $\mathcal{A}\subseteq\sigma(\rho(\mathcal{A}))$. Thus $\sigma(\rho(\mathcal{A}))=\mathcal{A}$.

\end{proof}

The above theorem is the counterpart to \cite[Lem. 4.4]{Stronkowski2018OnTB}, which was also proved later by the method of stable canonical rules in \cite[Thm. 2.51]{amaster}. Having this theorem, the following results can be obtained by well-known routine arguments shown in \cite{amaster, an} and \cite{can} for example.

\begin{lemma}
\label{4.1.21}
For each $L\in \mathbf{NExt}(\textbf{IntK$^R_\Box$})$ and $M\in \mathbf{NExt}(\textbf{S4}\otimes \textbf{K}^R)$, the following hold:

\begin{itemize}
    \item $Alg(\tau (L))=\tau (Alg(L))$
    \item $Alg(\sigma (L))=\sigma(Alg(L))$
    \item $Alg(\rho(M))=\rho(Alg(M))$

\end{itemize}
\end{lemma}

\begin{proof}
For any $(S4\otimes K)\oplus Mix$-algebra $\A$, $\A\in Alg(\tau (L))$ iff $\A\vDash t(\Gamma/\Delta)$ for any $\Gamma/\Delta\in L$ iff (by Proposition \ref{4.1.17})$\rho(\A)\vDash \Gamma/\Delta$  for any $\Gamma/\Delta\in L$ iff $\rho(\A)\in Alg(L)$ iff $\A\in \tau(Alg(L))$. Thus $Alg(\tau (L))=\tau (Alg(L))$.

For the second one, for any modal Heyting algebra $\A$, $\A_*$ is a modal Esakia space. By Proposition \ref{4.1.11}, we know that $\sigma(\A_*)$ validates $\textbf{Mix}$. For any modal Esakia space $\mathcal{X}=(X,\leq, R, \mathcal{O})$, by definition $(X,\leq, R)$ is just an Esakia space. If we dismiss $R$, then $\sigma$ operates on $(X,\leq, R)$ exactly the same way as that in \cite[Def. 3.32]{amaster}, which is known to give us a $Grz$-space. Thus $\sigma(\A_*)$ is a $(Grz\otimes K)\oplus Mix$-modal space. By Proposition \ref{4.1.13}, $(\sigma(\A))_*\cong \sigma(\A_*)$, and thus $\sigma(\A)$ is a $(Grz\otimes K)\oplus Mix$-algebra. Therefore, by Theorem \ref{4.1.20}, it suffices to prove that for any $\A=\sigma(\rho(\A))$ which is a $(Grz\otimes K)\oplus Mix$-algebra, $\A\in Alg(\sigma(L))$ iff $\A\in \sigma(Alg(L))$.

Suppose $\A=\sigma(\rho(\A))\in\sigma (Alg(L))$ where $\A$ is a $(Grz\otimes K)\oplus Mix$-algebra. For any $\B\in Alg(L)$, we have that $\B\vDash\Gamma/\Delta$ for any $\Gamma/\Delta\in L$. By Proposition \ref{4.1.14}, $\rho(\sigma(\B))\vDash\Gamma/\Delta$. Then by Proposition \ref{4.1.17}, $\sigma(\B)\vDash t(\Gamma/\Delta)$ for any $\B\in Alg(L)$ and $\Gamma/\Delta\in L$. As $\sigma(Alg(L))$ is generated by $\{\sigma(\B)\mid \B\in Alg(L)\}$, we have that $\sigma(Alg(L))\vDash t(\Gamma/\Delta)$ for any $\Gamma/\Delta\in L$. Thus $\A\vDash t(\Gamma/\Delta)$ for any $\Gamma/\Delta\in L$, namely $\A\in Alg(\sigma(L))$. For the other direction, suppose $\A\in Alg(\sigma(L))$, $\A\vDash t(\Gamma/\Delta)$ for any $\Gamma/\Delta\in L$. By Proposition \ref{4.1.17}, $\rho(\A)\vDash\Gamma/\Delta$ for any $\Gamma/\Delta\in L$. Thus $\rho(\A)\in Alg(L)$, $\A=\sigma(\rho(\A))\in \sigma(Alg(L))$.

For the third one, let $\A$ be a modal Heyting algebra, if $\A\in \rho(Alg(M))$, then $\A=\rho(\B)$ for some $\B\in Alg(M)$. For any $t(\Gamma/\Delta)\in M$, we have that $\B\vDash t(\Gamma/\Delta)$. By Proposition \ref{4.1.17}, $\A=\rho(\B)\vDash\Gamma/\Delta$ for any $t(\Gamma/\Delta)\in M$. Thus $\A\in Alg(\rho(M))$. This proves that $\rho(Alg(M))\subseteq Alg(\rho(M))$. For the other direction, if $\rho(Alg(M))\vDash\Gamma/\Delta$, then by Proposition \ref{4.1.17}, $Alg(M)\vDash t(\Gamma/\Delta)$. Thus $t(\Gamma/\Delta)\in M$, and $\Gamma/\Delta\in \rho(M)$. We have that $Ru(\rho(Alg(M)))\subseteq \rho(M)$, and thus by Theorem \ref{2.5.6},   $Alg(\rho(M))\subseteq\rho(Alg(M))$. Therefore, $Alg(\rho(M))=\rho(Alg(M))$

\end{proof}

The above lemma is the analogue to \cite[Thm. 5.9]{can}, whose proof, as we can see, heavily depends on Proposition \ref{4.1.17}. We then can give a characterization of modal companions of an intuitionistic modal logic as that in \cite[Lem. 2.53]{amaster}.

\begin{proposition}
\label{4.1.22}
For any $L\in \mathbf{NExt}(\textbf{IntK$^R_\Box$})$, we have that $M\in \mathbf{NExt}(\textbf{S4}\otimes \textbf{K}^R)$ is a modal companion of $L$ iff $Alg(L)=\rho(Alg(M))$

\end{proposition}

\begin{proof}
For the left-to-right direction, if $M$ is a modal companion of $L$, then $L=\rho(M)$. Thus $Alg(L)=Alg(\rho(M))=\rho(Alg(M))$ by the above lemma. For the other direction, assume $Alg(L)=\rho(Alg(M))$, then $\Gamma/\Delta\in L$ iff $Alg(L)\vDash \Gamma/\Delta$ iff $\rho(Alg(M))\vDash \Gamma/\Delta$ iff (by Proposition \ref{4.1.17}) $Alg(M)\vDash t(\Gamma/\Delta)$ iff (by Theorem \ref{2.4.7}) $t(\Gamma/\Delta)\in M$. $M$ is thus a modal companion of $L$.

\end{proof}

The following proposition is a generalization of \cite[Thm. 27]{zfon} from logics to multi-conclusion consequence relations.

\begin{proposition}
For every $L\in \mathbf{NExt}(\textbf{IntK$^R_\Box$})$, the modal \textcolor{black}{companions of $L$ which contain} $\textbf{Mix}$ form an interval $\{M\in \mathbf{NExt}((\textbf{S4}\otimes \textbf{K})\oplus\textbf{Mix}^R)\mid \tau(L)\subseteq M\subseteq \sigma(L)\}$.    
\end{proposition}

\begin{proof}
By Lemma \ref{4.1.21}, it suffices to prove that for any $M\in \mathbf{NExt}((\textbf{S4}\otimes \textbf{K})\oplus\textbf{Mix}^R)$, we have that $M$ is a modal companion of $L$ iff $\sigma(Alg(L))\subseteq Alg(M)\subseteq \tau (Alg(L))$.

For the left-to-right direction, assume $M$ is a modal companion of $L$ which contains $\textbf{Mix}$. By Proposition \ref{4.1.22}, $Alg(L)=\rho(Alg(M))$. For any $\A\in Alg(M)$, we have that $\rho(\A)\in \rho(Alg(M))$. Thus $\rho(\A)\in Alg(L)$, and $\A\in \tau(Alg(L))$. Therefore,  $Alg(M)\subseteq \tau (Alg(L))$. To see that $\sigma(Alg(L))\subseteq Alg(M)$, as $\sigma(A)$ is a $(Grz\otimes K)\oplus Mix$-algebra (shown in the proof of Lemma \ref{4.1.21}), by Theorem \ref{4.1.20}, it suffices to prove that for any $\A= \sigma(\rho(\A))\in \sigma (Alg(L))$, we have that $\A\in Alg(M)$. Suppose $\A\cong \sigma(\rho(\A))\in \sigma (Alg(L))$, by Lemma \ref{4.1.21}, $\A\in Alg(\sigma(L))$. Thus for any $\Gamma/\Delta\in L$, we have that $\A\vDash t(\Gamma/\Delta)$. By Proposition \ref{4.1.17}, $\rho(\A)\vDash \Gamma/\Delta$ for any $\Gamma/\Delta\in L$. Thus $\rho(\A)\in Alg(L)=\rho(Alg(M))$, namely $\rho(\A)\cong \rho(\B)$ for some $\B\in Alg(M)$. $\A=\sigma(\rho(\A))\cong \sigma(\rho(\B))\rightarrowtail \B\in Alg(M)$. As $Alg(M)$ is a universal class which is closed under subalgebras, $\A\in M$. Thus $\sigma(Alg(L))\subseteq Alg(M).$ 

For the other direction, assume that $\sigma(Alg(L))\subseteq Alg(M)\subseteq \tau (Alg(L))$. By Proposition \ref{4.1.14}, $\rho(\sigma(Alg(L)))=Alg(L)$. As $\sigma(Alg(L))\subseteq Alg(M)$, we have that $Alg(L)\subseteq \rho(Alg(M))$. As $Alg(M)\subseteq \tau (Alg(L))$, it follows that $\rho(Alg(M))\subseteq \rho(\tau(Alg(L)))$. Assume $\A\in \tau(Alg(L))$, then $\rho(\A)\in Alg(L)$, and $\sigma(\rho(\A))\in 
\sigma(Alg(L))\subseteq Alg(M)$. Thus $\rho(\sigma(\rho(\A)))\in \rho(Alg(M))$. By Proposition \ref{4.1.14}, $\rho(\A)\in \rho(Alg(M))$. Therefore, $\rho(\tau(Alg(L)))\subseteq\rho(Alg(M))$, and $\rho(\tau(Alg(L)))=\rho(Alg(M))$. By definition, $\rho(\tau(Alg(L)))\subseteq Alg(L)$, and thus $\rho(Alg(M))\subseteq Alg(L)$. Therefore, $\rho(Alg(M))= Alg(L)$. $M$ is a modal companion of $L$ by Proposition \ref{4.1.22}.

\end{proof}

\textcolor{black}{Now, we obtain the Blok-Esakia theorem for intuitionistic modal logics:}

\begin{theorem}
\label{4.1.24}
\quad
\begin{itemize}
    \item[1.] The mappings $\sigma:\mathbf{NExt}(\textbf{IntK$^R_\Box$})\rightarrow \mathbf{NExt}((\textbf{Grz}\otimes \textbf{K})\oplus \textbf{Mix}^R)$ and $\rho:\mathbf{NExt}((\textbf{Grz}\otimes \textbf{K})\oplus \textbf{Mix}^R)\rightarrow \mathbf{NExt}(\textbf{IntK$^R_\Box$})$ are lattice isomorphisms and mutual inverses.
    \item[2.] The mappings $\sigma:\mathbf{NExt}(\textbf{IntK$_\Box$})\rightarrow \mathbf{NExt}((\textbf{Grz}\otimes \textbf{K})\oplus \textbf{Mix})$ and $\rho:\mathbf{NExt}((\textbf{Grz}\otimes \textbf{K})\oplus \textbf{Mix})\rightarrow \mathbf{NExt}(\textbf{IntK$_\Box$})$
     are lattice isomorphisms and mutual inverses.

\end{itemize}

\end{theorem}

\begin{proof}
For item 1, by Lemma \ref{4.1.21}, it suffices to prove that $\sigma:\textbf{Uni}(\mathbf{MHA})\rightarrow \textbf{Uni}((\mathbf{Grz}\otimes\mathbf{K})\oplus\mathbf{Mix})$ and $\rho:\textbf{Uni}((\mathbf{Grz}\otimes\mathbf{K})\oplus\mathbf{Mix}) \rightarrow \textbf{Uni}(\mathbf{MHA})$ are lattice isomorphisms and mutual inverses.

We first prove that $\sigma$ and $\rho$ are mutual inverses and thus bijective. For any  $\mathcal{A}\in\textbf{Uni}((\mathbf{Grz}\otimes\mathbf{K})\oplus\mathbf{Mix})$, by Theorem \ref{4.1.20}, $\sigma(\rho(\mathcal{A}))=\mathcal{A}$. Let $\mathcal{A}\in \textbf{Uni}(\mathbf{MHA})$ be arbitrary, by Proposition \ref{4.1.14}, it is clear that $\mathcal{A}\subseteq \rho(\sigma(\mathcal{A}))$. Suppose $\mathcal{A}\vDash\Gamma/\Delta$, namely for any $\A\in\mathcal{A}$, we have that $\A\vDash\Gamma/\Delta$. By Proposition \ref{4.1.14}, $\rho(\sigma(\A))\vDash \Gamma/\Delta$, for any $\A\in\mathcal{A}$. By Proposition \ref{4.1.17}, $\sigma(\A)\vDash t(\Gamma/\Delta)$ for any $\A\in\mathcal{A}$. Thus $\sigma(\mathcal{A})\vDash t(\Gamma/\Delta)$. By Proposition \ref{4.1.17}, $\rho(\sigma(\mathcal{A}))\vDash \Gamma/\Delta$. Thus $Ru(\mathcal{A})\subseteq Ru(\rho(\sigma(\mathcal{A})))$, by Theorem \ref{2.4.7} $\rho(\sigma(\mathcal{A}))\subseteq \mathcal{A}$. Therefore, $\rho(\sigma(\mathcal{A}))=\mathcal{A}$. 

We then prove that $\sigma$ and $\rho$ preserve joins and meets, and thus are lattice morphisms. Clearly, for any $\mathcal{A},\mathcal{B}\in\textbf{Uni}((\mathbf{Grz}\otimes\mathbf{K})\oplus\mathbf{Mix})$, we have that  $\rho(\mathcal{A}\cap \mathcal{B})\subseteq \rho(\mathcal{A})\cap\rho(\mathcal{B})$. Assume $\mathfrak{C}\in \rho(\mathcal{A})\cap\rho(\mathcal{B})$, there exist $\A\in\mathcal{A}$ and $\B\in\mathcal{B}$ such that $\C=\rho(\A)= \rho(\B)$, and $\sigma(\rho(\A))=\sigma(\rho(\B))\in \mathcal{A}\cap\mathcal{B}$. Thus $\rho(\sigma(\rho(\A)))\in \rho(\mathcal{A}\cap\mathcal{B})$. By Proposition \ref{4.1.14}, $\mathfrak{C}=\rho(\A)\in \rho(\mathcal{A}\cap\mathcal{B})$. Thus $\rho(\mathcal{A})\cap\rho(\mathcal{B})\subseteq \rho(\mathcal{A}\cap\mathcal{B})$, and $\rho(\mathcal{A})\cap\rho(\mathcal{B})=\rho(\mathcal{A}\cap\mathcal{B})$. Clearly, for any $\mathcal{A},\mathcal{B}\in\textbf{Uni}((\mathbf{Grz}\otimes\mathbf{K})\oplus\mathbf{Mix})$, \textcolor{black}{we have that} $\rho(\mathcal{A})\lor \rho(\mathcal{B})\subseteq \rho(\mathcal{A}\lor\mathcal{B})$. Suppose $\rho(\mathcal{A})\lor \rho(\mathcal{B})\vDash \Gamma/\Delta$, then $\rho(\mathcal{A})\vDash\Gamma/\Delta$ and $\rho(\mathcal{B})\vDash\Gamma/\Delta$. By Proposition \ref{4.1.17}, $\mathcal{A}\vDash t(\Gamma/\Delta)$ and $\mathcal{B}\vDash t(\Gamma/\Delta)$. Thus $\mathcal{A}\lor \mathcal{B}\vDash t(\Gamma/\Delta)$. By Proposition \ref{4.1.17}, $\rho(\mathcal{A}\lor \mathcal{B})\vDash \Gamma/\Delta$. Thus $Ru(\rho(\mathcal{A})\lor \rho(\mathcal{B}))\subseteq Ru(\rho(\mathcal{A}\lor \mathcal{B}))$, and by Theorem \ref{2.5.6}, $\rho(\mathcal{A}\lor \mathcal{B})\subseteq \rho(\mathcal{A})\lor \rho(\mathcal{B})$. Therefore, $\rho(\mathcal{A}\lor \mathcal{B})= \rho(\mathcal{A})\lor \rho(\mathcal{B})$, $\rho$ preserves joins and meets. For any $\mathcal{A},\mathcal{B}\in\textbf{Uni}(\mathbf{MHA})$, by what we have proved above $\sigma(\mathcal{A})\lor \sigma(\mathcal{B})=\sigma(\rho(\sigma(\mathcal{A})\lor \sigma(\mathcal{B})))=\sigma(\rho(\sigma(\mathcal{A}))\lor \rho(\sigma(\mathcal{B})))=\sigma(\mathcal{A}\lor \mathcal{B})$. Clearly, $\sigma(\mathcal{A}\cap \mathcal{B})\subseteq \sigma(\mathcal{A})\cap \sigma(\mathcal{B})$. Suppose $\sigma(\mathcal{A}\cap\mathcal{B})\vDash\Gamma/\Delta$ while $\sigma(\mathcal{A})\cap\sigma(\mathcal{B})\not\vDash\Gamma/\Delta$, there exists $\A\in \sigma(\mathcal{A})\cap\sigma(\mathcal{B})$ \textcolor{black}{such that} $\A\not\vDash\Gamma/\Delta$. By Proposition \ref{4.1.14}, $\rho(\A)\in\mathcal{A}\cap\mathcal{B}$. Then $\sigma(\rho(\A))\in\sigma(\mathcal{A}\cap\mathcal{B})$, and thus $\sigma(\rho(\A))\vDash\Gamma/\Delta$ which contradicts Lemma \ref{4.1.19}. Therefore, \textcolor{black}{$Ru(\sigma(\mathcal{A}\cap \mathcal{B}))\subseteq Ru(\sigma(\mathcal{A})\cap\sigma(\mathcal{B}))$}, and by Theorem \ref{2.4.7}, $\textcolor{black}{\sigma(\mathcal{A})\cap\sigma(\mathcal{B})\subseteq \sigma(\mathcal{A}\cap\mathcal{B})}$. Therefore, $\textcolor{black}{\sigma(\mathcal{A}\cap\mathcal{B})}=\sigma(\mathcal{A})\cap\sigma(\mathcal{B})$, $\sigma$ preserves joins and meets. This finishes our proof.

\vspace{1mm}
Item 2 follows immediately from item 1 and Propositions \ref{2.4.4} and \ref{2.5.3}.

\end{proof}

The transformations $\sigma$ and $\rho$ are useful as they may allow us to transfer questions about intuitionistic modal logic to bimodal logics and vice versa. In particular, at this stage one can easily prove that $\rho$ preserves decidability, Kripke \textcolor{black}{completeness}, the finite model property and tabularity as shown in \cite[Thm. 11]{zfint}. 

The second item of the above theorem is the Blok-Esakia theorem for intuitionistic modal logics proved in \cite{zfon}, so what we have got here is a generalization of it from logics to multi-conclusion consequence relations. However, \textcolor{black}{one may note} that \cite[Cor. 28]{zfon} is slightly stronger than the second item of the above theorem, which says that the Blok-Esakia theorem holds not only for normal extensions but also for weaker extensions ($\textbf{K}$ could be replaced by $\textbf{C}$ on both sides of the map $\sigma$ and $\rho$)\footnote{$\textbf{C}$ stands for ``congruential". It is the least modal logic which has the algebraic semantics. See \cite{zfon} for more details.} . The reason why so far we can not prove the result for \textcolor{black}{non-normal} modal logic is that we do not have a nice dual description \textcolor{black}{of} the algebraic semantics for such weak logics\footnote{This also explains the reason why unlike Kripke frames, what Wolter and Zakharyaschev called ``frames" are quite algebraic.}. As a result, \textcolor{black}{we can not exploit geometric intuitions which are quite crucial} in Lemma \ref{4.1.19}.  

Having said that, compared to the proofs in \cite{zfon}, our proofs are arguably more self-contained. Since the construction in the proof of Lemma \ref{4.1.19} is the same as that in \cite[Lem. 2.50]{amaster}, this also indicates the robustness of \textcolor{black}{this} construction.

\subsection{The Dummett-Lemmon Conjecture}

The Dummett-Lemmon conjecture for superintuitionistic logics \textcolor{black}{states} that a superintuitionistic logic is Kripke complete iff its least modal companion is Kripke complete. This conjecture was \textcolor{black}{proved correct} by Zakharyaschev in \cite{zcss}, which is a very important application of his canonical formulas. Using stable canonical rules, Cleani \cite[Thm. 2.70]{amaster} proved that the Dummet-Lemmon conjecture holds for superintuitionistic rules systems as well. \textcolor{black}{However, as mentioned in the introduction, the proof there contains a gap that can be corrected. In particular, the ``rule-collapse lemma" in Cleani's thesis \cite[Lem. 2.62]{amaster} is problematic because the collapsed stable canonical rule for superintuitionistic logics is not of the right form. One way to fix this problem is to introduce so-called ``modalized stable canonical rules" and prove that in some sense, every stable canonical rule can be seen as the conjunction of finitely many modalized ones \cite[Section 5]{an}.} 

In this section, following a somewhat similar strategy, we prove the Dummett-Lemmon conjecture for intuitionistic modal multi-conclusion consequence relations which says that for any $L\in  \mathbf{NExt}(\textbf{IntK$^R_\Box$})$, $L$ is Kripke complete iff $\tau(L)$ (i.e. the least modal companion containing $\textbf{Mix}$) is Kripke complete. \textcolor{black}{Instead of introducing something like modalized stable canonical rules as in \cite{an}}, the proof uses only the stable canonical rules for intuitionistic modal logics (also those for bimodal logics) and the Blok-Esakia theorem proved in the previous section, and thus can be seen as the peak of what we have done so far.

We start with the definition of Kripke frames in the setting of intuitionistic modal logics.

\begin{definition}

An \textit{intuitionistic modal Kripke frame} is a triple $(X,\leq, R)$ where $X$ is a non-empty set, $\leq$ is a partial order on $X$ and $R\subseteq X\times X$ \textcolor{black}{is} such that $\leq \circ R=R\circ \leq =R$.     
\end{definition}

It is easy to check that for any intuitionistic modal Kripke frame $(X,\leq, R)$, the set of all upsets of $X$ is closed under $\Box_R$. \textcolor{black}{A \textit{valuation} on a intuitionistic modal Kripke frame $(X,\leq,R)$ is a map $V: Prop\rightarrow Up(X)$ which can be recursively extended to all intuitionistic modal formulas in the standard way. Besides, for any intuitionistic modal Kripke frame $\mathfrak{X}=(X,\leq, R)$, we use $\mathfrak{X}^*$ to denote the intuitionistic modal algebra of upsets of $\mathfrak{X}$ where $\Box_R$ is defined in the same way as that in Theorem \ref{3.3.3}.}

For any $S4\otimes K$ Kripke frame $\mathfrak{F}=(X,R_I,R_M)$, it is clear that $\rho(\mathfrak{F})=(\textcolor{black}{\rho[X]},\leq,[R_I\circ R_M\circ R_I])$ is an intuitionistic modal Kripke frame where $\rho$ is exactly the same as that in Definition \ref{4.1.10} (without the topology).

The following proposition can be proved in the same way as Proposition \ref{4.1.17}.

\begin{proposition}
\label{4.2.2}
For any $S4\otimes K$ Kripke frame $(X,R_I,R_M)$ and any intuitionistic modal multi-conclusion rule $\Gamma/\Delta$, $(X,R_I,R_M)\vDash t(\Gamma/\Delta)$ iff $(\rho[X],\leq, [R_I\circ R_M\circ R_I])\vDash \Gamma/\Delta$.
    
\end{proposition}

We first \textcolor{black}{obtain} the easy direction of the Dummett-Lemmon conjecture as follows:

\begin{proposition}
\label{4.2.3}
For any $L\in  \mathbf{NExt}(\textbf{IntK$^R_\Box$})$, if $\tau(L)$ is Kripke complete, then $L$ is Kripke complete.
    
\end{proposition}

\begin{proof}
Suppose $\tau(L)$ is Kripke complete. For any intuitionistic modal multi-conclusion rule $\Gamma/\Delta$, suppose $\Gamma/\Delta\not\in L$, then there exists a modal Heyting algebra $\A$ such that $\A\vDash L$ while $\A\not\vDash \Gamma/\Delta$. If $\sigma(\A)\vDash t(\Gamma/\Delta)$, then by Proposition \ref{4.1.17} $\rho(\sigma(\A))\vDash\Gamma/\Delta$. As $\rho(\sigma(\A))\cong \A$ by Proposition \ref{4.1.14}, $\A\vDash\Gamma/\Delta$, we get a contradiction. Thus $\sigma(\A)\not\vDash t(\Gamma/\Delta)$. As $\rho(\sigma(\A))\cong \A\vDash L$, by Proposition \ref{4.1.17}, $\sigma(\A)\vDash \tau(L)$. Thus $t(\Gamma/\Delta)\not\in \tau(L)$. As $\tau(L)$ is Kripke complete, there exists an $(S4\otimes K)\oplus Mix$ Kripke frame $(X,R_I,R_M)$ such that $(X,R_I,R_M)\vDash \tau(L)$ while $(X,R_I,R_M)\not \vDash t(\Gamma/\Delta)$. Now, by Proposition \ref{4.2.2}, $(\rho[X],\leq, [R_I\circ R_M\circ R_I])\vDash L$ while $(\rho[X],\leq,[R_I\circ R_M\circ R_I])\not\vDash\Gamma/\Delta$. As $(\rho[X],\leq, [R_I\circ R_M\circ R_I])$ is an intuitionistic modal Kripke frame, this proves that $L$ is Kripke complete.
    
\end{proof}

Then we prove the so-called rule collapse lemma\footnote{Here, we actually prove something weaker than \cite[Lem. 2.62]{amaster} which is problematic as mentioned above.}. Informally speaking, it describes what happens if we collapse $R_I$-clusters in the stable canonical rules for bimodal logics.

\begin{lemma}[Rule collapse lemma]
\label{lemma 4.2.4}
For any $S4\otimes K$-modal space $\mathfrak{X}$, and any $(S4\otimes K)\oplus Mix$-modal space $\A_*$, let $\mathfrak{D}_I,\mathfrak{D}_M\subseteq \mathcal{P}(\A_*)$, if $\mathfrak{X}\not\vDash\mu(\A_*, \overline{\mathfrak{D}_I}, \overline{\mathfrak{D}_M})$, then $\sigma(\rho(\mathfrak{X}))\not\vDash\mu(\sigma(\rho(\A_*)),\overline{\rho\mathfrak{D}_I},\overline{\rho\mathfrak{D}_M})$ where $\overline{\mathfrak{D}_I}=\{\Bar{\delta}\mid \delta\in \mathfrak{D}_I\}$, $\overline{\mathfrak{D}_M}=\{\Bar{\delta}\mid \delta\in \mathfrak{D}_M\}$, $\overline{\rho\mathfrak{D}_I}=\{\overline{\rho[\delta]}\mid \delta\in \mathfrak{D}_I\}$ and $\overline{\rho\mathfrak{D}_M}=\{\overline{\rho[\delta]}\mid \delta\in \mathfrak{D}_M\}$\footnote{Recall that we use `` $\Bar{}$ " to denote the set-theoretic complement.}.
    
\end{lemma}

\begin{proof}

Suppose $\mathfrak{X}\not\vDash\mu(\A_*, \overline{\mathfrak{D}_I}, \overline{\mathfrak{D}_M})$, by Proposition \ref{4.1.9}, there exists $f:\mathfrak{X}\rightarrow \A_*$ such that $f$ is surjective, continuous, stable and satisfies CDC for $\overline{\mathfrak{D}_I}$ and $\overline{\mathfrak{D}_M}$. Define $g:\sigma(\rho(\mathfrak{X}))\rightarrow \sigma(\rho(\mathfrak{\A_*}))$ as follows: $g(\rho(x))=\rho(f(x))$. Suppose $\rho(x)=\rho(y)$, then $xR_I y$ and $y R_I x$. As $f$ is stable, it follows that $f(x)R_I f(y)$ and $f(y) R_I f(x)$, and thus $f(x)\backsim f(y)$. Therefore, $g$ is well defined. As $f$ is surjective, so is $g$. Then we check that $g$ is stable, continuous and satisfies CDC for $\overline{\rho\mathfrak{D}_I},\overline{\rho\mathfrak{D}_M}$. 

Suppose $\rho(x)R_I \rho(y)$, then $\rho(x)\leq \rho(y)$ in $\rho(\mathfrak{X})$, and $xR_I y$ in $\mathfrak{X}$. As $f$ is stable, $f(x)R_I f(y)$. Thus $\rho(f(x))\leq \rho(f(y))$ in $\rho(\A_*)$, 
and $\rho(f(x))R_I \rho(f(y))$ in $\sigma(\rho(\A_*))$ by definition, namely $g(\rho(x))R_I g(\rho(y))$. Suppose $\rho(x) R_M \rho(y)$, then $\rho(x)[R_I\circ R_M \circ R_I]\rho(y)$ in $\rho(\mathfrak{X})$. Thus $xR_I\circ R_M \circ R_I y$ in $\mathfrak{X}$ by definition. There exist $x_1,x_2$ such that $xR_I x_1 R_M x_2 R_I y$. As $f$ is stable, $f(x)R_I f(x_1)R_M f(x_2) R_I f(y)$. Thus $f(x) R_I\circ R_M \circ R_I f(y)$, by definition $\rho(f(x))[R_I\circ R_M \circ R_M]\rho(f(y))$, namely $\rho(f(x))R_M \rho(f(y))$ in $\sigma(\rho(\A_*))$. Thus, $g(\rho(x))R_M g(\rho(y))$. This proves that $g$ is stable.

For any $\rho[F]\subseteq \rho(\A_*)$ where $F\subseteq \A_*$, we have that $x\in\rho^{-1}(g^{-1}(\rho[F]))$ iff $\rho(x)\in g^{-1}(\rho[F])$ iff $g(\rho(x))\in \rho[F]$ iff $\rho(f(x))\in \rho[F]$ iff $f(x)\in \rho^{-1}(\rho[F])$ iff $x\in f^{-1}(\rho^{-1}(\rho[F]))$. As $f$ is continuous, we know that $\rho^{-1}(g^{-1}(\rho[F]))=f^{-1}(\rho^{-1}(\rho[F]))$ is clopen in $\mathfrak{X}$. Clearly, $\rho^{-1}(g^{-1}(\rho[F]))$ does not cut any $R_I$-cluster in $\mathfrak{X}$. As $\sigma(\rho(\mathfrak{X}))$ has the quotient topology, $g^{-1}(\rho[F])=\rho[\rho^{-1}(g^{-1}(\rho[F]))]$ $ =\rho[f^{-1}(\rho^{-1}(\rho[F]))]$ is clopen in $\sigma(\rho(\mathfrak{X}))$. Therefore, $g$ is continuous.

Suppose $R_M[g(\rho(x))]\cap \rho[\delta]\not=\emptyset$ where $\delta\in \mathfrak{D}_M$, then there exists $a\in\delta$ such that $g(\rho(x))R_M \rho(a)$. Thus $\rho(f(x))R_M \rho(a)$ in $\sigma(\rho(\A_*))$, and $\rho(f(x))[R_I\circ R_M\circ R_I]\rho(a)$. By definition, $f(x) R_I\circ R_M\circ R_I a$ in $\A_*$. Since $\A_*$ validates $\textbf{Mix}$, $R_I\circ R_M\circ R_I=R_M$. Thus $f(x)R_M a$. As $f$ satisfies CDC for $\overline{\mathfrak{D}_M}$, there exists $y$ such that $xR_M y$ and $f(y)\in \delta$. As $R_I$ is reflexive, $xR_I\circ R_M \circ R_I y$, and by definition $\rho(x)[R_I\circ R_M\circ R_I]\rho(y)$ in $\rho(\A_*)$. Namely, $\rho(x)R_M \rho(y)$ in $\sigma(\rho(\A_*))$ by the definition of $\sigma$, and $g(\rho(y))=\rho(f(y))\in \rho[\delta]$. Thus $g[R_M[\rho(x)]]\cap \rho[\delta]\not=\emptyset$. This proves that $g$ satisfies CDC for $\overline{\rho\mathfrak{D}_M}$.

Suppose $R_I[g(\rho(x))]\cap \rho[\delta]\not=\emptyset$ where $\delta\in \mathfrak{D}_I$, then there exists $a\in\delta $ such that $g(\rho(x))R_I\rho(a)$. Thus $\rho(f(x))=g(\rho(x)\leq \rho(a)$ in $\rho(\A_*)$. By definition, $f(x)R_I a$. As $f$ satisfies CDC or $\overline{\mathfrak{D}_I}$, there exists $y$ such that $xR_I y$ and $f(y)\in \delta$. Then $\rho(x)\leq \rho(y)$ in $\rho(\mathfrak{X})$, and by definition of $\sigma$, we have that $\rho(x)R_I \rho(y)$ in $\sigma(\rho(\mathfrak{X}))$ and $g(\rho(y))=\rho(f(y))\in\rho[\delta]$. This proves that $g$ satisfies CDC for $\overline{\rho\mathfrak{D}_I}$.

Therefore, $\sigma(\rho(\mathfrak{X}))\not\vDash\mu(\sigma(\rho(\A_*)),\overline{\rho\mathfrak{D}_I},\overline{\rho\mathfrak{D}_M})$.
    
\end{proof}

Apart from the rule collapse lemma, we also need to show that the refutation conditions stated in Proposition \ref{3.3.12} work essentially the same way for intuitionistic modal Kripke frames.

\begin{theorem}
For any intuitionstic modal Kripke frame $(X,\leq,R)$, we have that $(X,\leq,R)\not\vDash \rho(\A,D^\rightarrow, D^\Box)$ iff there is a surjective stable order-preserving map $f:X\rightarrow X_A$ satisfying CDC$_\Box$ for any $\beta(a)$ where $a\in D^\Box$ and satisfies CDC$_\rightarrow$ for any $\beta(a)\setminus \beta(b)$ where $(a,b)\in D^\rightarrow$.      
\end{theorem}

\begin{proof}

See \ref{A}.
    
\end{proof}

Similarly (in fact, more easily), we get the following refutation conditions for bimodal Kripke frames.

\begin{theorem}
\label{4.2.6}
For any bimodal Kripke frame $(X,R_I,R_M)$, $(X,R_I,R_M)\not\vDash\mu(\A,D^I,D^M)$ iff there is a surjective relation-preserving (or stable) map $f:X\rightarrow X_A$ satisfying CDC$_\Box$ for any $\beta(a)$ and $\beta(b)$ where $a\in D^I$ and $b\in D^M$.
    
\end{theorem}

Then we prove the rule translation lemma, which is the counterpart to \cite[Lem. 2.56]{amaster} and shows that the G\"odel translation connects the stable canonical rules for intuitionistic modal logics and the stable canonical rules for bimodal logics. 

\begin{lemma}[Rule translation lemma]
\label{4.2.7}
For every stable canonical rule $\rho(\B,D^\rightarrow,D^\Box)$ and every $(S4\otimes K)\oplus Mix$-modal space $\mathfrak{X}$, we have that $\mathfrak{X}\vDash t(\rho(\B_*,\mathfrak{D}_\rightarrow,\mathfrak{D}_M))$ iff $\mathfrak{X}\vDash\mu(\sigma(\B_*),\mathfrak{D}_I,\mathfrak{D}_M)$ where $\mathfrak{D}_\rightarrow=\{\beta(a)\setminus\beta(b)\mid (a,b)\in D^\rightarrow\}$, $\mathfrak{D}_M=\{\beta(a)\mid a\in D^\Box\}$ and $\mathfrak{D}_I=\{\overline{\beta(a)}\cup \beta(b)\mid (a,b)\in D^\rightarrow\}$.
\end{lemma}

\begin{proof}

By Propositions \ref{4.1.13} and \ref{4.1.17}, it suffices to prove that $\rho(\mathfrak{X})\vDash\rho(\B_*,\mathfrak{D}_\rightarrow,\mathfrak{D}_M)$ iff $\mathfrak{X}\vDash\mu(\sigma(\B_*),\mathfrak{D}_I,\mathfrak{D}_M)$.

Suppose $\mathfrak{X}\not\vDash\mu(\sigma(\B_*),\mathfrak{D}_I,\mathfrak{D}_M)$, then there is a continuous stable surjection $f:\mathfrak{X}\rightarrow \sigma(\B_*)$ satisfying CDC$_\Box$ for any $\beta(a)\in \mathfrak{D}_M$ and for any $\beta(b)\in \mathfrak{D}_I$. Define $g:\rho(\mathfrak{X})\rightarrow \B_*$ by $g(\rho(x))=f(x)$. Suppose $x\backsim y$ in $\mathfrak{X}$, then $xR_I y$ and $yR_I x$. As $f$ is stable, $f(x) R_I f(y)$ and $f(y) R_I f(x)$ in $\sigma(\B_*)$, namely $f(x)\subseteq f(y)$ and $f(y)\subseteq f(x)$ in $\B_*$. Thus $f(x)=f(y)$, $g$ is well defined.

As $f$ is surjective, so is $g$. Suppose $\rho(x)\leq \rho(y)$ in $\rho(\mathfrak{X})$, then $xR_I y$ in $\mathfrak{X}$, as $f$ is stable, $f(x)R_I f(y)$ in $\sigma(\B_*)$, namely $f(x)\subseteq f(y)$ in $\B_*$. Thus $g(\rho(x))\subseteq g(\rho(y))$, $g$ is order-preserving. Suppose $\rho(x)[R_I\circ R_M\circ R_I]\rho(y)$, then there exist $x_1,x_2$ such that $xR_I x_1 R_M x_2 R_I y$ in $\mathfrak{X}$. As $f$ is stable, $f(x)R_If(x_1)R_M f(x_2)R_I f(y)$ in $\sigma(\B_*)$, namely $f(x)\subseteq f(x_1)R f(x_2)\subseteq f(y)$ in $\B_*$. As $\B_*$ is a modal Esakia space, by Proposition \ref{4.1.11}, $\subseteq \circ R\circ \subseteq = R$. Thus $f(x)Rf(y)$ in $\B_*$, namely $g(\rho(x))Rg(\rho(y))$. \textcolor{black}{This proves that} $g$ is stable.

For any $p\in \B_*$, we have that $x\in \rho^{-1}(g^{-1}(p))$ iff $g(\rho(x))=p$ iff $f(x)=p$ iff $x\in f^{-1}(p)$. As $f$ is continuous, $\rho^{-1}(g^{-1}(p))=f^{-1}(p)$ is clopen in $\mathfrak{X}$. Clearly, $\rho^{-1}(g^{-1}(p))$ does not cut any $R_I$-cluster, so $g^{-1}(p)$ is clopen in $\rho(\mathfrak{X})$ as the topology is the quotient topology. Thus $g$ is continuous.

Suppose ${\uparrow} g(\rho(x))\cap \beta(a)\setminus \beta(b)\not=\emptyset$ where $(a,b)\in D^\rightarrow$, then ${\uparrow} f(x)\cap \beta(a)\setminus \beta(b)\not=\emptyset$. Thus $R_I[f(x)]\not\subseteq \overline{\beta(a)}\cup\beta(b)$. As $f$ satisfies CDC$_\Box$ 
for $\overline{\beta(a)}\cup\beta(b)$, \textcolor{black}{we have that} $f[R_I[x]]\not\subseteq \overline{\beta(a)}\cup\beta(b)$. There exists $xR_I y$ such that $f(y)\in \beta(a)$ while $f(y)\not\in\beta(b)$. Then $\rho(x)\leq \rho(y)$ in $\rho(\mathfrak{X})$, $g(\rho(y))=f(y)\in\beta(a)\setminus\beta(b)$, and thus $g[{\uparrow}\rho(x)]\cap \beta(a)\setminus\beta(b)\not=\emptyset$. Therefore, $g$ satisfies CDC$_\rightarrow$ any $\beta(a)\setminus\beta(b)$ where $(a,b)\in D^\rightarrow$.

Suppose $g[[R_I\circ R_M\circ R_I][\rho(x)]]\subseteq \beta(a)$ where $a\in D^\Box$, then for any $xR_M y$ in $\mathfrak{X}$, as $R_I$ is reflexive, $xR_I\circ R_M\circ R_I y$. By definition, $\rho(x)[R_I\circ R_M\circ R_I]\rho(y)$ in $\rho(\mathfrak{X})$. As $g[[R_I\circ R_M\circ R_I][\rho(x)]]\subseteq \beta(a)$, it follows that $g(\rho(y))=f(y)\in\beta(a)$. Thus $f[R_M[x]]\subseteq \beta(a)$, and as $f$ satisfies CDC$_\Box$ for $\beta(a)$, \textcolor{black}{we have that}  $R_M[f(x)]\subseteq \beta(a)$, namely $R[g(\rho(x))]\subseteq \beta(a)$ in $\B_*$. Therefore, $g$ satisfies CDC$_\Box$ for any $\beta(a)$ where $a\in D^\Box$. By Proposition \ref{3.3.12}, $\rho(\mathfrak{X})\not\vDash\rho(\B_*,\mathfrak{D}_\rightarrow,\mathfrak{D}_M)$.

For the other direction, suppose $\rho(\mathfrak{X})\not\vDash\rho(\B_*,\mathfrak{D}_\rightarrow,\mathfrak{D}_M)$, by Proposition \ref{3.3.12}, there is a surjective stable Priestley morphism $f: \rho(\mathfrak{X})\rightarrow \B_*$ satisfying CDC$\rightarrow$ for any $\beta(a)\setminus\beta(b)$ where $(a,b)\in D^\rightarrow$  and CDC$_\Box$ for any $\beta(a)$ where $a\in D^\Box$. Define $g:\mathfrak{X}\rightarrow \sigma(\B_*)$ as follows: $g(x)=f(\rho(x))$. 

As $f$ is surjective, so is $g$. Suppose $xR_I y$ in $\mathfrak{X}$, then $\rho(x)\leq \rho(y)$ in $\rho(\mathfrak{X})$. As $f$ is order-preserving, $f(\rho(x))\subseteq f(\rho(y))$, namely $f(\rho(x))R_I f(\rho(y))$ in $\sigma(\B_*)$, and $g(x)R_I g(y)$. Thus $g$ is relation-preserving w.r.t $R_I$. Suppose $xR_M y$ in $\mathfrak{X}$, as $R_I$ is reflexive, $x R_I\circ R_M\circ R_I y$. By definition, $\rho(x)[R_I\circ R_M\circ R_I]\rho(y)$ in $\rho(\mathfrak{X})$. As $f$ is stable, $f(\rho(x))Rf(\rho(y))$ in $\B_*$, namely $g(x)R_M g(y)$ in $\sigma(\B_*)$. $g$ is thus relation-preserving w.r.t $R_M$.

For any $p\in\sigma(\B_*)$, $g^{-1}(p)=\rho^{-1}(f^{-1}(p))$. As $f$ is continuous, $f^{-1}(p)$ is clopen in $\rho(\mathfrak{X})$. As $\rho(\mathfrak{X})$ has the quotient topology, $\rho^{-1}(f^{-1}(p))$ is clopen in $\mathfrak{X}$. Thus $g$ is continuous.

Suppose $g[R_I[x]]\subseteq \overline{\beta(a)}\cup \beta(b)$ where $(a,b)\in D^\rightarrow$, for any $\rho(x)\leq \rho(y)$ in $\rho(\mathfrak{X})$, $xR_I y$ in $\mathfrak{X}$, and $f(\rho(y))=g(y)\in\overline{\beta(a)}\cup \beta(b)$. Thus $f[{\uparrow} \rho(x)]\cap \beta(a)\setminus\beta(b)=\emptyset$. As $f$ satisfies CDC$_\rightarrow$ for $\beta(a)\setminus\beta(b)$, \textcolor{black}{we have that} ${\uparrow} f(\rho(x))\cap \beta(a)\setminus\beta(b)=\emptyset$. Thus $R_I[g(x)]\cap \beta(a)\setminus\beta(b)=\emptyset$, namely $R_I[g(x)]]\subseteq\overline{\beta(a)}\cup \beta(b)$. \textcolor{black}{This proves that} $g$ satisfies CDC$_\Box$ for any element in $\mathfrak{D}_I$.

Suppose $g[R_M[x]]\subseteq \beta(a)$ where $a\in D^\Box$. For any $\rho(x)[R_I\circ R_M\circ R_I]\rho(y)$ in $\rho(\mathfrak{X})$, by definition $xR_I\circ R_M \circ R_I y$ in $\mathfrak{X}$. As $\mathfrak{X}$ validates $\textbf{Mix}$, $R_I\circ R_M\circ R_I=R_M$, and thus $xR_My$. As $g[R_M[x]]\subseteq \beta(a)$, it follows that $f(\rho(y))=g(y)\in \beta(a)$. Thus $f[[R_I\circ R_M\circ R_I][\rho(x)]]\subseteq \beta(a)$. As $f$ satisfies CDC$_\Box$ for $\beta(a)$, \textcolor{black}{we have that} $R[f(\rho(x))]\subseteq \beta(a)$. Thus $R_M[g(x)]=R[f(\rho(x))]\subseteq \beta(a)$. \textcolor{black}{This proves that} $g$ satisfies CDC$_\Box$ for any $\beta(a)$ where $a\in D^\Box$. Therefore, $\mathfrak{X}\not\vDash\mu(\sigma(\B_*),\mathfrak{D}_I,\mathfrak{D}_M)$ by Proposition \ref{4.1.9}.

\end{proof}

We can also state the above lemma in algebraic terms just like \cite[Lem. 2.56]{amaster}: for every stable canonical rule $\rho(\B,D^\rightarrow,D^\Box)$ and every $(S4\otimes K)\oplus Mix$-algebra $\A$, we have that $\A\vDash t(\rho(\B,D^\rightarrow,D^\Box))$ iff $\A\vDash\mu(\sigma(\B),D^I, D^M)$ where $D^I=\{\neg a\lor b\mid (a,b)\in D^\rightarrow\}$ and $D^M=D^\Box$.

Now, we can finish the proof of the Dummett-Lemmon conjecture which is the main result of this section.

\begin{theorem}
For any $L\in  \mathbf{NExt}(\textbf{IntK$^R_\Box$})$, $L$ is Kripke complete iff $\tau(L)$ is Kripke complete. 
\end{theorem}

\begin{proof}

The right-to-left direction is given by Proposition \ref{4.2.3}. For the other direction, let $L\in  \mathbf{NExt}(\textbf{IntK$^R_\Box$})$ be arbitrary, suppose $L$ is Kripke complete and $\Gamma/\Delta\not\in \tau(L)$ where $\Gamma/\Delta$ is a bimodal multi-conclusion rule. By Proposition \ref{4.1.6}, we can assume that $\Gamma/\Delta=\mu(\A_*, \mathfrak{D}_I,\mathfrak{D}_M)$ where $\A$ is a finite $(S4\otimes K)\oplus Mix$-algebra and $\mathfrak{D}_I,\mathfrak{D}_M\subseteq \mathcal{P}(\A_*)$. By Corollary \ref{combi} and Theorem \ref{4.1.8}, there exists an $(S4\otimes K)\oplus Mix$-modal space $\mathfrak{X}$ such that $\mathfrak{X}\vDash\tau(L)$ and $\mathfrak{X}\not\vDash\mu(\A_*, \mathfrak{D}_I,\mathfrak{D}_M)$. By Lemma \ref{lemma 4.2.4}, $\sigma(\rho(\mathfrak{X}))\not\vDash\mu(\sigma(\rho(\A_*)),\overline{\rho\overline{\mathfrak{D}_I}},\overline{\rho\overline{\mathfrak{D}_M}})$ where $\overline{\rho\overline{\mathfrak{D}_I}}=\{\overline{\rho[\Bar{\delta}]}\mid \delta\in \mathfrak{D}_I\}$ and $\overline{\rho\overline{\mathfrak{D}_M}}=\{\overline{\rho[\Bar{\delta}]}\mid \delta\in \mathfrak{D}_M\}$. Note that $\sigma(\rho(\mathfrak{X}))$ is a $(Grz\otimes K)\oplus Mix$-modal space.

Now by Theorem \ref{4.1.24}, there exists an $L'\in  \mathbf{NExt}(\textbf{IntK$^R_\Box$})$ such that $\sigma(L')=(\textbf{Grz}\otimes \textbf{K})\oplus \textbf{Mix}^R\oplus \mu(\sigma(\rho(\A_*)),\overline{\rho\overline{\mathfrak{D}_I}},\overline{\rho\overline{\mathfrak{D}_M}})$, namely $(\textbf{Grz}\otimes \textbf{K})\oplus \textbf{Mix}^R\oplus \mu(\sigma(\rho(\A_*)),\overline{\rho\overline{\mathfrak{D}_I}},\overline{\rho\overline{\mathfrak{D}_M}})=(\textbf{Grz}\otimes \textbf{K})\oplus \textbf{Mix}^R\oplus\{t(\Gamma/\Delta)\mid \Gamma/\Delta\in L'\}$. Therefore, $\sigma(\rho(\mathfrak{X}))\not\vDash t(\Gamma/\Delta)$ for some $\Gamma/\Delta\in L'$. By Theorem \ref{3.1.9}, we can assume that $\Gamma/\Delta$ is a stable canonical rule of the form $\rho(\B,D^\rightarrow, D^\Box)$. Thus $\sigma(\rho(\mathfrak{X}))\not\vDash t(\rho(\B,D^\rightarrow, D^\Box))$. By Proposition \ref{4.1.17} and Proposition \ref{4.1.14}, $\rho(\mathfrak{X})\not\vDash\rho(\B,D^\rightarrow, D^\Box)$. As $\mathfrak{X}\vDash \tau(L)$, by Proposition \ref{4.1.17}, $\rho(\mathfrak{X})\vDash L$. Therefore,   $\rho(\B,D^\rightarrow, D^\Box)\not\in L$. As $L$ is Kripke complete, there is an intuitionistic modal Kripke frame $\mathfrak{F}=(X,\leq,R)$ such that $\mathfrak{F}\vDash L$ while $\mathfrak{F}\not\vDash \rho(\B,D^\rightarrow, D^\Box)$. Viewing $\mathfrak{F}$ as an $(S4\otimes K)\oplus Mix$-bimodal Kripke frame, as $\rho(\mathfrak{F})=\mathfrak{F}\vDash L$, by Proposition \ref{4.2.2}, $\mathfrak{F}\not\vDash t(\rho(\B,D^\rightarrow, D^\Box))$ and $\mathfrak{F}\vDash \tau(L)$.

\vspace{1mm}
Then we prove that $(\textbf{S4}\otimes \textbf{K})\oplus \textbf{Mix}^R\oplus\{t(\Gamma/\Delta)\mid \Gamma/\Delta\in L'\}\subseteq(\textbf{S4}\otimes \textbf{K})\oplus \textbf{Mix}^R\oplus \mu(\sigma(\rho(\A_*)),\overline{\rho\overline{\mathfrak{D}_I}},\overline{\rho\overline{\mathfrak{D}_M}})$: let $\mathfrak{X}$ be an arbitrary $(S4\otimes K)\oplus Mix$-modal space, suppose $\mathfrak{X}\not\vDash t(\Gamma/\Delta)$ for some $\Gamma/\Delta \in L'$. By Theorem \ref{3.1.9}, we can assume $\Gamma/\Delta$ is a stable canonical rule of the form $\rho(\C_*,\mathfrak{D'}_\rightarrow,\mathfrak{D}'_M)$. As $\mathfrak{X}\not\vDash t(\rho(\C_*,\mathfrak{D}'_\rightarrow,\mathfrak{D}'_M))$, by Lemma \ref{4.2.7}, $\mathfrak{X}\not\vDash\mu(\sigma(\C_*),\mathfrak{D}'_I,\mathfrak{D}'_M)$. By Lemma \ref{lemma 4.2.4} and the fact that $\sigma(\rho(\sigma(\C_*)))=\sigma(\C_*)$, it follows that  $\sigma(\rho(\mathfrak{X}))\not\vDash\mu(\sigma(\C_*),\mathfrak{D}'_I,\mathfrak{D}'_M)$. By Lemma \ref{4.2.7}, $\sigma(\rho(\mathfrak{X}))\not\vDash t(\rho(\C_*,\mathfrak{D}'_\rightarrow,\mathfrak{D}'_M))$. As $\sigma(\rho(\mathfrak{X}))$ is a $(Grz\otimes K)\oplus Mix$-modal space and $(\textbf{Grz}\otimes \textbf{K})\oplus \textbf{Mix}^R\oplus \mu(\sigma(\rho(\A_*)),\overline{\rho\overline{\mathfrak{D}_I}},\overline{\rho\overline{\mathfrak{D}_M}})=(\textbf{Grz}\otimes \textbf{K})\oplus \textbf{Mix}^R\oplus\{t(\Gamma/\Delta)\mid \Gamma/\Delta\in L'\}$, we get that $\sigma(\rho(\mathfrak{X}))\not\vDash \mu(\sigma(\rho(\A_*)),\overline{\rho\overline{\mathfrak{D}_I}},\overline{\rho\overline{\mathfrak{D}_M}})$. By Proposition \ref{4.1.15} and the duality, $\mathfrak{X}\not\vDash\mu(\sigma(\rho(\A_*)),\overline{\rho\overline{\mathfrak{D}_I}},\overline{\rho\overline{\mathfrak{D}_M}})$. This proves that for any $(S4\otimes K)\oplus Mix$-modal space, if $\mathfrak{X}\not\vDash (\textbf{S4}\otimes \textbf{K})\oplus \textbf{Mix}^R\oplus\{t(\Gamma/\Delta)\mid \Gamma/\Delta\in L'\}$, then $\mathfrak{X}\not\vDash\mu(\sigma(\rho(\A_*)),\overline{\rho\overline{\mathfrak{D}_I}},\overline{\rho\overline{\mathfrak{D}_M}})$. By Theorems \ref{2.4.7} and \ref{4.1.8}, $(\textbf{S4}\otimes \textbf{K})\oplus \textbf{Mix}^R\oplus\{t(\Gamma/\Delta)\mid \Gamma/\Delta\in L'\}\subseteq(\textbf{S4}\otimes \textbf{K})\oplus \textbf{Mix}^R\oplus \mu(\sigma(\rho(\A_*)),\overline{\rho\overline{\mathfrak{D}_I}},\overline{\rho\overline{\mathfrak{D}_M}})$.

\vspace{1mm}
Now, since $\mathfrak{F}\not\vDash t(\rho(\B,D^\rightarrow, D^\Box))$ where $\rho(\B,D^\rightarrow, D^\Box)\in L'$, we have that $\mathfrak{F}\not \vDash\mu(\sigma(\rho(\A_*)),\overline{\rho\overline{\mathfrak{D}_I}},\overline{\rho\overline{\mathfrak{D}_M}})$. Therefore, by Theorem \ref{4.2.6}, there is a surjective stable map $f:X\rightarrow \sigma(\rho(\A_*))$ satisfying CDC for any element in $\overline{\rho\overline{\mathfrak{D}_I}}$ and $\overline{\rho\overline{\mathfrak{D}_M}}$. We construct a new bimodal Kripke frame  $\mathfrak{F}'=(X',R_I,R_M)$ as follows: for any $x\in X$, say $f(x)=\{a_1,...,a_n\}$ where $a_1,...,a_n\in \A_*$ (an $R_I$-cluster), replace $x$ by $n$ many copies of it, say $x_1,...,x_n$, then let $X'$ be the set of all such elements. For any $x_i,y_j\in X'$ ($x_i$ is a copy of $x$ and $y_j$ is a copy of $y$), define $x_iR_Iy_j$ iff $x\leq y$ (in $\mathfrak{F}$), and $x_iR_M y_j$ iff $xR_My$. It is easy to check that $\mathfrak{F}'$ is an $(S4\otimes K)\oplus Mix$-Kripke frame. As $\rho(\mathfrak{F}')=\mathfrak{F}=\rho(\mathfrak{F})\vDash L$, by Proposition \ref{4.2.2}, $\mathfrak{F}'\vDash\tau(L)$.

We then define $g:\mathfrak{F}'\rightarrow \A_*$ as follows: $g(x_i)=a_i$ where $x_i$ is copy of $x$ and $f(x)=\{a_1,...,a_n\}$. As $f$ is surjective , $g$ is surjective by the construction of $X'$. For any $x_i, y_j\in X'$(say $g(x_i)=a_i$, $g(y_j)=b_j$), suppose $x_iR_I y_j$ in $\mathfrak{F}'$, then $x\leq y$ in $\mathfrak{F}$ (or $x R_I y$ in $\mathfrak{F}$ when viewed $\mathfrak{F}$ as a bimodal Krikpe frame). As $f$ is stable, $f(x)R_If(y)$ in $\sigma(\rho(\A_*))$, namely $f(x)\leq f(y)$ in $\rho(\A_*)$. As $a_i$ is an element of $f(x)$ and $b_j$ is an element of $f(y)$ by the construction, $a_i R_I b_j$ in $\A_*$, namely $g(x_i)R_I g(y_j)$. Suppose $x_i R_M y_j$, then $xR_M y$ in $\mathfrak{F}$. As $f$ is stable, $f(x)[R_I\circ R_M \circ R_M]f(y)$. As $\A_*$ validates $\textbf{Mix}$, $R_I\circ R_M\circ R_I=R_M$, and thus $f(x)[R_M]f(y)$. We have that $a_i R_M b_j$, namely $g(x_i)R_M g(b_j)$. Therefore, $g$ is stable.

Suppose $R_M[g(x_i)]\cap \Bar{\delta}\not=\emptyset $ where $\delta\in \mathfrak{D}_M$, there exists $p\in\Bar{\delta}$ such that $g(x_i)R_M p$. Thus $\rho(p)\in \rho[\Bar{\delta}]$, and $f(x)=\rho(g(x_i))[R_M]\rho(p)$. Since $f$ satisfies CDC for $\overline{\rho[\Bar{\delta}]}$, there exists $z\in \mathfrak{F}$ such that $xR_M z$ and $f(z)\in\rho[\Bar{\delta}]$. By the construction, there exists $z_j$ such that $g(z_j)\in \Bar{\delta}$. Then as $x_iR_Mz_j$, $g[R_M[x_i]]\cap \Bar{\delta}\not=\emptyset$. Thus $g$ satisfies CDC for any element in $\mathfrak{D}_M$. Similarly, we can prove that $g$ satisfies CDC for any element in $\mathfrak{D}_I$. Therefore, by Theorem \ref{4.2.6}, $\mathfrak{F}'\not\vDash\mu(\A_*, \mathfrak{D}_I,\mathfrak{D}_M )$. As $\mathfrak{F}'\vDash\tau(L)$, this proves that $\tau(L)$ is Kripke complete.
    
\end{proof}

\textcolor{black}{The above theorem is not only interesting for its own sake, but also strengthens the connection between the lattice of intuitionistic modal logics and the lattice of normal extensions of the bimodal logic $(\mathbf{S4}\otimes\mathbf{K})\oplus\mathbf{Mix}$. It improves our toolkit when we try to study intuitionistic modal logics or bimodal logics via the G\"odel translation. In particular,
it allows us to reduce the problem about Kripke completeness of an intuitionistic modal  multi-conclusion consequence relation to the same problem about a bimodal multi-conclusion consequence relation and vice versa.}

\section{Conclusion}

In this paper we developed the method of stable canonical formulas and rules for intuitionistic modal logics. We proved that every intuitionistic modal multi-conclusion consequence relation is axiomatizable by stable canonical rules. Besides, \textcolor{black}{our rules have proved useful}. In particular, with some adaptations to known techniques from \cite{zmo} and \cite{amaster, an}, we provided an alternative proof of the Blok-Esakia theorem and \textcolor{black}{proved an analogue of the Dummett-Lemmon conjecture for intuitionistic modal logics}. Therefore, we hope that our work will contribute to the theory of stable canonical rules and also provide a uniform method to study intuitionistic modal logics.

For reasons of space, we did not include stable canonical rules for intuitionistic modal logics with both $\Box$ and $\Diamond$ in this paper. This can be done for $\textbf{IntK$_{\Box,\Diamond}$}$  (i.e. $\textbf{IntK$_\Box$}$ extended with $\Diamond$ which satisfies $\Diamond \bot=\bot$ and $\Diamond (p\lor q)=\Diamond p\lor \Diamond q$) with some adaptations to the techniques in Section 3. In particular, we can prove the corresponding completeness theorem with no difficulty and use  the duality in \cite{Al} to give a dual description of the rules. However, once we require some interactions between $\Box$ and $\Diamond$, \textcolor{black}{it is unclear} how to do filtrations in the counterpart to Proposition \ref{3.1.3}. \textcolor{black}{Solutions to some open problems about the finite model property of some intuitionistic modal logics with both $\Box$ and $\Diamond$ may help for this type of questions.}

Besides, \textcolor{black}{using similar methods, one can also develop stable canonical rules for Heyting-Lewis logics as well, which are introduced as superintuitionistic logics with a strict implication by Litak and Visser in \cite{lit} . This may not be surprising as implications can often be viewed as implicit modal operators. In fact,} our proof strategy of the Blok-Esakia theorem for intuitionistic modal logics also leads us to a gap in the proof of the Blok-Esakia theorem for Heyting-Lewis logics given in \cite{litak}\footnote{\textcolor{black}{In the proof of \cite[Lem. 4.18]{litak}, the induction step for formulas of the form $\Box_M\psi$ does not work. This lemma corresponds to a version of Lem 4.19 for Heyting-Lewis logics which is equivalent to the Blok-Esakia theorem. See \cite[Chap 5]{cmaster} for more details.}}. As far as we know, so far whether the Blok-Esakia theorem holds for Heyting-Lewis logics remains an open problem. 

Lastly, the applications of stable canonical rules in this paper are far from exhaustive. It is highly likely that we could apply our stable canonical rules to some other topics in intuitionistic modal logics. One such example is given in \cite{nadm} where stable canonical rules for superintuitionistic logics and transitive modal logics were used to give an alternative proof of decidability of admissibility of these logics. Since whether there is an algorithm to decide the admissibility of a rule remains an open problem for most intuitionistic modal logics, we hope and believe that our stable canonical rules for intuitionistic modal logics may also help us attacking these questions.

\section*{Acknowledgment}
The author thanks the referees for many useful suggestions which have improved the paper. The author is very grateful to his master supervisor Nick Bezhanishvili for his guidance that led to the first draft of this paper. The author is also thankful to Rodrigo Almeida and 
Antonio Maria Cleani for helpful comments and interesting discussions.

\appendix

\section{Applications of Stable Canonical Rules for Intuitionistic Modal
Logics}
\label{sec:sample:appendix}

\subsection{Proof of Theorem 4.29}
\label{A}
For any intuitionistic modal Kripke frame $\mathcal{X}=(X,\leq,R)$, $\mathcal{X}\vDash\rho(\A,D^\rightarrow, D^\Box)$ iff $\mathcal{X}^*\vDash\rho(\A,D^\rightarrow, D^\Box)$ iff $(\mathcal{X}^*)_*\vDash \rho(\A,D^\rightarrow, D^\Box)$.

Suppose $\mathcal{X}=(X,\leq,R)\not\vDash\rho(\A,D^\rightarrow, D^\Box)$, then $(\mathcal{X}^*)_*\not\vDash \rho(\A,D^\rightarrow, D^\Box)$. By Proposition \ref{3.3.12}, there is a surjective stable Priestley morphism $g:(\mathcal{X}^*)_*\rightarrow X_A$ satisfying CDC$_\Box$ for any $\beta(a)$ where $a\in D^\Box$ and CDC$_\rightarrow$ for any $\beta(a)\setminus \beta(b)$ where $(a,b)\in D^\rightarrow$. Define $\epsilon:\mathcal{X}\rightarrow (\mathcal{X}^*)_*$ as follows: $\epsilon(x)=\{{{\uparrow}} x\subseteq U\mid U \text{ is an upset of }X\}$ (It is easy to see that $\epsilon(x)$ is a prime filter of upsets of $X$). Then we check that $g\circ\epsilon$ satisfies all the expected conditions.

For any $x\leq y$ in $\mathcal{X}$, ${{\uparrow}} y\subseteq {{\uparrow}} x$, and thus $\epsilon(x)\subseteq \epsilon(y)$, $g\circ\epsilon (x)\leq g\circ\epsilon (y)$. $g\circ\epsilon$ is order-preserving. Suppose $xRy$ in $\mathcal{X}$, then for any $\Box_R U\in \epsilon(x)$ where $U$ is an upset, by definition, ${{\uparrow}} x\subseteq \Box_R U$, and thus $x\in \Box_R U$, namely $R[x]\subseteq U$. Therefore, $y\in U$. As $U$ is an upset, ${{\uparrow}} y\subseteq U$, and $U\in \epsilon(y)$. This means that $\epsilon(x)R^*\epsilon(y)$ in $(\mathcal{X}^*)_*$. As $g$ is stable, $g(\epsilon(x))R_A g(\epsilon(y))$ in $\A_*$, and thus $g\circ \epsilon$ is stable.

For any $p\in X_A$, as $g$ is continuous, we know that $g^{-1}(p)$ is a clopen set of $(\mathcal{X}^*)_*$. As $(\mathcal{X}^*)_*$ is an Esakia space if we dismiss $R^*$, by Remark \ref{2.6.6} we have that $g^{-1}(p)=\bigcup_{1\leq i\leq n}(\beta(U_i)\setminus\beta(V_i))$ where $U_i's,V_i's$ are upsets of $X$. Since $g$ is surjective, $g^{-1}(p)$ is not empty, there exists $1\leq i\leq n$ such that $\beta(U_i)\setminus \beta(V_i)\not=\emptyset$. Thus $U_i\setminus V_i\not=\emptyset$, there exists $x\in U_i\setminus V_i$. Then ${\uparrow} x\subseteq U_i$ while ${\uparrow} x\not \subseteq V_i$. Thus $\epsilon(x)\in \beta(U_i)$ while $\epsilon(x)\not\in \beta(V_i)$. Therefore, $\epsilon(x)\in g^{-1}(p)$, and $g(\epsilon(x))=p$. As $p\in X_A$ is arbitrary, this proves that $g\circ \epsilon$ is surjective.

Let $x\in X$ be arbitrary, suppose $R_A[g(\epsilon(x))]\not\subseteq\beta(a)$ while $(g\circ \epsilon)[R[x]]\subseteq \beta(a)$ where $a\in D^\Box$. As $g$ satisfies CDC$_\Box$ for $\beta(a)$, \textcolor{black}{we have that} $g[R^*[\epsilon(x)]]\not\subseteq\beta(a)$, there exists $\epsilon(x)R^* \mathfrak{q}$ such that $g(\mathfrak{q})\not\in \beta(a)$, namely $a\not\in g(\mathfrak{q})$. As $(\mathcal{X}^*)_*\not\vDash \rho(\A,D^\rightarrow, D^\Box)$, we know that $\mathcal{X}^*\not\vDash \rho(\A,D^\rightarrow, D^\Box)$. By Proposition \ref{3.1.6}, there is a stable bounded lattice embedding $h: A\rightarrow Up(X)$ satisfying CDC for $D^\rightarrow$ and $D^\Box$. By duality, we can assume that $g=h^{-1}$, and thus there exists an upset $U$ of $X$ (in fact just $h(a)$) such that for any prime filter $\mathfrak{p}$ of upsets of $X$ (i.e any element of $(\mathcal{X}^*)_*$) $a\in g(\mathfrak{p})$ iff $U\in\mathfrak{p}$. As $(g\circ \epsilon)[R[x]]\subseteq \beta(a)$, for any $xRy$ in $X$, we have that $g(\epsilon(y))\in\beta(a)$, namely $a\in g(\epsilon(y))$ Thus $U\in \epsilon(y)$, and $y\in{\uparrow} y\subseteq  U$. Therefore, $R[x]\subseteq U$. As $\epsilon(x) R^*\mathfrak{q}$,  for any $\Box_R V\in\epsilon(x)$, it follows that $V\in \mathfrak{q}$ where $V$ is an upset of $X$. For any $x\leq z$ in $X$, \textcolor{black}{we have that} $R[z]\subseteq R[x]\subseteq U$ as $\leq \circ R=R$ in $\mathcal{X}$. Thus ${\uparrow} x\subseteq \Box_R U$, and by definition $\Box_R U\in \epsilon(x)$. Thus $U\in \mathfrak{q}$ and $a\in g(\mathfrak{q})$, contradicting the assumption that $a\not\in  g(\mathfrak{q})$. Therefore, for any $x\in X$, we have that $(g\circ \epsilon)[R[x]]\subseteq \beta(a)$ implies that $R_A[g(\epsilon(x))]\subseteq\beta(a)$ where $a\in D^\Box$. \textcolor{black}{This proves that} $g\circ \epsilon$ satisfies CDC$_\Box$ for any $\beta(a)$ where $a\in D^\Box$.

Let $x\in X$ be arbitrary, suppose ${\uparrow} g(\epsilon(x))\cap \beta(a)\setminus \beta(b)\not=\emptyset$ where $(a,b)\in D^\rightarrow$, as $g$ satisfies CDC$_\rightarrow$ for $\beta(a)\setminus \beta(b)$, \textcolor{black}{we have that} $g[{\uparrow} \epsilon(x)]\cap \beta(a)\setminus\beta(b)\not=\emptyset$. Thus there exists a prime filter $\mathfrak{p}$ of upsets of $X$ such that $\epsilon(x)\subseteq \mathfrak{p}$ and $g(\mathfrak{p})\in \beta(a)\setminus\beta(b)$, namely $a\in g(\mathfrak{p})$ and $b\not\in g(\mathfrak{p})$. As we have proved above, there exist upsets $U_a$, $U_b$ of $X$ such that for any prime filters $\mathfrak{q}$ of upsets of $X$ (i.e any element of $(\mathcal{X}^*)_*$) $a\in g(\mathfrak{q})$ iff $U_a\in\mathfrak{q}$, and $b\in g(\mathfrak{q})$ iff $U_b\in\mathfrak{q}$. Thus $U_a\in\mathfrak{p}$ and $U_b\not\in\mathfrak{p}$. Suppose $(g\circ\epsilon)[{\uparrow} x]\cap \beta(a)\setminus\beta(b)=\emptyset$, then for any $x\leq y$, \textcolor{black}{we have that} $U_a\in \epsilon(y)$ implies that $U_b\in\epsilon(y)$, namely $y\in U_a$ implies that $y\in U_b$. Thus ${\uparrow} x\subseteq U_a\rightarrow U_b$. As $\epsilon(x)\subseteq \mathfrak{p}$, \textcolor{black}{we get that} $U_a\rightarrow U_b\in\mathfrak{p}$, contradicting the fact that $U_a\in\mathfrak{p}$ and $U_b\not\in\mathfrak{p}$. Therefore, $(g\circ\epsilon)[{\uparrow} x]\cap \beta(a)\setminus\beta(b)\not=\emptyset$. This proves that $g\circ\epsilon$ satisfies CDC$_\rightarrow$ for $\beta(a)\setminus \beta(b)$ where $(a,b)\in D^\rightarrow$. Therefore, $g\circ \epsilon: X\rightarrow X_A$ is a surjective stable order-preserving map satisfying CDC$_\Box$ for any $\beta(a)$ where $a\in D^\Box$ and satisfies CDC$_\rightarrow$ for any $\beta(a)\setminus \beta(b)$ where $(a,b)\in D^\rightarrow$.

For the other direction, suppose there is a surjective stable order-preserving map $f:X\rightarrow X_A$ satisfying CDC$_\Box$ for any $\beta(a)$ where $a\in D^\Box$ and satisfies CDC$_\rightarrow$ for any $\beta(a)\setminus \beta(b)$ where $(a,b)\in D^\rightarrow$. 

As $\A$ is finite, every prime filter of $\A$ (i.e elements of $X_A$) is principal and is given by a join-irreducible element. For each join-irreducible element of $a$ of $\A$, we denote the prime filter corresponding to it by $P_{a}$. Define $g:(\mathcal{X}^*)_*\rightarrow X_A$ as follows: $g(\mathfrak{p})=\{a\in A\mid {\uparrow} f^{-1}(P_{a_i})\in \mathfrak{p} \text{ for some join-irreducible element } a_i\leq a \}$. Clearly, $0\not\in g(\mathfrak{p})\not=\emptyset$ and $g(\mathfrak{p})$ is upward-closed. Suppose $a,b\in g(\mathfrak{p})$, then there exist join-irreducible elements $a_i,b_j\in A$ such that $a_i\leq a$, $b_j\leq b$, ${\uparrow} f^{-1}(P_{a_i})\in\mathfrak{p}$ and ${\uparrow} f^{-1}(P_{b_j})\in\mathfrak{p}$. As $\mathfrak{p}$ is a prime filter, ${\uparrow} f^{-1}(P_{a_i})\cap {\uparrow} f^{-1}(P_{b_j})\not=\emptyset$, say $x\in {\uparrow} f^{-1}(P_{a_i})\cap {\uparrow} f^{-1}(P_{b_j})$. There exist $y\leq x, z\leq x$ in $X$ such that $f(y)=P_{a_i}$ and $f(z)=P_{b_j}$. As $f$ is order-preserving, \textcolor{black}{we have that} $P_{a_i}, P_{b_j}\subseteq f(x)$, and thus $a_i,b_j\in f(x)$. As $f(x)$ is a prime filter of $\A$, it follows that $a_i\land b_j\not=0$. Let $c_1,...,c_n$ enumerate all the join-irreducible elements of $\A$ less than or equal to $a_i\land b_j$ (note $1\leq n$ as $a_i\land b_j\not=0$). For any $x'\in {\uparrow} f^{-1}(P_{a_i})\cap {\uparrow} f^{-1}(P_{b_j})$, as we have shown $P_{a_i}\subseteq f(x')$ and $P_{b_j}\subseteq f(x')$. Thus $f(x')=P_{c_k}$ for some $1\leq k\leq n$. Since $x'\leq x'$, $x'\in {\uparrow} f^{-1}(P_{c_1})\cup...\cup {\uparrow} f^{-1}(P_{c_n})$. As $x'$ is arbitrary, this proves that ${\uparrow} f^{-1}(P_{a_i})\cap {\uparrow} f^{-1}(P_{b_j})\subseteq {\uparrow} f^{-1}(P_{c_1})\cup...\cup {\uparrow} f^{-1}(P_{c_n})$. As ${\uparrow} f^{-1}(P_{a_i})\in\mathfrak{p}$, ${\uparrow} f^{-1}(P_{b_j})\in\mathfrak{p}$ and $\mathfrak{p}$ is a prime filter, ${\uparrow} f^{-1}(P_{c_m})\in\mathfrak{p}$ for some $1\leq m\leq n$. As $c_m\leq a_i\land b_j\leq a\land b$ is join-irreducible, $a\land b\in g(\mathfrak{p})$. This proves that $g(\mathfrak{p})$ is a proper filter. By the definition of join-irreducibility, it is easy to see that $g(\mathfrak{p})$ is also prime. Therefore, $g$ is well defined.

Clearly, if $\mathfrak{p}\subseteq \mathfrak{q}$, then $g(\mathfrak{p})\subseteq g(\mathfrak{q})$, and thus $g$ is order-preserving. Suppose $f(x)=P_{a}$ where $x\in X$ and $a\in\A$ is join-irreducible. Then by definition ${\uparrow} f^{-1}(P_{a})\in \epsilon(x)$, and thus $P_{a}\subseteq g(\epsilon(x))$. Now suppose ${\uparrow} f^{-1}(P_b)\in \epsilon(x)$ where $b$ is join-irreducible, then ${\uparrow} x\subseteq {\uparrow} f^{-1}(P_b)$, there exists $y\leq x$ such that $f(y)=P_b$. As $f$ is order-preserving, $f(y)=P_b\subseteq f(x)=P_{a}$. Thus $g(\epsilon(x))\subseteq P_{a}$, and $g(\epsilon(x))=f(x)$. As $f$ is surjective, so is $g$. For any $P_a\in X_A$, let $\Gamma=\{b\in A\mid b \text{ is join-irreducible and } a\not\leq b\}$ (note that $\Gamma$ is finite as $\A$ is finite), then by definition $g^{-1}(P_a)=\beta({\uparrow} f^{-1}(P_a))\setminus \bigcup_{b\in\Gamma}\beta({\uparrow} f^{-1}(P_b))$ which is clopen in $(\mathcal{X}^*)_*$. Thus $g$ is continuous. Therefore, $g$ is a surjective Priestley morphism.

Suppose $\mathfrak{p}R^*\mathfrak{q}$ in $(\mathcal{X}^*)_*$, then for any $\Box a\in g(\mathfrak{p})$, there exists a join-irreducible element $a_i\leq \Box a$ such that ${\uparrow} f^{-1}(P_{a_i})\in\mathfrak{p}$. For any $x\in {\uparrow} f^{-1}(P_{a_i})$, there exists $y\leq x$ such that $f(y)=P_{a_i}$. As $f$ is order-preserving, $f(y)\subseteq f(x)$, $a_i\in f(x)$, and thus $\Box a\in f(x)$. For any $xRz$ in $\mathcal{X}$, as $f$ is stable, $f(x)R_Af(z)$. Thus $a\in f(z)$, and $f(z)=P_{b_i}$ for some join-irreducible element $b_i\leq a$. As $x\in{\uparrow} f^{-1}(P_{a_i})$ is arbitrary, \textcolor{black}{this proves that} ${\uparrow} f^{-1}(P_{a_i})\subseteq \Box_R ({\uparrow} f^{-1}(P_{b_1})\cup...\cup {\uparrow} f^{-1}(P_{b_n}))$ where $b_1,...,b_n$ enumerate all the join-irreducible elements less than or equal to $a$.
As ${\uparrow} f^{-1}(P_{a_i})\in\mathfrak{p}$, it follows that $\Box_R ({\uparrow} f^{-1}(P_{b_1})\cup...\cup {\uparrow} f^{-1}(P_{b_n}))\in\mathfrak{p}$. As $\mathfrak{p}R^*\mathfrak{q}$, \textcolor{black}{we know that} ${\uparrow} f^{-1}(P_{b_1})\cup...\cup {\uparrow} f^{-1}(P_{b_n})\in \mathfrak{q}$, and ${\uparrow} f^{-1}(P_{b_i})\in\mathfrak{q}$ for some $1\leq i\leq n$. Thus $a\in g(\mathfrak{q})$. As $\Box a\in g(\mathfrak{p})$ is arbitrary, this proves that $g(\mathfrak{p})R_A g(\mathfrak{q})$ in $X_A$, $g$ is thus stable.

Suppose $g[R^*[\mathfrak{p}]]\subseteq \beta(a)$ where $a\in D^\Box$, then for any $\mathfrak{p}R^*\mathfrak{q}$, \textcolor{black}{we have that} $g(\mathfrak{q})\in \beta(a)$, namely $a\in g(\mathfrak{q})$. By definition, there exists a join-irreducible element $a_i\leq a$ such that ${\uparrow} f^{-1}(P_{a_i})\in \mathfrak{q}$. As $\mathfrak{q}$ is a prime filter of upsets of $X$, \textcolor{black}{we also get that} ${\uparrow} f^{-1}(P_{a_1})\cup...\cup {\uparrow} f^{-1}(P_{a_n})\in\mathfrak{q}$ where $a_1,...,a_n$ enumerate all the join-irreducible elements less than or equal to $a$. As $\mathfrak{q}$ is arbitrary, $\Box_R({\uparrow} f^{-1}(P_{a_1})\cup...\cup {\uparrow} f^{-1}(P_{a_n}))\in\mathfrak{p}$ (note that $\Box_R({\uparrow} f^{-1}(P_{a_1})\cup...\cup {\uparrow} f^{-1}(P_{a_n}))$ is not empty as $\mathfrak{p}$ is a prime filter). 
For any $x\in \Box_R({\uparrow} f^{-1}(P_{a_1})\cup...\cup {\uparrow} f^{-1}(P_{a_n}))$, by definition, $R[x]\subseteq{\uparrow} f^{-1}(P_{a_1})\cup...\cup {\uparrow} f^{-1}(P_{a_n})$, and thus $f[R[x]]\subseteq \beta(a)$. As $f$ satisfies CDC$_\Box$ for $\beta(a)$, \textcolor{black}{we have that}  $R_A[f[x]]\subseteq \beta(a)$, namely $\Box a\in f(x)$, and thus $f(x)=P_{b_i}$ for some join-irreducible element $b_i\leq \Box a$. As $x\leq x$, $x\in {\uparrow} f^{-1}(P_{b_1})\cup...\cup {\uparrow} f^{-1}(P_{b_m})$ where $b_1,...,b_m$ enumerate all join-irreducible elements less than or equal to $\Box a$. As $x\in \Box_R({\uparrow} f^{-1}(P_{a_1})\cup...\cup {\uparrow} f^{-1}(P_{a_n}))$ is arbitrary, \textcolor{black}{this proves that} $\Box_R({\uparrow} f^{-1}(P_{a_1})\cup...\cup {\uparrow} f^{-1}(P_{a_n}))\subseteq {\uparrow} f^{-1}(P_{b_1})\cup...\cup {\uparrow} f^{-1}(P_{b_m})$. Since $\Box_R({\uparrow} f^{-1}(P_{a_1})\cup...\cup {\uparrow} f^{-1}(P_{a_n}))\in \mathfrak{p}$ and $\mathfrak{p}$ is a prime filter, ${\uparrow} f^{-1}(P_{b_1})\cup...\cup {\uparrow} f^{-1}(P_{b_m})\in \mathfrak{p}$, and thus ${\uparrow} f^{-1}(P_{b_j})\in\mathfrak{p}$ for some $1\leq j\leq m$. Therefore, $\Box a\in g(\mathfrak{p})$, and thus $R_A[g(\mathfrak{p})]\subseteq \beta(a)$. This proves that $g$ satisfies CDC$_\Box$ for $\beta(a)$ where $a\in D^\Box$.

Suppose $g({\uparrow} \mathfrak{p})\cap \beta(a)\setminus\beta(b)=\emptyset$ where $(a,b)\in D^\rightarrow$. For any $\mathfrak{p}\subseteq\mathfrak{q}$ in $(\mathcal{X}^*)_*$, we have that $g(\mathfrak{q})\not\in \beta(a)\setminus\beta(b)$. Namely, if $a\in g(\mathfrak{q})$, then $b\in g(\mathfrak{q})$. Thus ${\uparrow} f^{-1}(P_{a_1})\cup...\cup {\uparrow} f^{-1}(P_{a_n})\in \mathfrak{q}$ implies that ${\uparrow} f^{-1}(P_{b_1})\cup...\cup {\uparrow} f^{-1}(P_{b_m})\in \mathfrak{q}$ where $a_1,...,a_n$ enumerate all join-irreducible elements less than or equal to $a$ while $b_1,...,b_m$ enumerate all join-irreducible elements less than or equal to $b$. As $\mathfrak{p}\subseteq\mathfrak{q}$ is arbitrary, \textcolor{black}{this means that} $({\uparrow} f^{-1}(P_{a_1})\cup...\cup {\uparrow} f^{-1}(P_{a_n}))\rightarrow ({\uparrow} f^{-1}(P_{b_1})\cup...\cup {\uparrow} f^{-1}(P_{b_m}))\in \mathfrak{p}$ (note that $({\uparrow} f^{-1}(P_{a_1})\cup...\cup {\uparrow} f^{-1}(P_{a_n}))\rightarrow ({\uparrow} f^{-1}(P_{b_1})\cup...\cup {\uparrow} f^{-1}(P_{b_m}))$ is not empty). Suppose $x\in ({\uparrow} f^{-1}(P_{a_1})\cup...\cup {\uparrow} f^{-1}(P_{a_n}))\rightarrow ({\uparrow} f^{-1}(P_{b_1})\cup...\cup {\uparrow} f^{-1}(P_{b_m}))$, for any $x\leq y$, if $y\in {\uparrow} f^{-1}(P_{a_1})\cup...\cup {\uparrow} f^{-1}(P_{a_n})$, then $y\in {\uparrow} f^{-1}(P_{b_1})\cup...\cup {\uparrow} f^{-1}(P_{b_m})$. Thus $f[{\uparrow} x]\cap \beta(a)\setminus \beta(b)=\emptyset$. As $f$ satisfies CDC$_\rightarrow$ for $\beta(a)\setminus \beta(b)$, \textcolor{black}{we have that} ${\uparrow} f(x)\cap \beta(a)\setminus\beta(b)=\emptyset$. Thus $a\rightarrow b\in f(x)$, and $f(x)=P_{c_j}$ for some join-irreducible element $c_j\leq a\rightarrow b$. As $x\leq x$, $x\in {\uparrow} f^{-1}(P_{c_j})\subseteq {\uparrow} f^{-1}(P_{c_1})\cup...\cup {\uparrow} f^{-1}(P_{c_k})$ where $c_1,...,c_k$ enumerate all join-irreducible elements less than or equal to $a\rightarrow b$. As $x\in ({\uparrow} f^{-1}(P_{a_1})\cup...\cup {\uparrow} f^{-1}(P_{a_n}))\rightarrow ({\uparrow} f^{-1}(P_{b_1})\cup...\cup {\uparrow} f^{-1}(P_{b_m}))$ is arbitrary, \textcolor{black}{this proves that} $({\uparrow} f^{-1}(P_{a_1})\cup...\cup {\uparrow} f^{-1}(P_{a_n}))\rightarrow ({\uparrow} f^{-1}(P_{b_1})\cup...\cup {\uparrow} f^{-1}(P_{b_m}))\subseteq {\uparrow} f^{-1}(P_{c_1})\cup...\cup {\uparrow} f^{-1}(P_{c_k})$. Then ${\uparrow} f^{-1}(P_{c_1})\cup...\cup {\uparrow} f^{-1}(P_{c_k})\in \mathfrak{p}$, and thus $a\rightarrow b\in g(\mathfrak{p})$. Therefore, ${\uparrow} g(\mathfrak{p})\cap \beta(a)\setminus\beta(b)=\emptyset$.  Thus, if ${\uparrow} g(\mathfrak{p})\cap \beta(a)\setminus \beta(b)\not=\emptyset$, then $g[{\uparrow} \mathfrak{p}]\cap \beta(a)\setminus\beta(b)\not=\emptyset$ where $(a,b)\in D^\rightarrow$. This proves that $g$ satisfies CDC$_\rightarrow$ for $\beta(a)\setminus\beta(b)$ where  $(a,b)\in D^\rightarrow$.

Therefore, $g$ is a surjective stable Priestley morphism satisfying CDC$_\Box$ for any $\beta(a)$ where $a\in D^\Box$ and satisfies CDC$_\rightarrow$ for any $\beta(a)\setminus \beta(b)$ where $(a,b)\in D^\rightarrow$. By Proposition \ref{3.3.12}, $(\mathcal{X}^*)_*\not\vDash \rho(\A,D^\rightarrow, D^\Box)$, and thus $\mathcal{X}=(X,\leq,R)\not\vDash\rho(\A,D^\rightarrow, D^\Box)$.

\bibliographystyle{elsarticle-num} 
\bibliography{cas-refs}

\end{document}